\documentclass[11pt,reqno]{amsart}
\usepackage{amssymb,eucal}

\usepackage[all]{xy}
\SilentMatrices
\SelectTips{eu}{}
\newdir{ >}{{}*!/-10pt/\dir{>}}

\usepackage[dvipsnames,usenames]{color}

\newtheorem{theorem}{Theorem}[section]
\newtheorem{proposition}[theorem]{Proposition}
\newtheorem{corollary}[theorem]{Corollary}
\newtheorem{lemma}[theorem]{Lemma}

\theoremstyle{definition}
\newtheorem{definition}[theorem]{Definition}

\theoremstyle{remark}

\newtheoremstyle{citing}
  {}
  {}
  {\itshape}
  {}
  {\bfseries}
  {.}
  {.5em}
  {\thmnote{#3}}

\theoremstyle{citing}
\newtheorem*{varthm}{}

\numberwithin{equation}{section}

\newcommand\shf\mathcal
\newcommand\projective\mathbb

\newcommand\PP{{\projective P}}
\newcommand\OO{{\shf O}}
\newcommand\ZZ{{\projective Z}}

\newcommand\RR{{\projective R}}

\newcommand\CC{{\projective C}}
\newcommand\NN{{\projective N}}
\DeclareMathOperator{\HP}{HP}
\DeclareMathOperator{\HGr}{H Gr}
\DeclareMathOperator{\HFlag}{H Flag}
\newcommand\GG{{\projective G}}
\newcommand{\Aff}{{\projective A}}

\newcommand{\GW}{GW}

\DeclareMathOperator{\Gr}{Gr}
\DeclareMathOperator{\GrSp}{GrSp}
\DeclareMathOperator{\Flag}{Flag}
\DeclareMathOperator{\Sp}{Sp}
\DeclareMathOperator{\GL}{GL}
\DeclareMathOperator{\SL}{SL}
\DeclareMathOperator{\UU}{U}
\DeclareMathOperator{\Pic}{Pic}
\DeclareMathOperator{\Spec}{Spec}

\DeclareMathOperator{\Hom}{Hom}

\DeclareMathOperator{\rk}{rk}

\DeclareMathOperator{\Th}{th}

\DeclareMathOperator{\pr}{pr}
\DeclareMathOperator{\Hyp}{h}

\DeclareMathOperator{\coeff}{coeff}

\newcommand{\parens}[1]{\textup{(}#1\textup{)}}

\DeclareMathOperator{\bl}{bl}

\newcommand\dual{^{\vee}}

\newcommand{\op}{{\rm op}}
\newcommand{\Sm}{\mathbf{Sm}}
\newcommand{\Set}{\mathbf{Set}}

\newcommand{\arrowiso}{\stackrel{\smash{\lower3pt\hbox{$\scriptstyle\cong$}}}\lra}
\newcommand{\mono}{\rightarrowtail}

\newcommand{\rest}[1]{{\mid}_{#1}}

\newcommand{\into}{\hookrightarrow}
\newcommand{\onto}{\twoheadrightarrow}
\newcommand{\lra}{\longrightarrow}

\newcounter{subtheorem}[theorem]
\renewcommand{\thesubtheorem}{\alph{subtheorem}}
\newcommand{\subthm}{\refstepcounter{subtheorem}{\rm(\thesubtheorem)\ }}

\DeclareMathOperator{\Spin}{Spin}

\newcommand{\bigperp}{\mathop{\text{\LARGE$\perp$}}\limits}
\newcommand{\tbigperp}{\mathop{\text{\large$\perp$}}\nolimits}

\newcommand{\sm}[1]{\left(\begin{smallmatrix}#1\end{smallmatrix}\right)}

\begin{document}

\title
[Quaternionic Grassmannians and Borel classes]
{Quaternionic Grassmannians and Borel classes in algebraic geometry}
\author{Ivan Panin}
\address{Steklov Institute of Mathematics\\Saint Petersburg}
\author{Charles Walter}
 \address{Laboratoire J.-A.\ Dieudonn\'e (UMR 6621 du CNRS)\\
 D\'epartement de math\'ematiques\\
 Universit\'e de Nice -- Sophia Antipolis\\ 06108 Nice Cedex 02\\
 France}

\thanks{The first author gratefully acknowledge excellent working conditions and support provided by 
Laboratoire J.-A. Dieudonn\'{e}, UMR 6621 du CNRS, Universit\'{e} de Nice  Sophia Antipolis, 
and by the RCN Frontier Research Group Project no. 250399 “Motivic Hopf equations" at University of Oslo.
}

\begin{abstract}
The quaternionic Grassmannian $\HGr(r,n)$ is the affine open subscheme
of the ordinary Grassmannian
parametrizing
those $2r$-dimensional subspaces of a $2n$-dimensional symplectic vector space on which the
symplectic form is nondegenerate.  In particular there is $\HP^{n} = \HGr(1,n+1)$.
For a symplectically oriented cohomology theory $A$, including oriented theories
but also hermitian $K$-theory, Witt groups and algebraic symplectic cobordism,
we have $A(\HP^{n}) = A(\operatorname{pt})[p]/(p^{n+1})$.
We define Borel classes for symplectic bundles.  They satisfy the
splitting principle and the Cartan sum formula, and we
use them to calculate the cohomology
of quaternionic Grassmannians.  In a symplectically oriented theory the
Thom classes of rank $2$ symplectic bundles determine Thom and Borel classes for all
symplectic bundles, and the symplectic Thom classes can be recovered from the Borel classes.

The cell structure of the $\HGr(r,n)$ exists in the cohomology, but it is difficult to see more than part of it
geometrically.
The exception is $\HP^{n}$ where the cell of codimension $2i$ is
a quasi-affine quotient of $\Aff^{4n-2i+1}$ by a
nonlinear action of $\GG_{a}$.
\end{abstract}

\maketitle


\section{Introduction}

The quaternionic projective spaces $\mathbb{HP}^{n}$
and Grassmannians
$\mathbb{HG}\mathrm{r}(r,n)$
are important spaces in topology.
They have cell structures like real and complex projective spaces and Grassmannians, but
the dimensions of the cells are multiples of $4$.
In symplectically oriented cohomology
theories $E^{*}$, including oriented theories but also
$KO^{*}$,
$MSp^{*}$ and $MSU^{*}$,
there is a quaternionic projective bundle theorem
$E^{*}(\mathbb{HP}^{n}) = E^{*}(\operatorname{pt})[b]/(b^{n+1})$
which leads to a theory of Borel classes of quaternionic bundles satisfying the Cartan sum formula
and the Splitting Principle.
These classes are used to
prove a number of well-known theorems, including Conner and Floyd's description of $KO^{*}$ as
a quotient of $MSp^{*}$ \cite[Theorem 10.2]{Conner:1966uk}.
Infinite-dimensional quaternionic Grassmannians provide models for the classifying spaces
$\operatorname{BSp}_{2r}$ and $\operatorname{BSp}$ and for symplectic
$K$-theory.  For symplectically oriented cohomology
theories we have $E^{*}(\operatorname{BSp}_{2r}) = E^{*}(\operatorname{pt})[[b_{1},\dots,b_{r}]]$
and $E^{*}(\operatorname{BSp}) = E^{*}(\operatorname{pt})[[b_{1},b_{2},\dots]]$
where the $b_{i}$ are the Borel classes of the universal bundle.

In this paper we lay the foundations for a similar theory in motivic algebraic geometry.
Since quaternionic Grassmannians are quotients of compact Lie groups
$\operatorname{\mathbb{HG}r}(r,n) =
\UU_{n}(\mathbb{H})/(\UU_{r}(\mathbb{H}) \times \UU_{n-r}(\mathbb{H}))$,
we take as our models the corresponding quotients of algebraic groups
$\HGr(r,n) = \Sp_{2n}/(\Sp_{2r} \times \Sp_{2n-2r})$.  Then
$\HGr(r,n)(\CC)$ has the same homotopy type as
$\operatorname{\mathbb{HG}r}(r,n)$ but
twice the dimension and a significantly more complicated geometry.

An alternative description of $\HGr(r,n)$ is that
if we equip the trivial vector bundle on the base $\OO^{2n}$
with the standard symplectic form $\psi_{2n}$, then
$\HGr(r,n)$ is the open subscheme of $\Gr(2r,2n)$ parametrizing the subspaces $U \subset \OO^{2n}$
of dimension $2r$
on which $\psi_{2n}$ is nonsingular.  The $\HGr(r,n)$ are smooth and affine of dimension
$4r(n-r)$ over the base scheme with the same global units and the same Picard group
as the base.

Our first results concern $\HP^{n} = \HGr(1,n+1)$.
In Theorems \ref{th.basic}, \ref{th.subgrass} and \ref{A.X2i} we prove the following facts.

\begin{theorem}
The scheme $\HP^{n}$ is smooth of dimension $4n$ over the base scheme.
It has a decomposition into locally closed strata
\begin{equation*}
\tag{\ref{E:strat}}
\HP^n
= \bigsqcup_{i=0}^{n} X_{2i} = X_{0} \sqcup X_{2} \sqcup
\cdots \sqcup X_{2n},
\end{equation*}
such that each $X_{2i}$ is of codimension $2i$, smooth and quasi-affine,
but $X_{0}, \dots, X_{2n-2}$ are not affine.
The closure $\overline X _{2i} = X_{2i} \sqcup X_{2i+2} \sqcup \cdots \sqcup X_{2n}$
is a vector bundle of rank $2i$ over $\HP^{n-i}$.
Each $X_{2i}$ is the quotient of a free action of $\GG_{a}$ on $\Aff^{4n-2i+1}$.
\end{theorem}

One can study ordinary Grassmannians inductively by considering the
closed embedding
$\Gr(r,n-1) \into \Gr(r,n)$.  The complement of the image is a vector bundle over $\Gr(r-1,n-1)$.
The normal bundle of the embedding
is isomorphic to the dual $\shf U_{r,n-1}^{\vee}$ of the tautological
subbundle on $\Gr(r,n-1)$, and it
embeds as an open subvariety of $\Gr(r,n)$.  This gives a long exact
sequence
\[
\cdots \to A_{\Gr(r,n-1)}(\shf U_{r,n-1}^{\vee}) \to A(\Gr(r,n)) \to A(\Gr(r-1,n-1)) \to \cdots
\]
in any cohomology theory.

There is an analogous long exact sequence for quaternionic Grassmannians, but
one has to wade through a more complicated geometry to reach it.
We take a symplectic bundle $(E,\phi)$ of rank $2n$ over a scheme $S$.   We then
let $(F,\psi)$
be the symplectic bundle of rank $2n$ with
\begin{align*}
\tag{\ref{E:basic.setup}}
F & = \OO_{S} \oplus E \oplus \OO_{S}, &
\psi & =
\begin{pmatrix}
0 & 0 & 1 \\ 0 & \phi & 0 \\ -1 & 0 & 0
\end{pmatrix}.
\end{align*}
We have the quaternionic Grassmannian bundles
$\HGr(E) = \HGr_{S}(r;E,\phi)$ and $\HGr(F) = \HGr(r;F,\psi)$.
Earlier, the complement of
the open stratum $X_{0} \subset \HP^{n}$ was a vector bundle $\overline X_{2}$ over $\HP^{n-1}$
rather than $\HP^{n-1}$ itself.  The same thing happens now.  Let
\begin{align*}
N^{+} & = \HGr(F) \cap \Gr_{S}(2r;\OO_{S}\oplus E), &
N^{-} & = \HGr(F) \cap \Gr_{S}(2r;E\oplus \OO_{S}).
\end{align*}
For any cohomology theory $A$ we have a long exact sequence
\begin{equation*}
\tag{\ref{E:long.exact}}
\cdots \to A_{N^{+}}(\HGr(F)) \to A(\HGr(F)) \to A(\HGr(F) \smallsetminus N^{+}) \to \cdots.
\end{equation*}

\begin{varthm}
[Theorem]

\parens{a}
The loci $N^{+}$ and $N^{-}$ are vector bundles over
$\HGr(E)$ isomorphic to the tautological rank $2r$ symplectic subbundle $\shf U_{r,E}$.

\parens{b}
The vector bundle
$N^{+}\oplus N^{-}$ is the normal bundle of
$\HGr(E) \subset \HGr(F)$, and it naturally
embeds as an open subscheme of
the Grassmannian bundle $\Gr_{S}(2r;F)$.

\parens{c}
$N^{+}$ and $N^{-}$ are the loci in $\HGr(F)$ where certain sections $s_{+}$ and $s_{-}$ of the
tautological bundle $\shf U_{r,F}$ intersect the zero section transversally.
\end{varthm}

This is part of Theorem \ref{T:normal.geom}.
Since $\HGr(F) \subset \Gr_{S}(2r;F)$ is also an open subscheme, we deduce natural isomorphisms
\[
A_{N^{+}}(\HGr(F)) \cong A_{N^{+}}(N^{+} \oplus N^{-}) \cong A_{\HGr(r;E,\phi)}(\shf U_{r,E}).
\]

The study of the open subscheme $Y = \HGr(F) \smallsetminus N^{+}$ in \S \ref{S:open.stratum}
is crucial to this paper.  The main result is:

\begin{varthm}
[Theorem]

We have morphisms
\[
Y \xleftarrow{g_{1}} Y_{1} \xleftarrow{g_{2}} Y_{2} \xrightarrow{q} \HGr_{S}(r-1,E,\phi)
\]
over $S$ with $g_{1}$ an $\Aff^{2r-1}$-bundle, $g_{2}$ an $\Aff^{2r-2}$-bundle and
$q$ an $\Aff^{4n-3}$-bundle.  Moreover,
there is an explicit symplectic automorphism
of the pullback of $(F,\psi)$ to $Y_{2}$ which induces an isometry
\[
\xymatrix @M=5pt @C=30pt {
g_{2}^{*}g_{1}^{*}(\shf U_{r},\phi_{r})
\ar[r]^-{\cong} &
\left( \OO \oplus q^{*}\shf U_{r-1} \oplus \OO, \sm{0&0&1\\0&\phi_{r-1}&0\\-1&0&0} \right),
}
\]
where $(\shf U_{r},\phi_{r})$ and $(\shf U_{r-1},\phi_{r-1})$ are the tautological symplectic
subbundles on $\HGr(r,F,\psi)$ and $\HGr(r-1,E,\phi)$, respectively.
\end{varthm}

This is a simplified version of Theorem \ref{th.Ga}.  From it we deduce that the
natural closed embedding $\HGr_{S}(r-1,E,\phi) \into Y$ induces
isomorphisms $A(Y) \cong A(\HGr_{S}(r-1,E,\phi))$ in any cohomology theory.
For $r=1$ this is an isomorphism $A(Y) \cong A(S)$.
We thus get our long exact sequence of cohomology groups
\begin{equation*}
\tag{\ref{E:the.long.exact}}
\cdots \to A_{\HGr(r,E,\phi)}(\shf U_{r,E}) \to A(\HGr(r,F,\psi))
\to A(\HGr(r-1,E,\phi)) \to \cdots.
\end{equation*}

The geometry of Theorems \ref{T:normal.geom} and \ref{th.Ga}
and the decomposition of $\HGr_{S}(r,F,\psi)$ into the vector bundle
$N^{+}$
and the complementary open locus $Y$
is recurrent throughout the paper.   This seems to be a fundamental geometry of quaternionic
Grasssmannian bundles.

After the initial description of the geometry, we
%
begin to introduce symplectically oriented cohomology theories.
We follow the point of view used for oriented cohomology theories in \cite{Panin:2003rz}.  In the end
(Definition \ref{D:symp.orient})
a symplectic orientation on a ring
cohomology theory is a family of
\emph{Thom isomorphisms} $\Th^{E,\phi}_{X} \colon A(X) \cong A_{X}(E)$ for every
scheme $X$ and every symplectic bundle $(E,\phi)$ over $X$.  There are also isomorphisms
$\Th^{E,\phi}_{Z} \colon A_{Z}(X) \cong A_{Z}(E)$ for closed subsets $Z \subset X$.
These isomorphisms satisfy several axioms of functoriality, of compatibility with the ring structure,
and of compatibility with orthogonal direct sums $(E_{1},\phi_{1}) \perp (E_{2},\phi_{2})$.

There are five ways of presenting a symplectic orientation.
The Thom class $\Th(E,\phi) \in A_{X}(E)$
is the image of $1_{X} \in A(X)$.  One can present a symplectic orientation by giving either the Thom
isomorphisms, the Thom classes or the Borel classes of all symplectic bundles or
by giving
the Thom classes or the Borel classes of rank $2$ symplectic bundles only.  In each case
the classes are supposed to obey a certain list of axioms.  In the end (Theorem \ref{T:bijections})
the five ways of presenting a symplectic orientation on a ring cohomology theory are equivalent.

We start in \S \ref{S:symp.thom} by presenting the version where one gives the Thom classes
for rank $2$ symplectic bundles only.  This is a \emph{symplectic Thom structure}.
With the Thom classes one can define Thom isomorphisms $A(X) \cong A_{X}(E)$ for rank $2$
symplectic bundles and direct image maps $i_{A,\flat} \colon A(X) \to A_{X}(Y)$
and $i_{A,\natural} \colon A(X) \to A(Y)$
for regular embeddings $i \colon X \into Y$ of codimension $2$ whose normal bundle is
equipped with a symplectic form $(N_{X/Y},\phi)$.
The \emph{Borel class} of a rank $2$ symplectic bundle $(E,\phi)$ on $X$
is $b(E,\phi) = -z^{A}e^{A}\Th(E,\phi)$ where $e^{A} \colon A_{X}(E) \to A(E)$ is extension
of supports and $z^{A} \colon A(E) \cong A(X)$ is the restriction to the zero section.
The main formula is that if a section of $E$ intersects the zero section transversally in a
subscheme $Z$, then for the inclusion $i \colon Z \into X$ and
for all $b \in A(X)$ we have
\begin{equation*}
\tag{\ref{E:i_*i^*}}
i_{A,\natural}i^A b  = - b \cup b(E,\phi).
\end{equation*}

With the geometry of Theorems \ref{T:normal.geom} and \ref{th.Ga}, the long exact sequence
of cohomology \eqref{E:long.exact} and formula \eqref{E:i_*i^*} we prove one of the
main results of the paper.

\begin{varthm}[Theorem \protect\ref{th.HPn.twist}]
\textup{(Quaternionic projective bundle theorem).}
Let $A$ be a ring cohomology theory with a symplectic Thom structure.
Let $(E, \phi)$ be a rank $2n$ symplectic bundle over a scheme
$S$, let $(\shf U,\phi \rest {\shf U})$ be the tautological
rank $2$ symplectic subbundle over the quaternionic projective bundle
$\HP_S(E, \phi)$, and let $\zeta = b(\shf U,\phi \rest{\shf U})$ be its Borel class.
Write $\pi \colon \HP_{S}(E,\phi) \to S$ for the projection.
Then for any closed subset $Z \subset X$ we have an
isomorphism of two-sided $A(S)$-modules
$(1, \zeta, \dots, \zeta^{n-1}) \colon A_{Z}(S)^{\oplus n} \arrowiso
A_{\pi^{-1}(Z)}(\HP_S(E,\phi))$,
and we have
unique classes $b_i(E,\phi) \in A(S)$ for $1 \leq i \leq n$
such that there is a relation
\[
\zeta^{n} - b_1(E,\phi) \cup \zeta^{n-1} + b_2(
E, \phi) \cup \zeta^{n-2} - \cdots +
(-1)^{n} b_{n}(E,\phi) = 0.
\]
If $(E,\phi)$ is trivial, then $b_i(E,\phi) = 0$
for $1 \leq i \leq n$.
\end{varthm}

The classes $b_{i}(E,\phi)$ are the \emph{Borel classes} of $(E,\phi)$ with respect to
the symplectic Thom structure on $A$.  One also sets
$p_{0}(E,\phi) = 1$ and $b_{i}(E,\phi) = 0$ for $i > n$.
Borel classes are $\Aff^{1}$-deformation invariant
and nilpotent.

The Borel classes of an orthogonal direct sum of symplectic bundles
$(F, \psi) \cong (E_{1},\phi_{1}) \perp (E_{2},\phi_{2})$ satisfy the
Cartan sum formula
\begin{equation*}
\tag{\ref{E:sum.2nd}}
b_{i}(F,\psi)  = b_{i}(E_{1},\phi_{1}) + \sum_{j=1}^{i-1} b_{i-j}(E_{1},\phi_{1}) b_{j}(E_{2},\phi_{2}) +
b_{i}(E_{2},\phi_{2}).
\end{equation*}
The $b_{1}$ is additive, and the top Borel classes are multiplicative.
The corresponding formula for Chern classes is usually proven with a geometric argument
shortly after a proof of the projective bundle theorem.  Unfortunately we
have not found a geometric proof of the Cartan sum formula for Borel classes.
The geometry of quaternionic projective bundles (as we have defined them)
is just not as nice as the geometry of projective bundles.
The best we have done
geometrically is to show that if $(E,\phi)$ is an orthogonal direct summand of $(F,\psi)$
then the Borel polynomial $P_{E,\phi}(t)$ divides the Borel polynomial
$P_{F,\psi}(t)$
(Lemma \ref{L:divides.2}).

So we give a roundabout cohomological proof of the Cartan sum formula.  The first step is to consider
the scheme $\HFlag(1^{r};n) = \Sp_{2n}/(\Sp_{2}^{\times r} \times \Sp_{2n-2r})$.
It classifies decompositions
\begin{equation*}
\tag{\ref{E:univ.bdl}}
(\OO^{2n}, \psi_{2n}) \cong
(\shf U^{(1)}_{n},\phi^{(1)}_{n}) \perp \cdots \perp (\shf U^{(r)}_{n}, \phi^{(r)}_{n}) \perp (\shf V_{r,n}, \psi_{r,n})
\end{equation*}
of the trivial symplectic bundle of rank $2n$
into the orthogonal direct sum of $r$ symplectic subbundles of rank $2$ plus a symplectic
subbundle of rank $2n-2r$.  It is an iterated quaternionic projective bundle over both the base scheme
and $\HGr(r,n)$.
The weak substitute for the Cartan sum formula (Lemma \ref{L:divides.2}) is good enough
for us to show that we have
\begin{equation*}
\tag{\ref{E:HFlag.lim}}
\varprojlim_{n \to \infty} A(\HFlag(1^{r};n)) \cong A(k) [[y_{1},\dots, y_{r}]]
\end{equation*}
with the indeterminate $y_{i}$ corresponding to the element
given by the system $( b(\shf U^{(i)}_{n},\phi^{(i)}_{n}))_{n \geq r}$.
Using the universality of the families of schemes $\HFlag(1^{r};n)$ and $\HGr(r,n)$ and the
fact that the $y_{i}-y_{j}$ are not zero divisors in $A(k)[[y_{1},\dots,y_{r}]]$, we are able to
deduce the Symplectic Splitting Principle (Theorem \ref{T:splitting}) from
Lemma \ref{L:divides.2}.  The Cartan sum formula then follows (Theorem \ref{T:Cartan}).

We then calculate the cohomology of quaternionic Grassmannians for any ring cohomology theory
with a symplectic Thom structure.  We get
\begin{gather*}
\tag{\ref{E:hn.2}}
A(\HFlag(1^{r},n))  = A(k)[y_{1},\dots,y_{r}]/(h_{n-r+1},\dots,h_{n}), \\
\tag{\ref{E:HGr.cohom.1}}
A(\HGr(r,n))  = A(k)[b_{1},\dots,b_{r}]/(h_{n-r+1},\dots,h_{n}),
\end{gather*}
where the $y_{i}$ are the Borel classes of the $r$ tautological rank $2$ symplectic
subbundles on $\HFlag(1^{r},n)$, the $b_{i}$ are the Borel classes of the tautological
rank $2r$ symplectic subbundle on $\HGr(r,n)$ and are the elementary symmetric polynomials
in the $y_{i}$, and the $h_{i}$ are the complete symmetric polynomials.  It follows that for the
standard embeddings the restrictions $A(\HGr(r,n+1)) \to A(\HGr(r,n))$ and
$A(\HGr(r+1,n+1)) \to A(\HGr(r,n))$ are surjective, and we have isomorphisms
\begin{gather}
\tag{\ref{E:HGr(r,infty)}}
\varprojlim\limits_{n \to \infty} A(\HGr(r,n)) = A(k)[[b_{1},b_{2},\dots,b_{r}]], \\
\tag{\ref{E:HGr(infty,2.infty)}}
\varprojlim\limits_{n \to \infty} A(\HGr(n,2n)) = A(k)[[b_{1},b_{2},b_{3},\dots]].
\end{gather}

In \S\S \ref{S:Pont.thom}--\ref{S:orientations} we prove the equivalence of the five
ways of presenting a symplectic orientation on a ring cohomology theory.
Above all, this involves using the exact sequence \eqref{E:the.long.exact} and the
calculations of the cohomology of Grassmannian bundles to recover the Thom classes
from the Borel classes.

The isomorphism $A(Y) \cong A(S)$ when $Y = X_{0} \times S \subset \HP^{n} \times S$
is the open stratum
is absolutely critical to the entire paper.  Without it there is no proof of the quaternionic
projective bundle theorem and no definition of Borel classes.
Therefore in \S \ref{S:original} we give a second proof of the isomorphism using
a completely different geometric argument.  However, the geometry of Theorem
\ref{th.Ga} seems to be more generally useful than that of Theorem \ref{2nd}.

We stop here.  This is already a long paper because the theory of Thom and Borel
classes of symplectic bundles and the calculation of the cohomology of quaternionic
Grassmannian bundles are intertwined, and we do not see how to separate one from the other.
We leave for other papers the discussion of symplectic orientations on specific nonoriented
cohomology theories like
hermitian $K$-theory, derived Witt groups and symplectic algebraic cobordism,
as well as discussion of
the $2$ and $4$-valued formal group laws
obeyed by the symplectic Borel classes.

\subsection*{Acknowledgements}
The authors are grateful for the hospitality of Paul Balmer, Max-Albert Knus and
ETH Z\"urich where the project began.
They warmly thank Alexander Nenashev for his interest: \S\ref{S:Pont.thom} is joint work.

\section{Cohomology theories}
\label{S:cohomology}

We review the notions of a cohomology theory and a ring cohomology theory as used in
\cite{Panin:2003rz}.

We fix a base scheme $k$.  We will study ring cohomology theories on some
category $\mathcal V$ of $k$-schemes.  The category could be nonsingular
quasi-projective varieties over $k = \Spec F$ with $F$ a field, it could be
nonsingular quasi-affine varieties over $k = \Spec F$,
it could be quasi-compact semi-separated
schemes with an ample set of line bundles over $k$, it could be
regular noetherian separated schemes of finite Krull dimension over $k$,
or it could be something else.
Which it is is unimportant
 as long as the definitions make sense and the constructions work.
There are only some minor complications concerning deformation to the
normal bundle when we go beyond smooth varieties over a field.
Our theorems all hold if $\mathcal V$ has the following properties (partly borrowed from
Levine-Morel \cite[(1.1.1)]{Levine:2007ys}):

\begin{enumerate}
\item
All schemes in $\mathcal V$ are quasi-compact and quasi-separated and have an ample family of line
bundles.
\item
The schemes $k$ and $\varnothing$ are in $\mathcal V$.

\item
If $X$ and $Y$ are in $\mathcal V$, then so are $X \sqcup Y$ and $X \times_{k} Y$.

\item
If $X$ is in $\mathcal V$, and $U \subset X$ is a quasi-compact open subscheme, then $U$ is in $\mathcal V$.
\item
If $X$ is in $\mathcal V$, and $Y$ is a quasi-compact open subscheme of a Grassmannian bundle over
$X$ for which the projection map $Y \to X$ is an affine morphism, then $Y$
is in $\mathcal V$.
\end{enumerate}

Some of our discussions of deformation to the normal bundle and of direct images make more
sense if a sixth property is also true.
Recall that a \emph{regular embedding} $Z \to X$ is a closed
embedding such that locally $Z$ is cut out by a regular sequence.

\begin{enumerate}
\addtocounter{enumi}{5}
\item
For a regular embedding $Z \to X$ with $X$ and $Z$ in $\mathcal V$ the deformation to the
normal bundle space
$D(Z,X)$ of \S\ref{S:normal.bundle} is in $\mathcal V$.
\end{enumerate}

We will use our language imprecisely as if all schemes were noetherian.  Thus when we write
``open subscheme'' we really mean a quasi-compact open subscheme,
and ``closed subset'' means the complement of a quasi-compact open subscheme.

The condition that every scheme in $\mathcal V$
have an ample family of line bundles is used in the
proof of the symplectic splitting principle (Theorem \ref{T:splitting}).


\begin{definition}
[\protect{\cite[Definition 2.1]{Panin:2003rz}}]
\label{D:cohom}
A \emph{cohomology theory} on a category $\mathcal V$ of $k$-schemes is a pair
$(A,\partial)$ with $A$ a
functor assigning an abelian group $A_Z(X)$ to
every scheme $X$ and closed subset $Z \subset X$ with
$X$ and $X \smallsetminus Z$ in $\mathcal V$
and assigning a morphism of abelian groups $f^A \colon A_Z(X) \to A_{Z'}(X')$ to every map
$f \colon X' \to X$ such that $Z' \supset f^{-1}(Z)$.
One writes $A(X) = A_{X}(X)$.
In addition one has a morphism of functors $\partial \colon A(X \smallsetminus Z) \to A_{Z}(X)$.
Together they have the following properties.
\begin{enumerate}
\item
\emph{Localization}:  The functorial sequences
\[
A(X \smallsetminus Z) \xrightarrow{\partial} A_Z(X) \xrightarrow{e^{A}}
A(X) \xrightarrow{j^{A}} A(X \smallsetminus Z) \xrightarrow{\partial} A_Z(X),
\]
with $e^{A}$ and $j^{A}$ the appropriate maps of $A$, are exact.

\item
\emph{Etale excision}: $f^A \colon A_Z(X) \to A_{Z'}(X')$
is an isomorphism if $f \colon X' \to X$ is \'etale, $Z' = f^{-1}(Z)$, and
$f \rest {Z'}\colon Z' \to Z$ is an isomorphism,

\item
\emph{Homotopy invariance}: the maps $\pr_1^A \colon A(X) \to A(X \times \Aff^1)$ are
isomorphisms.
\end{enumerate}
\end{definition}

Cohomology theories have Mayer-Vietoris sequences and satisfy $A_{\varnothing}(X) = 0$.
They have
homotopy invariance for $\Aff^{n}$-bundles (torsors for vector bundles).
Deformation to the normal bundle isomorphisms will be discussed in \S\ref{S:normal.bundle}.
See \cite[\S 2.2]{Panin:2003rz} for these and other properties.

Zariski excision suffices for the main results of the paper.
Excision is used mainly for Mayer-Vietoris and for direct images, but we use direct images only
for one
Grassmannian embedded in another, and then there are global Zariski tubular neighborhoods.

\begin{definition}
[\protect{\cite[Definition 2.13]{Panin:2003rz}}]
\label{D:ring}
A \emph{ring cohomology theory} is a cohomology theory in the sense of
Definition \ref{D:cohom}
with cup products
\[
\cup \colon A_{Z}(X) \times A_{W}(X) \to A_{Z \cap W}(X)
\]
which are functorial, bilinear and associative and have two other properties:
\begin{enumerate}
\item
There exists an element $1 \in A(k)$ such that for every scheme
$\pi_{X} \colon X \to k$ in $\mathcal{V}$ and every closed subset $Z \subset X$,
the pullback $1_{X} = \pi_{X}^{A}(1) \in A(X)$
satisfies $1_{X} \cup a = a \cup 1_{X} = a$ for all $a \in A_{Z}(X)$.

\item
For the maps $\partial \colon A(X \smallsetminus Z) \to A_{Z}(X)$ one has
$\partial(a \cup b) = \partial a \cup b$ for all $a \in A(X \smallsetminus Z)$ and all
$b \in A(X)$.
\end{enumerate}
\end{definition}

\section{Basic geometry of $\HP^n$ }
\label{S:basic.geom}

We define $\HP^{n}$ and discuss a stratification resembling the cell decomposition
of the topological $\HP^{n}$.  We also present quaternionic projective bundles,
Grassmannians and flag varieties.

Let $(V,\phi)$ be a trivial symplectic bundle of rank $2n+2$ over
the base scheme $k$.  The symplectic group $\Sp_{2n+2} = \Sp(V,\phi)$ acts on the
Grassmannian $\Gr(2,V)$ with (i) a closed orbit $\GrSp(2,V,\phi)$
parametrizing $2$-dimensional subspaces $U \subset V$ with $\phi\rest
U \equiv 0$, and (ii) a complementary open orbit parametrizing
$2$-dimensional subspaces $U \subset V$ with $\phi \rest U$
nondegenerate
which we will call the
{\em quaternionic projective space}\/ $\HP^n$.
We will use this object as a motivic analogue of the topological $\mathbb{HP}^{n}$.
Determining the stabilizer of a point of $\HP^n$
yields an identification $\HP^n = \Sp_{2n+2} / ( \Sp_2 \times
\Sp_{2n} )$, which compares well with
$\mathbb{HP}^{n} = \UU_{n+1}(\mathbb{H})/\UU_{1}(\mathbb{H}) \times \UU_{n}(\mathbb{H})$.
So the manifold $\HP^{n}(\CC)$ of complex points is the complexification of the quotient
of compact Lie groups $\mathbb{HP}^{n}$ and has the same homotopy type.
It does not have the homotopy type
of a complex projective manifold.

The topological $\mathbb{HP}^{n}$ is the union of cells of dimensions $0,4,8,\dots,4n$.
A related decomposition of the space $\HP^n$ may be defined.  Fix a
flag
\begin{equation}
\label{iso.flag}
0  = E_0  \subset E_1 \subset \cdots \subset E_{n+1} =
E_{n+1}^\perp \subset \cdots \subset E_1^\perp \subset E_0^\perp = V
\end{equation}
of subbundles of $(V, \phi)$ with the $E_i \cong \OO_{k}^{\oplus i}$ totally isotropic and
satisfying $\dim E_i = i$ and $\dim E_i^\perp = 2n+2-i$.  Set
\begin{align}
\label{X2i}
  \overline X_{2i} & = \Gr(2, E_i^\perp) \cap \HP^n, & X_{2i} & =
  \overline X_{2i} \smallsetminus \overline X_{2i+2}.
\end{align}
Choose for convenience a lagrangian supplementary to the lagrangian
$E_{n+1}$, and let $\GL_{n+1} \subset \Sp_{2n+2}$ be the subgroup fixing
the two lagrangians.

\begin{theorem}
\label{th.basic}

The scheme $\HP^n$ is the disjoint union of the locally closed
strata
\begin{equation}
\label{E:strat}
\HP^n
= \bigsqcup_{i=0}^{n} X_{2i} = X_{0} \sqcup X_{2} \sqcup
\cdots \sqcup X_{2n},
\end{equation}
which have the following properties{\rm :}

\parens{a}
The scheme $X_{2i}$ and its closure
$\overline X _{2i} = X_{2i} \sqcup X_{2i+2} \sqcup \cdots \sqcup X_{2n}$
are smooth of relative dimension $4n-2i$ over the base scheme $k$.  The
$\overline X_{2i}$ and $X_{2n} = \overline X_{2n}$ are affine over the base,
but $X_{0}, \dots, X_{2n-2}$ are not.  We have $
\Pic(\overline X_{2i}) = \Pic(k)$ and $\OO(\overline X_{2i})^\times =
\OO(k)^\times$.

\parens{b}
Each $\overline X_{2i}$ is the transversal intersection
of $i$ translates of $\overline X_{2}$ under the action of the
subgroup $\GL_{n+1} \subset \Sp_{2n+2}$.

\parens{c}
The intersection of $n+1$ general translates of $\overline X_{2}$
under the action of $\GL_{n+1}$ is empty.

\end{theorem}

In particular, since $\overline X_0 = \HP^n$ we have $\Pic(\HP^n)
= \Pic(k)$ and $\OO(\HP^n)^\times = \OO(k)^\times$.

\begin{proof}

(a)
Let $\shf U_{\Gr}$ be the tautological rank $2$ subbundle on $\Gr
= \Gr(2,V)$.  Then $\GrSp(2,V,\phi)$ is the zero locus of a section of
$\Lambda^2 \shf U_{\Gr}\dual \cong \shf O_{\Gr}(1)$ induced by $\phi$.
So $\HP^n$ is the complement in the smooth projective scheme $\Gr$
of an irreducible ample divisor whose class generates $\Pic(\Gr)/\Pic(k) \cong
\ZZ$.  It follows that $\HP^n$ is smooth and affine of the same
dimension $4n$ as the Grassmannian, and that it satisfies
$\Pic(\HP^n)/\Pic(k) = 0$ and $\OO(\HP^n)^\times = \OO(k)^\times$.

The same argument applied to $\Gr(2,E_i^\perp)$ shows that the
nonempty $\overline X_{2i}$ are smooth and affine of the same
dimension $4n-2i$ as $\Gr(2,E_i^\perp)$ and satisfy $\Pic(\overline
X_{2i})/\Pic(k) = 0$ and $\OO(\overline X_{2i})^\times = \OO(k)^\times$.  Moreover,
$\overline X_{2i}$ is empty if and only if $E_i^\perp$ is totally
isotropic, and this is true only for $i = n+1$.  This implies
$\overline X_{2n+2} = \varnothing$ and $\HP^n = \bigsqcup_{i=0}^n
X_{2i}$ and that $X_{2n} = \overline X_{2n}$ is affine.  But the other
$X_{2i}$ are obtained by removing a nonempty closed subscheme of
codimension $2$ from a smooth affine scheme, and such schemes are
never affine.

(b)
If one chooses $g_1, \dots, g_i \in \GL_{n+1} \subset
\Sp_{2n+2}$ satisfying $\bigoplus_{j=1}^i g_j(E_1) = E_i$ then one has
$\overline X_{2i} = \bigcap_{j=1}^i g_j(\overline X_2)$ because both
parametrize subspaces $U \subset \bigcap g_j(E_1)^\perp = E_i^\perp$
with $\phi \rest U \not\equiv 0$.
Moreover, the intersection is transversal because the map
$\bigoplus_{j=1}^i N\dual_{g_j(\overline X_1)/\HP^n,[U]} \to
T\dual_{\HP^n,[U]}$ which needs to be injective can be identified with
the natural map $\bigoplus_{j=1}^i \bigl( U \otimes g_j(E_1) \bigr)
\to U \otimes U^\perp$.

(c)
{\em Idem} with $i = n+1$.
\end{proof}

We have two other results which
make the stratification \eqref{E:strat} of $\HP^{n}$ resemble
the cell decomposition of $\mathbb{HP}^{n}$.
We will prove them later, but we state them now so they don't get lost amid
more general results about quaternionic Grassmannian bundles.

\begin{theorem}
\label{th.subgrass}
There is a natural map $q \colon \overline X_{2i} \to \HP^{n-i} =
\HP(E_i^\perp/E_i, \overline{\phi})$ which is an $\Aff^{2i}$-bundle.  The
stratification $\overline X_{2i} = X_{2i} \sqcup X_{2i+2} \sqcup
\cdots \sqcup X_{2n}$ is the inverse image under $q$ of the
stratification $\HP^{n-i} = X_0' \sqcup X_2' \sqcup \cdots \sqcup
X_{2n-2i}'$ associated to the flag of subspaces of $E_i^\perp/E_i$
induced by \eqref{iso.flag}.  The pullbacks to $\overline X_{2i}$ of
the two tautological symplectic subbundles are naturally
isometric $q^* (\shf U_{\HP^{n-i}}, \overline \phi \rest {\shf
  U_{\HP^{n-i}}}) \cong (\shf U_{\HP^n}, \phi \rest {\shf U_{\HP^n}} )
\rest {\overline X_{2i}}$.
\end{theorem}



\begin{corollary}
\label{C:small.cell}
There is a natural isomorphism $\overline X_{2n} = X_{2n} \cong
\Aff^{2n}$.
\end{corollary}

\begin{theorem}
\label{A.X2i}
\parens{a}
The stratum $X_{2i}$ is the quotient of the
free action of $\GG_{a}$ on
$\Aff^{4n-2i+1} = \Aff^{i} \times \Aff^{2n-2i} \times \Aff^{i} \times \Aff^{2n-2i} \times \Aff^{1}$
given by
\begin{equation}
\label{E:action.1}
t \cdot (\alpha,a,\beta,b, r) = (\alpha,a,\beta+t\alpha, b+ ta,
r + t(1-\phi({a},{b})))
\end{equation}
where
$\phi
\colon \Aff^{2n-2i} \times \Aff^{2n-2i} \to \Aff^{1}$
is the standard symplectic form.

\parens{b}
For any scheme $S$, pullback along the structural
map $t \colon X_{2i} \to k$ induces isomorphisms $(t \times 1_S)^A \colon
A(S) \arrowiso A(X_{2i} \times S)$ for any cohomology theory $A$.

\end{theorem}

Theorem \ref{th.subgrass} is
essentially a special case of
Theorem \ref{T:normal.geom}.
We will prove Theorem \ref{A.X2i} in \S\ref{S:open.stratum}.

Some curious things
are beginning to happen.  The strata in our decomposition are cohomological cells but
not affine spaces or even affine schemes.  They are of the right codimension but the
wrong dimension, for $\dim_{\RR}\HP^{n}(\CC) = 8n$ but $\dim_{\RR}\mathbb{HP}^{n} = 4n$.
The $\overline{X}_{2i}$ are not copies of the subspaces $\HP^{n-i}$ but vector bundles
over copies of the $\HP^{n-i}$, and the difference in dimensions is worse:  $\dim_{\RR}\HP^{n}(\CC) = 8n-4i$ but $\dim_{\RR}\mathbb{HP}^{n-i} = 4n-4i$.  There is no problem cohomologically, but it is
curious geometrically.


For a rank $2n+2$ symplectic bundle $(E,\phi)$ over a scheme $S$ there is a
\emph{quaternionic projective bundle} $\HP_{S}(E,\phi)$.  It is the open subscheme of the Grassmannian
bundle $\Gr_{S}(2,E)$ whose points over $s \in S$ correspond to those $U \subset E_{s}$ on which $\phi$
is nondegenerate.  Let $\pi \colon \HP_{S}(E,\phi) \to S$ be its structure map,
i.e.\ the natural projection
to $S$.
Then the symplectic bundle $\pi^{*}(E,\phi)$ splits as the orthogonal direct sum
\begin{equation}
\label{E:univ.split}
\pi^{*}(E,\phi) \cong (\shf U, \phi\rest{\shf U}) \perp (\shf U^{\perp}, \phi\rest{\shf U^{\perp}})
\end{equation}
of the tautological rank $2$ symplectic subbundle and its orthogonal complement.
This decomposition is \emph{universal} in the sense that for any
a morphism $g \colon X \to S$ any orthogonal direct sum decomposition
$g^{*}(E,\phi) \cong (F_1,\psi_{1}) \perp (F_{2},\psi_{2})$ with $\rk F_{1} = 2$
is the pullback along a unique morphism $f \colon X \to \HP_{S}(E,\phi)$  of the universal decomposition
\eqref{E:univ.split}.  The map $f$ is said to \emph{classify} either the decomposition or the
rank $2$ symplectic subbundle $(F_{1},\psi_{1}) \subset g^{*}(E,\phi)$.

We define the quaternionic  Grassmannians and partial flag varieties as the quotient varieties
\begin{gather*}
\HGr(r,n) = \Sp_{2n} / ( \Sp_{2r} \times \Sp_{2n-2r}),
\\
\HFlag(a_{1},\dots,a_{r};n) = \Sp_{2n} / ( \Sp_{2a_{1}} \times  \cdots \times\Sp_{2a_{r}} \times \Sp_{2n-\sum 2a_{i}}).
\end{gather*}
The second family of schemes includes the first, so we discuss it.
Over  $\HFlag(a_{1},\dots,a_{r};n)$ there
are $r+1$ universal symplectic subbundles $(\shf U_{1},\phi_{1}),\dots,(\shf U_{r},\phi_{r}),(\shf V,\psi)$
with $\rk \shf U_{i} = 2a_{i}$ and $\rk \shf V = 2n-\sum 2a_{i}$ and a canonical decomposition of the
trivial symplectic bundle of rank $2n$ into their orthogonal direct sum
\[
(V,\phi)\otimes \OO \cong (\shf U_{1},\phi_{1}) \perp \cdots \perp
(\shf U_{r},\phi_{r}) \perp (\shf V,\psi)
\]
Moreover, any decomposition over a scheme $S$ of the trivial symplectic bundle
$(V,\phi) \otimes \OO_{S}$ into an orthogonal direct sum of symplectic subbundles of the appropriate
ranks is the pullback of this universal decomposition
along a unique morphism $S \to \HFlag(a_{1},\dots,a_{r};n)$.

Now write $\shf F_{i} =  \bigoplus_{j=1}^{i} \shf U_{i}$.  We then have a filtration
\begin{equation}
\label{E:flag}
0 \subset \shf F_{1} \subset \shf F_{2} \subset \cdots \subset \shf F_{r} \subset V \otimes \OO
\end{equation}
Clearly $\HFlag(a_{1},\dots,a_{r};n)$ parametrizes flags of subspaces
$0 \subset F_{1} \subset \cdots \subset F_{r} \subset V$ of the appropriate dimensions such that
$\phi \rest{F_{i}}$ is nondegenerate for all $i$.  So it is an open subscheme of a flag variety.

\begin{theorem}
The quaternionic flag varieties are dense open subschemes of the flag varieties
\[
\HFlag(a_{1},\dots,a_{r};n) \subset \Flag(2m_{1},2m_{2},\dots,2m_{r};2n).
\]
with $m_{i} = \sum_{j=1}^{i} a_{j}$.
They are smooth and affine of relative dimension $4n \sum_{i} a_{i} - 4 \sum_{i\leq j} a_{i}a_{j}$
over the base $k$,
and they satisfy $H^{0}(\HFlag,\OO^{\times}) = \OO(k)^{\times}$ and $\Pic(\HFlag) = \Pic(k)$.
\end{theorem}

The proof is essentially the same as that of Theorem \ref{th.basic}.

%

\section{Normal bundles of sub-Grassmannians}
\label{S:normal.subgrass}

For ordinary Grassmannians, there are closed embeddings
$\Gr(r,n-1) \into \Gr(r,n)$, and the complement of the image is isomorphic
to a vector bundle over $\Gr(r-1,n-1)$.  For the cohomology there is then a long exact sequence
\begin{equation}
\label{E:Grass.sequence}
\cdots \to A_{\Gr(r,n-1)}(\Gr(r,n)) \to
A(\Gr(r,n)) \to A(\Gr(r-1,n-1)) \to \cdots.
\end{equation}
Moreover,
the normal bundle $N$ of $\Gr(r,n-1)$ in $\Gr(r,n)$ embeds as an open subvariety of $\Gr(r,n)$,
and so excision gives an isomorphism
$A_{\Gr(r,n-1)}(\Gr(r,n)) \cong A_{\Gr(r,n-1)}(N)$.

We will show that something similar happens for the cohomology of quaternionic Grassmannians
but with many differences in the details.

We work in a relative situation.  Our basic setup is as follows.
Suppose $(E,\phi)$ is a symplectic bundle of rank $2n-2$ over $S$, and let $(F,\psi)$
be the symplectic bundle of rank $2n$ with
\begin{align}
\label{E:basic.setup}
F & = \OO_{S} \oplus E \oplus \OO_{S}, &
\psi & =
\begin{pmatrix}
0 & 0 & 1 \\ 0 & \phi & 0 \\ -1 & 0 & 0
\end{pmatrix}.
\end{align}
We will consider the natural embedding of
$\HGr_{S}(r;E,\phi)$ in $\HGr_{S}(r;F,\psi)$.
Let $(\shf U_{E},\phi_{\shf U_{E}})$ and $(\shf U_{F},\psi_{\shf U_{F}})$ be the
tautological symplectic subbundles of rank $2r$.
We will abbreviate
\begin{align}
\label{E:HGr.abbrv}
\HGr(E) & = \HGr_{S}(r;E,\phi)
&
\HGr(F) & = \HGr_{S}(r;F,\psi).
\end{align}
%
We now study the embedding of $\HGr(E)$ in $\HGr(F)$.

\begin{theorem}
\label{T:normal.geom}
\parens{a} The normal bundle $N$ of $\HGr(E)$ in $\HGr(F)$ has a
canonical embedding
as an open subscheme of $\Gr_{S}(2r;F)$ overlapping
the open subscheme $\HGr_{S}(r;F,\psi)$.

\parens{b} The subschemes $N^{+} = \HGr(F) \cap \Gr_{S}(2r;\OO_{S}\oplus E)$ and
$N^{-} = \HGr(F) \cap \Gr_{S}(2r;E\oplus \OO_{S})$ are closed in $\HGr(F)$
and are subbundles of $N$ with $N = N^{+} \oplus N^{-}$.
We have $N^{+}\cap N^{-} = \HGr(E)$.

\parens{c} There are natural
vector bundle isomorphisms $N^{+} \cong \shf U_{E}$ and
$N^{-} \cong \shf U_{E}$.

\parens{d} There is a natural section $s_{+}$ of $\shf U_{F}$
intersecting the zero section transversally in $N^{+}$ and similarly for $N^{-}$.

\parens{e} Let $\pi_{+} \colon N^{+} \to \HGr(E)$ be the structural map.  Then
$\pi_{+}^{*}(\shf U_{E},\phi\rest{\shf U_{E}})$ is isometric to
$(\shf U_{F},\psi\rest{\shf U_{F}})\rest{N^{+}}$ and similarly for $N^{-}$.
\end{theorem}

Actually in (b) and (c) the truly natural isomorphisms are $N^{+} \cong \shf U_{E}^{\vee}$ and
$N^{-} \cong \shf U_{E}^{\vee}$, while $s_{+}$ is naturally a section of $\shf U_{F}^{\vee}$.
But since $\shf U_{E}$ and $\shf U_{F}$ are symplectic, this does not matter.

When $S$ is $k$, and $(E,\phi)$ is the trivial hyperbolic symplectic bundle $\Hyp(\OO_{k}^{\oplus n-1})$,
then we are in the situation of Theorem \ref{th.subgrass} with
$\HGr(F) = \HP^{n-1}$ and $N^{+} = \overline{X}_{2}$ and $\HGr(E) = X'_{0}$.

\begin{proof}
(a)
%
The normal bundle $N$ of $\HGr(E)$ in $\HGr(F)$
is isomorphic to $\shf U_{E}^{\vee} \otimes F/E \cong \shf U_{E}^{\vee} \otimes \OO^{\oplus 2}$.
Therefore $N$ has a universal property: to give a map $T \to N$ one gives a map
$t \colon T \to S$ plus a symplectic subbundle $i \colon U \subset t^{*}E$ of rank $2r$ plus a morphism
$(\alpha_{1},\alpha_{2}) \colon U \to \OO_{T}^{\oplus 2}$.
Giving such data is the same as giving $t \colon T \to S$ plus a rank $2r$ subbundle
$(\alpha_{1},i,\alpha_{2}) \colon U \into \OO_{T} \oplus t^{*}E \oplus \OO_{T} = t^{*}F$
such $i^{\vee}\phi i$ is everywhere of maximal rank.  So $N$ is naturally isomorphic to an open subscheme
of $\Gr_{S}(2r;F)$.

(b)(e)
To give a morphism $T \to N^{+}$ one gives a map $t \colon T \to S$ plus a subbundle
of rank $2r$ of the form
$(\alpha_{1},i,0) \colon U \into \OO_{T} \oplus t^{*}E \oplus \OO_{T} = t^{*}F$ such that
$(\alpha_{1},i,0)^{\vee}\psi(\alpha_{1},i,0)$ is nonsingular.  That is equivalent to giving
$t \colon T \to S$ and $i \colon U \into t^{*}E$ such that $i^{\vee}\phi i$ is nonsingular plus
$\alpha_{1} \colon U \to \OO_{T} \oplus 0$.  So $N^{+}$ is the subbundle
$\shf U_{E}^{\vee} \otimes (\OO_{S} \oplus 0)$ of $N = \shf U_{E}^{\vee} \otimes \OO_{S}^{\oplus 2}$.
Similarly $N^{-} = \shf U_{E}^{\vee} \otimes (0 \oplus \OO_{S})$.  The other assertions of (b) are clear.
The two presentations have the same bundle $U$ and the same form
$(\alpha_{1},i,0)^{\vee}\psi(\alpha_{1},i,0) = i^{\vee}\phi i$.  Hence (e).

(c) By construction we have $N^{+} \cong \shf U_{E}^{\vee}$.  But $\shf U_{E}$ is symplectic,
so $\shf U_{E}^{\vee} \cong \shf U_{E}$.

(d) For $\shf U_{F}$ we have the inclusion
$(\alpha_{1},i,\alpha_{2}) \colon \shf U_{F} \into \OO_{S} \oplus E \oplus \OO_{S} = F$.
The scheme $N^{+}$ is the zero locus of $\alpha_{2}$ or equivalently of
$\alpha_{2}^{\vee} \colon \OO_{S} \to \shf U_{F}^{\vee} \cong \shf U_{F}$.
\end{proof}

In the proof of Theorem \ref{T:unique.2} we will use the following subtle point about
the geometry of Theorem \ref{T:normal.geom}.
The \emph{tautological section} of
$\pi_{-}^{*}N_{-} = N_{-} \times_{\HGr(E)} N_{-} \xrightarrow{b_{2}} N_{-}$
is the diagonal $\Delta$.

\begin{lemma}
\label{L:subtle}
The isomorphism
$\shf U_{F} \rest{N^{-}} \cong \pi_{-}^{*} \shf U_{E}$ of \parens{d}
and the pullback $\pi_{-}^{*} \shf U_{E} \cong \pi_{-}^{*} N^{-}$ of the isomorphism of \parens{b}
identifies the restriction $s_{+} \rest{N^{-}}$ of the section of \parens{c} with the tautological section
of $\pi_{-}^{*} N^{-}$.
\end{lemma}

Both are $\alpha_{2}^{\vee}$.  We leave the details to the reader.

\begin{proposition}
\label{P:iso}
In the situation of Theorem \ref{T:normal.geom} let
$f \colon N^{-} \into \HGr(F)$ be the inclusion.
Then for any cohomology theory $A$ the map
$f^{A} \colon A_{N^{+}}(\HGr(F)) \to A_{\HGr(E)}(N^{-})$ is an isomorphism.
\end{proposition}

\begin{proof}
We have a direct sum $N = N^{+} \oplus N^{-}$ of vector bundles over
$\HGr(E)$, so the pullback and its left inverse the
restriction map $A_{N^{+}}(N) \to A_{\HGr(E)}(N^{-})$ are isomorphisms by homotopy invariance.
Moreover, $N^{+}$ is a closed subscheme of both $N$ and of $\HGr(F)$, which are both
open subschemes of $\Gr_{S}(2r;F)$.  So the diagram
\begin{equation}
\label{E:Zar.excision}
\vcenter{
\xymatrix @M=5pt @C=30pt {
A_{N^{+}}\bigl( N \cap \HGr(F) \bigr) \ar@{<-}[d]_-{\text{excision}}^-{\cong}
\ar@{<-}[r]^-{\text{excision}}_-{\cong} \ar[rd]
& A_{N^{+}}(\HGr(F)) \ar[d]^-{f^{A}}
\\
A_{N^{+}}(N) \ar[r]_-{\text{v.b.}}^-{\cong}
& A_{\HGr(E)}(N^{-})
}}
\end{equation}
commutes, and the top and left arrows are isomorphisms by Zariski excision.  Hence $f^{A}$ is also
an isomorphism.
\end{proof}

\section{Geometry of the open stratum}
\label{S:open.stratum}

In this section we study the geometry and cohomology of the open stratum in the localization
sequence for quaternionic sub-Grassmannians.  Our main results are Theorems \ref{th.Ga}
and \ref{T:Ga.cohom}.  We include a proof of Theorem \ref{A.X2i} as a special case of the latter.
A second proof is Theorem \ref{2nd}.

Recall the basic setup of \eqref{E:basic.setup}.  We have
a symplectic bundle $(E,\phi)$ of rank $2n-2$ over $S$, and let $(F,\psi)$
is of rank $2n$ with $F = \OO_{S} \oplus E \oplus \OO_{S}$ and $\psi$ the orthogonal
direct sum of $\phi$ and the hyperbolic symplectic form.
We will consider the natural embedding of
$\HGr_{S}(r;E,\phi)$ in $\HGr_{S}(r;F,\psi)$.
Let $\lambda \colon F \onto \OO_{S}$ be the projection onto the third factor with kernel
$\OO_{S} \oplus E$.  For $r \leq n-1$ let
\begin{align}
\label{E:XY}
N^{+} & = \HGr_{S}(r,F,\psi) \cap \Gr_{S}(2r,\OO_{S} \oplus E), &
Y & = \HGr_{S}(r,F,\psi) \smallsetminus N^{+}.
\end{align}
with $N^{+} \subset \HGr_{S}(r,F,\psi)$ closed of codimension $2r$ and $Y$ open.
Let $(\shf U_{r},\psi_{r})$ be the tautological
symplectic subbundle of rank $2r$ on $\HGr(r,F,\psi)$, and let $(\shf U_{r-1},\phi_{r-1})$ be the
tautological symplectic subbundle of rank $2r-2$ on $\HGr(r-1,E,\phi)$.
In this section we will investigate the open stratum $Y$.

\begin{theorem}
\label{th.Ga}
In the situation of \eqref{E:basic.setup} and \eqref{E:XY} we have morphisms over $S$
\[
Y \xleftarrow{g_{1}} Y_{1} \xleftarrow{g_{2}} Y_{2} \xrightarrow{q} \HGr_{S}(r-1,E,\phi)
\]
with $g_{1}$ an $\Aff^{2r-1}$-bundle, $g_{2}$ an $\Aff^{2r-2}$-bundle, and
$q$ an $\Aff^{4n-3}$-bundle.  Moreover,
$g_{2}^{*}g_{1}^{*}\shf U_{r}$ has two tautological sections $e,f$ over $Y_{2}$ satisfying the properties
\begin{align}
\label{E:threecond}
\lambda(f) & = 1, &
\lambda(e) & = 0, &
\psi(e,f) & = 1,
\end{align}
and writing $\pi \colon Y_{2} \to S$ for the projection,
there is a symplectic automorphism
$\rho$ of  $\pi^{*}(F,\psi) = (\OO_{Y_{2}} \oplus \pi^{*}E \oplus \OO_{Y_{2}},\psi)$
with $\rho(1,0,0) = e$ and $\rho(0,0,1) = f$ and an orthogonal direct sum of
symplectic subbundles of $\pi^{*}(F,\psi)$
\begin{equation}
\label{E:breakoff}
g_{2}^{*}g_{1}^{*}
\shf U_{r}
=
\langle e,f \rangle
\perp
\rho
( q^{*}
\shf U_{r-1}
)  \subset \pi^{*}
F.
\end{equation}
\end{theorem}

In \eqref{E:threecond} the bilinear form is $\psi(f,e) = f^{\vee}\psi e$ for $\psi \colon F \to F^{\vee}$.

\begin{theorem}
\label{T:Ga.cohom}
In the situation of Theorem \ref{th.Ga} with $Y = \HGr_{S}(r,F,\psi) \smallsetminus N^{+}$
the open locus of \eqref{E:XY}, the following hold for any cohomology theory $A$.

\subthm
For $r=1$ let $t \colon Y \to S$ be the projection map.
Then $t^{A} \colon A(S) \to A(Y)$ is an isomorphism.

\subthm
For $1 \leq r \leq n$ let
\[
\HGr_{S}(r-1,E,\phi) \xrightarrow{\sigma} Y \subset \HGr_{S}(r,F,\psi)
\]
be the map classifying the rank $2r$ symplectic subbundle $\OO \oplus \shf U_{r-1} \oplus \OO$
of the pullback of $F$.  Then the map
$\sigma^{A} \colon A(Y) \to A(\HGr_{S}(r-1,E,\phi))$ is an isomorphism.
\end{theorem}

The localization exact sequence
\begin{equation}
\label{E:long.exact}
\cdots \to A_{N^{+}}(\HGr(F)) \to A(\HGr(F)) \to A(\HGr(F) \smallsetminus N^{+}) \to \cdots
\end{equation}
combines with
Proposition \ref{P:iso} and Theorem \ref{T:Ga.cohom} to give the following result.

\begin{corollary}
\label{C:localization}
In the situation of \eqref{E:basic.setup}--\eqref{E:XY} for any cohomology theory $A$
the localization sequence for the closed embedding
$N^{+} \into \HGr(r,F,\psi)$ and the complementary open embedding
$j \colon Y \into \HGr(r,F,\psi)$
is isomorphic to
\begin{equation}
\label{E:the.long.exact}
\cdots \to A_{\HGr(r,E,\phi)}(\shf U_{E}) \xrightarrow{e^{A}(f^{A})^{-1}} A(\HGr(r,F,\psi))
\xrightarrow{\sigma^{A}j^{A}} A(\HGr(r-1,E,\phi)) \xrightarrow{f^{A}\partial(\sigma^{A})^{-1}} \cdots
\end{equation}
where $\shf U_{E}$ is the tautological rank $2r$ symplectic subbundle on $\HGr(r,E,\phi)$, and
$f$ is as in Proposition \ref{P:iso}, $\sigma$ as in Theorem \ref{T:Ga.cohom}, and $e^{A}$ is the
extension of supports operator.
\end{corollary}

We prove a series of lemmas before proving the theorems.

\begin{lemma}
\label{L:bdl.surj.1}
Let $U$ be a vector bundle of rank $m$ over $S$ with a quotient line bundle
$\lambda \colon U \onto \OO_{S}$ with a fixed trivialization.  Then the locus
\[
Y = \{ f \in U \mid \lambda(f) = 1 \} \subset U
\]
is an $\Aff^{m-1}$-bundle over $S$ which has a section if and only if $\lambda$ is split.
Moreover,
giving a morphism $T \to Y$ is equivalent to giving a pair $(g,f_{T})$ with
$g \colon T \to S$ a morphism, and $f_{T}$ a section of $g^{*}U$ with $\lambda(f_{T}) = 1$.
\end{lemma}


\begin{proof}
The subbundle $\ker \lambda$ acts on $U$ by translation in the fibers, and this action
preserves $Y$.
Over any open subscheme $S_{\alpha} \subset S$ where $\lambda$ has a local splitting
$\sigma_{\alpha}$, one has $Y \rest{S_{\alpha}} = (\ker \lambda) \rest{S_{\alpha}} + \sigma_{\alpha}(1)$,
so $Y \to S$ is a torsor for $\ker \lambda$.
A global section $S \to Y$ is by construction a section $s$ of $U$ with $\lambda(s) = 1$.
It exists if and only if $\lambda$ is split.
Finally, giving a morphism $T \to U$ is equivalent to giving a morphism $g \colon T \to S$
plus a section $f_{T}$ of $g^{*}U$ because of the cartesianness of
\[
\xymatrix @M=5pt @C=30pt {
g^{*}U \ar[r] \ar[d] \ar@{}[dr]|-{\square}
& U \ar[d]
\\
T \ar[r]_-{g} \ar@/^1.5pc/[u]^-{f_{T}}
& S
}
\]
The image of $T \to U$ lies in $Y \subset U$ if and only if $\lambda(f_{T}) = 1$.
\end{proof}

\begin{lemma}
\label{L:bdl.surj.2}
Suppose that $(U,\psi)$ is a rank $2s$ symplectic bundle over $S$ with a quotient line bundle
$\lambda \colon U \onto \OO_{S}$ with a fixed trivialization.  Let
\begin{align*}
Y_{1} & = \{ f \in U \mid \lambda(f) = 1 \} \subset U, \\
Y_{2} & =\{ (e,f) \in U \oplus U \mid
\text{$\lambda(f) = 1$ and $\lambda(e) = 0$ and $\psi(e,f) = 1$} \} \subset U \oplus U.
\end{align*}
Then the projection map $Y_{1} \to S$ is an $\Aff^{2s-1}$-bundle which has a section
if and only if $\lambda$ is split.
The projection map $Y_{2} \to Y_{1}$ is an $\Aff^{2s-2}$-bundle which has a section, and when
$\lambda$ is split the composition $Y_{2} \to Y_{1} \to S$ is an $\Aff^{4s-3}$-bundle with a
section.
Moreover, giving a morphism $T \to Y$ is equivalent to giving a triple $(g,e_{T},f_{T})$ with
$g \colon T \to S$ a morphism and $e_{T},f_{T}$ sections of $g^{*}U$ satisfying the three
conditions \eqref{E:threecond}.
\end{lemma}

Writing $\pi \colon Y_{2} \to S$ for the projection, the universal property implies the existence of
\emph{tautological sections} $e,f$ of $\pi^{*}U$ satisfying the three conditions and such that the
triples $(g,e_{T},f_{T})$ of the lemma are the pullbacks of the tautological $(\pi,e,f)$ along
the classifying map $T \to Y_{2}$.

\begin{proof}
The statements about $Y_{1}$ follow from immediately from the previous lemma.
Let $i \colon \ker \lambda \to U$ be the inclusion, and $h \colon Y_{1} \to S$ the projection map.
Over $Y_{1}$ we have morphisms
\[
\xymatrix @M=5pt @C=20pt {
\OO_{Y_{1}} \oplus h^{*}(\ker \lambda)
\xrightarrow[\cong]{(f,i)} h^{*}U \ar[r]_-{\cong}^-{\psi} &
h^{*}U^{\vee}
\ar@{->>}[r]^-{f^{\vee}} &
\OO_{Y_{1}}.
}
\]
The map $f^{\vee}$ is surjective because $f \colon \OO_{Y_{1}} \to h^{*}U$ is nowhere vanishing.  We have
$f^{\vee}\psi f =0$ since $\psi$ is alternating.
So $f^{\vee}\psi i \colon h^{*}(\ker \lambda) \onto \OO_{Y_{1}}$ is
surjective.  Applying the previous lemma to $f^{\vee}\psi i$ gives $Y_{2}$.
The projection $Y_{2} \to Y_{1}$ is always split because by the nondegeneracy of $\psi$
there exists a nowhere vanishing section $e_{0}$ of $U$ such that
$\lambda(?) = \psi(e_{0},?)$,  and $-e_{0}$ gives a splitting of $f^{\vee}\psi i$.

When $\lambda$ is split by a section $f_{0} \colon \OO_{S} \to U$, then the two sections $e_{0},f_{0}$ form
a hyperbolic plane, and setting $E = \langle e_{0},f_{0} \rangle^{\perp}$ and $\phi = \psi \rest{E}$, we find
that up to isometry we have
$U = \OO_{S} \oplus E \oplus \OO_{S}$, with $\psi$ as in \eqref{E:basic.setup},
and with $\lambda \colon U \onto \OO_{S}$ the third projection.
By construction $Y_{2} \subset U \oplus U$ is the locus of points $(e,f)$ satisfying the three conditions
\eqref{E:threecond}.
The conditions $\lambda(f) = 1$ and $\lambda(e) = 0$ imply that
$e$ and $f$ are of the forms $e = (a,u,0)$ and $f = (b,v,1)$
in terms of the direct sum decomposition.
The condition $\psi(e,f) = 1$ is $a+\phi(u,v) = 1$, so we get $e = (1-\phi(u,v),u,0)$.
So $Y_{2}$ is the image of the vector bundle $E \oplus E \oplus \OO_{S}$ under the map
sending $(u,v,b) \mapsto ((1-\phi(u,v),u,0),(b,v,1))$.  So it is an $\Aff^{4s-3}$-bundle with a section.

Finally, applying Lemma \ref{L:bdl.surj.1} twice we see that giving a morphism $T \to Y_{2}$
is equivalent to giving
a pair $(g_{1},e)$ with $g_{1} \colon T \to Y_{1}$ a morphism and $e$ a section of $g_{1}^{*}h^{*}(\ker \lambda)$
such that $f^{\vee}\psi i(e) = \psi(i(e),f) = 1$, and giving $g_{1}$ is equivalent to giving $(g,f)$ with
$g  \colon T \to S$ a morphism and $f$ a section of $g^{*}U$ such that $\lambda(f)=1$.  Since in these
equivalences one has $g = hg_{1}$, we see that giving $T \to Y_{2}$ is equivalent to giving $(g,e,f)$
with $e,f$ sections of $g^{*}U$ satisfying \eqref{E:threecond}.
\end{proof}

Suppose now that we have $(E,\phi)$ and $(F,\psi)$ as in \eqref{E:basic.setup}.
Let $\Sp_{S}(F,\psi) \to S$ be the group scheme of automorphisms of $(F,\psi)$.  Its sections correspond
to endomorphisms $\alpha \colon F \to F$ such that $\alpha^{\vee}\psi \alpha = \psi$.
There are two maps $\rho_{+} \colon E \oplus \OO_{S} \to \Sp_{S}(F,\psi)$ and
$\rho_{-} \colon E \oplus \OO_{S} \to \Sp_{S}(F,\psi)$ with
\begin{align}
\label{E:unipotent}
\rho_{+}(e,s)  & =
\begin{pmatrix}
1_{\OO_{S}} & e^{\vee}\phi & s \\ 0 & 1_{E} & e \\ 0 & 0 & 1_{\OO_{S}}
\end{pmatrix}, &
\rho_{-}(e_{1},s_{1})  & =
\begin{pmatrix}
1_{\OO_{S}} & 0 & 0 \\  e_{1 }& 1_{E} & 0 \\ s_{1} & -e_{1}^{\vee}\phi & 1_{\OO_{S}}
\end{pmatrix}
\end{align}
These maps parametrize the unipotent radicals of the parabolic subgroups of $\Sp_{S}(F,\psi)$
associated with the two filtrations compatible with the splitting \eqref{E:basic.setup}.

\begin{lemma}
\label{L:Sp}
In the above situation let $\lambda = (0,0,1) \colon F = \OO_{S} \oplus E \oplus \OO_{S} \onto \OO_{S}$.
Suppose we have a morphism $\pi \colon T \to S$ and two sections $e,f$ of $\pi^{*}F$
with $\lambda(f) = 1$ and  $\lambda(e) = 0$ and $\psi(e,f) = 1$.
Then $e$ and $f$ are of the form
\begin{align}
\label{E:auv}
e & =
\begin{pmatrix}
1-\phi(u,v) \\ u \\ 0
\end{pmatrix},
&
f & =
\begin{pmatrix}
b \\ v \\ 1
\end{pmatrix}.
\end{align}
Set $\rho = \rho_{+}(v,b) \rho_{-}(u,0)$.
Then $\rho$ is a symplectic automorphism of $\pi^{*}(F,\psi)$ satisfying
\begin{align*}
\rho
\begin{pmatrix}
0 \\ 0 \\ 1
\end{pmatrix}
& = f, &
\rho
\begin{pmatrix}
1 \\ 0 \\ 0
\end{pmatrix}
& = e, &
\rho(E) & = \langle e,f \rangle^{\perp}.
\end{align*}
\end{lemma}

\begin{proof}
The conditions $\lambda(f) = 1$ and $\lambda(e) = 0$ imply that $e$ and $f$ can be written in the
forms $e = (a,u,0)$ and $f = (b,v,1)$.  The condition $\psi(e,f) = 1$ is $a+\phi(u,v) = 1$, so
we get $e = (1-\phi(u,v),u,0)$.  One calculates
$\rho (0,0,1) = f$ and $\rho (1,0,0) = e$,
and since $\rho$ preserves the symplectic form $\psi$,
it takes $\langle (1,0,0),(0,0,1) \rangle^{\perp} = E$
isometrically onto $\langle e,f \rangle^{\perp}$.
\end{proof}

\begin{proof}[Proof of Theorem \ref{th.Ga}]
Applying Lemma \ref{L:bdl.surj.2} with $U=F$ and $\lambda \colon F \onto \OO_{S}$ the
third projection, we get maps $Z_{2} \to Z_{1} \to S$
which are $\Aff^{2n-2}$- and $\Aff^{2n-1}$-bundles, respectively.  Moreover, since $\lambda$
has a splitting, the composition is an $\Aff^{4n-3}$-bundle.
Over $Z_{2}$
there are tautological sections $e,f$ of the pullback of $F$ with a universal property.

Over $Y$ the restriction $\lambda\rest{\shf U_{r}} \colon \shf U_{r} \to \OO_{Y}$ of the
third projection $F \onto \OO$ is surjective.  Hence Lemma \ref{L:bdl.surj.2}, applied with
$U = \shf U_{r}$ over $Y$, gives us maps
$Y_{2} \to Y_{1} \to Y$ which are $\Aff^{2r-2}$- and $\Aff^{2r-1}$-bundles, resp., with the
first one with a section,
and tautological sections $e,f$ of the pullback of $\shf U_{r}$ to $Y_{2}$
with a certain universal property.

We claim that the varieties $Y_{2}$ and $Z_{2} \times_{S} \HGr_{S}(r-1,E,\phi)$ are
characterized by equivalent universal properties and thus represent isomorphic functors
$(\Sm/S)^{\op} \to \Set$.  They are therefore isomorphic over $S$.

To give a morphism $T \to Y_{2}$ one gives $(g,U_{T},e_{T},f_{T})$ with $g \colon T \to S$ a morphism,
$U_{T} \subset g^{*}F$ a symplectic subbundle of rank $2r$, and
and $e_{T},f_{T}$ sections of $U_{T}$ satisfying the three conditions \eqref{E:threecond}.

To give a morphism $T \to Z_{2} \times_{S} \HP_{S}(r-1,E,\phi)$ one gives $(g,e_{T},f_{T},V_{T})$ with
$g \colon T \to S$ a morphism,
$e_{T},f_{T}$ sections of $g^{*}F$ satisfying the three conditions \eqref{E:threecond},
and a symplectic subbundle
$V_{T} \subset g^{*}E$ of rank $2r-2$.

Lemma \ref{L:Sp} gives us a symplectic automorphism $\rho$ of $g^{*}(F,\psi)$ with
$\rho(1,0,0) = e_{T}$ and $\rho(0,0,1) = f_{T}$ and $\rho(E) = \langle e_{T},f_{T} \rangle^{\perp}$.
It follows that the data determining morphisms from $T$ to the two varieties determine each other
using the formulas
\begin{align*}
U_{T} & = \langle e_{T},f_{T}, \rho(V_{T}) \rangle, &
V_{T} & = \rho^{-1}(U_{T}) \cap E.
\end{align*}
So $Y_{2}$ and $Z_{2} \times_{S} \HP_{S}(r-1,E,\phi)$ are isomorphic over $S$.

The last sentence of the theorem is a description of these data in the universal case
$T = Y_{2} \cong Z_{2} \times_{S} \HP_{S}(r-1,E,\phi)$.
\end{proof}

\begin{proof}[Proof of Theorem \ref{T:Ga.cohom}]
(b)
Let $\pi_{r-1} \colon \HGr_{S}(r-1,E,\phi) \to S$ be the projection.
According to the proof we have just given of Theorem \ref{th.Ga} there is a morphism
\[
s \colon \HGr_{S}(r-1,E,\phi) \to Z_{2} \times_{S} \HGr_{S}(r-1,E,\phi)
\]
corresponding to the data $(\pi_{r-1}, (1,0,0), (0,0,1), \shf U_{r-1})$.  It is a section of
the second projection.  It composition with the isomorphism and projections
\[
\HGr_{S}(r-1,E,\phi) \xrightarrow{s} Z_{2} \times_{S} \HGr_{S}(r-1,E,\phi) \cong
Y_{2} \xrightarrow{g_{2}} Y_{1} \xrightarrow{g_{1}} Y
\]
is the map classified by the symplectic subbundle
$\OO\oplus \shf U_{r-1} \oplus \OO \subset \pi_{r-1}^{*}F$, which is the
$\sigma$ of the statement of the theorem.  By Theorem \ref{th.Ga}
$g_{1}$ and $g_{2}$ are affine bundles, and $s$ is the section of an affine bundle.
So $g_{1}^{A}$, $g_{2}^{A}$, $s^{A}$ and their composition $\sigma^{A}$ are
isomorphisms in any cohomology theory $A$.

(a) For $r=1$ we have $\HGr_{S}(0,E,\phi) = S$, and $\sigma \colon S \to Y$ is a section of
$t \colon Y \to S$.  By (b) $\sigma^{A}$ is an isomorphism, and we have $1_{S}^{A} = t^{A}\sigma^{A}$,
so $t^{A}$ is an isomorphism.
\end{proof}

We conclude this section by proving Theorem \ref{A.X2i}.

\begin{lemma}
\label{L:torsor}
The fibration $Y_{1} \to Y$ of Lemma \ref{L:bdl.surj.2} is
a torsor under the unipotent group scheme $R_{u} \subset \Sp(U,\psi)$
of automorphisms of $U$ respecting the filtration
$0 \subset (\ker \lambda)^{\perp} \subset \ker \lambda \subset U$ and the trivialization
$U/\ker \lambda \cong \OO_{Y}$ and inducing the identity on
$\ker \lambda/(\ker \lambda)^{\perp}$.
\end{lemma}

This is easily seen when
the filtration is split with $U = F = \OO_{S} \oplus E \oplus \OO_{S}$ and $R_{u}$ is the
$\rho_{+}(E \oplus \OO_{S})$ of \eqref{E:unipotent}.


We now return to the decomposition $\HP^{n} = \bigsqcup_{i=0}^{n} X_{2i}$ of Theorem
\ref{th.basic}.

\begin{proof}[Proof of Theorem \ref{A.X2i}]
(a)
We first look at the open stratum $X_{0}$.
Let $(U,\psi)$ be the tautological rank $2$ symplectic subbundle over $X_{0} \subset \HP^{n}$.
By Lemmas \ref{L:bdl.surj.2} and \ref{L:torsor} the fibration $Y_{2} = Y_{1} \to X_{0}$ is a
torsor under
$R_{u} = \GG_{a}$, and
\[
Y_{2} = Y_{1} = \{(e,f) \mid \Aff^{2n+2} \times \Aff^{2n+2} \mid
\lambda(f) = 1, \, \lambda(e) = 0, \, \psi(e,f) = 1\},
\]
with the action of $\GG_{a}$ given by $t \cdot (e,f) = (e,f+te)$.  Writing
$f = (b_{1},\dots,b_{2n},b_{2n+1},1)$ and
$e = (a_{1},\dots,a_{2n}, 1-\sum_{j=1}^{n}(a_{2j-1}b_{2j}-a_{2j}b_{2j-1}),0)$ gives the
presentation of $X_{0}$ as a quotient of a free action of $\GG_{a}$ on $\Aff^{4n+1}$.
The argument for the other $X_{2i}$ is similar.

(b) This follows from (a) or directly from Theorem \ref{T:Ga.cohom}(a).
\end{proof}


\section{Deformation to the normal bundle}
\label{S:normal.bundle}

The localization sequence for Grassmannians \eqref{E:Grass.sequence} exists for all cohomology
theories, but it is particularly well-behaved for oriented theories.  (See Balmer-Calm\`es \cite{Balmer:2008fk} for what can happen in a non-oriented theory.)  So we should
now start discussing the class of cohomology theories for which the localization sequence for
quaternionic Grassmannians of Corollary \ref{C:localization} is well-behaved.  But before doing so
we discuss the deformation to the normal bundle construction.
Our notation follows Nenashev \cite{Nenashev:2007rm}.


The \emph{deformation space} $D(Z,X)$ for a closed embedding $Z \into X$ of regular schemes
is the complement
of the strict transform of
$Z \times 0$
in the blowup $\bl_{Z \times 0}(X \times \Aff^{1})$.
The inclusion and blowup maps compose to a map
$D(Z,X) \to X \times \Aff^{1}$, so $D(Z,X)$ has projections $p \colon D(Z,X) \to X$ and
$D(Z,X) \to \Aff^{1}$.
The fiber of $D(Z,X)$ over $0 \in \Aff^{1}$ is $N_{Z/X}$,
and its  fiber over $1 \in \Aff^{1}$ is $X$.
The strict transform of $Z \times \Aff^{1}$ is isomorphic to $Z \times \Aff^{1}$,
and the restriction of $p$ to it is the first projection $Z \times \Aff^{1} \to Z$.
We thus have maps
\begin{equation}
\label{E:def.norm.cone}
\vcenter{
\xymatrix @M=5pt @C=30pt @R=11pt {
N_{Z/X} \ar[r]^-{i_{0}} &
D(Z,X) \ar@<-3pt>[r]_-{p} &
X \ar@<-3pt>[l]_-{i_{1}},
\\
A_{Z}(N_{Z/X}) &
A_{Z \times \Aff^{1}}(D(Z,X)) \ar[r]^-{i_{1}^{A}}\ar[l]_-{i_{0}^{A}} &
A_{Z}(X).
}}
\end{equation}
The map $p^{A} \colon A(X) \to  A(D(Z,X))$ makes $i_{0}^{A}$ and $i_{1}^{A}$ into
morphisms of $A(X)$-bimodules.  Because of the blowup geometry the
$A(X)$-bimodule structure on $A_{Z}(N_{Z/X})$ factors through the restriction
$A(X) \to A(Z) \cong A(N_{Z/X})$.
When the maps $i_{0}^{A}$ and $i_{1}^{A}$ are isomorphisms, as they frequently are,
their composition
\begin{equation}
\label{E:deform}
d^{A}_{Z/X} \colon A_{Z}(N_{Z/X})   \xrightarrow{\cong} A_{Z}(X)
\end{equation}
is the \emph{deformation to the normal bundle isomorphism}.
The functoriality of the space $D(Z,X)$ makes the deformation to the normal bundle isomorphisms
functorial.
\begin{equation}
\label{diag.pullback.1}
\vcenter{
\xymatrix @M=5pt @C=40pt {
Z' = Z \times_X X' \ar[r]^-{g'} \ar[d]_-{i'} &
Z \ar[d]^-i &
N_{Z'/X'} \cong N_{Z/X}\times_{Z} Z'  \ar[r]^-{g_{N}} \ar[d] &
N_{X/Z} \ar[d]
\\
X' \ar[r]^-g
&
X
&
Z' \ar[r]^-{g'} &
Z
}}
\end{equation}

\begin{proposition}
\label{P:torindep.1}
Suppose that in \eqref{diag.pullback.1} we have a pullback square of schemes
with $i$ and $i'$ closed embeddings and that that induces a
pullback square of normal bundles.  Then in any cohomology theory $A$
for which the deformation to the normal bundle isomorphisms $d^{A}_{Z'/X'}$
and $d^{A}_{Z/X}$ exist, we have
$d^{A}_{Z'/X'} \circ g_{N}^{A} = g^{A} \circ d^{A}_{Z/X}$.
\end{proposition}

The main existence result is the following.

\begin{theorem}[\protect{\cite[Theorem 2.2]{Panin:2003rz}}]
\label{T:def.normal.bdl}
When $Z \into X$
is a closed embedding of smooth quasi-projective varieties over a field,
the deformation to the normal bundle isomorphism $d^{A}_{Z/X}$ exists for any cohomology theory
$A$.
\end{theorem}

D\'evissage theorems such as Quillen's for $K$-theory or Gille's for Witt groups \cite{Gille:2007hb}
produce Thom isomorphisms $A(Z) \cong A_{Z}(X)$ directly.  The compatibility of these morphisms
with the $i_{0}^{A}$ and $i_{1}^{A}$ of \eqref{E:def.norm.cone} gives deformation to the normal
bundle isomorphisms for all closed embedding of regular schemes for these theories.

The proof of Theorem \ref{T:def.normal.bdl} uses Nisnevich tubular neighborhoods which do not
always exist in mixed characteristic.  But we only need a simple case.

A regular closed subscheme of a regular scheme $i \colon Z \into X$
has a \emph{Zariski tubular neighborhood} if there exists a open subscheme $U \subset X$
containing $Z$ and
an open embedding $U \into N_{Z/X}$ such that the composition $Z \into U \into N_{Z/X}$ is the
zero section.

\begin{proposition}
\label{P:tubular}
If $Z \into X$ has a Zariski tubular neighborhood, then the deformation to the normal bundle
isomorphism $d^{A}_{Z/X}$ exists for any cohomology theory $A$.
\end{proposition}

\begin{proof}
If $U \subset Y$ is an open embedding with $Z$ closed in both $U$ and $Y$, then the
isomorphism $d^{A}_{Z/U}$ exists if and only if $d^{A}_{Z/Y}$ exists because
$D(Z,U)$ is an open subscheme of $D(Z,Y)$ in which $Z \times \Aff^{1}$ is closed, so Zariski
excision provides isomorphisms between the sources and targets of the different maps
$i_{0}^{A}$ and $i_{1}^{A}$.

Hence given that for the vector bundle $N_{Z/X}$ the deformation to the normal bundle isomorphism
$d_{Z/N_{Z/X}}$ exists and is the identity by a local calculation, $d_{Z/U}$ and $d_{Z/X}$ also exist.
\end{proof}

\begin{proposition}
\label{P:N+1}
In the situation of Theorem \ref{T:normal.geom} or \eqref{E:XY}
the deformation
to the normal bundle isomorphisms $d^{A}(N^{+},\HP(F,\psi)) \colon A_{N^{+}}(\HP(F,\psi)) \to
A_{N^{+}}(\shf U_{F}\rest{N^{+}})$ exist for any cohomology theory $A$.
\end{proposition}

\begin{proof}
First, by Theorem \ref{T:normal.geom}(a)(b) $N^{+}$ is a direct summand of the normal bundle
$N = N^{+} \oplus N^{-}$ of $\HP_{S}(E,\phi) \subset \HP_{S}(F,\psi)$,
which embeds as an open subscheme $N \subset \Gr_{S}(2,F)$
overlapping the open subscheme $\HP_{S}(F,\psi)$.  It follows
that the projection map $N \onto N^{+}$ is also the structural map of the normal bundle of
$N^{+}$ in $\HP_{S}(F,\psi)$ and that $N \cap \HP_{S}(F,\psi)$ is a Zariski tubular neighborhood
of $N^{+}$ in $\HP_{S}(F,\psi)$.  Therefore by Proposition \ref{P:tubular} the deformation to the
normal bundles isomorphism $d^{B}(N^{+},\HP_{S}(F,\psi))$ exists.

Second, by Theorem \ref{T:normal.geom}(d) $N^{+}$ is the transversal intersection of a section $s_{+}$
of $\shf U_{F}$ with the zero section.  So the normal bundle of $N^{+}$ in
$\HP_{S}(F,\psi)$ is naturally isomorphic to $\shf U_{F} \rest{N^{+}}$.
\end{proof}

\section{Symplectic Thom structures}
\label{S:symp.thom}

An \emph{oriented cohomology theory} in the sense of Panin and
Smirnov \cite[Definition 3.1]{Panin:2003rz} is a ring cohomology theory such that for every vector
bundle $E \to X$ over a nonsingular variety and closed subset $Z
\subset X$ there is a distinguished {\em Thom isomorphism}
\[
\Th_Z^E \colon A_Z(X) \to A_Z(E)
\]
of $A(X)$-bimodules satisfying several axioms.
Using deformation to the normal bundle, Thom
isomorphisms induce isomorphisms $A_Z(X) \cong A_Z(Y)$ and in
particular $A(X) \cong A_{X}(Y)$
for any
closed embedding of smooth varieties $X \into Y$ over a field for oriented cohomology theories.

Some cohomology theories have Thom isomorphisms only for vector
bundles with some extra structure.
For instance Balmer's derived Witt groups
\cite{Balmer:1999if}
have Thom isomorphisms
$W^{i}(X) \to W^{i+n}_{X}(E)$ for triples $(E,L,\lambda)$ with $E$ a vector bundle of rank $n$,
$L$ a line bundle and $\lambda \colon L \otimes L \cong \det E$ an isomorphism.
This is because the isomorphisms depending on $E$ alone involve a twisted Witt group
$W^{i}(X,\det E) \cong W^{i+n}_{X}(E)$ (see for instance \cite{
Calmes:2008pi,
Gille:2007hb,
Nenashev:2007rm}),
and the $(L,\lambda)$ are required to specify an
isomorphism $W^{i}(X) \cong W^{i}(X,\det E)$.
Isometric symmetric bilinear line bundles $(L,\lambda)$ induce the same isomorphism, and
the isometry classes of possible $(L,\lambda)$ for a given $E$ are a torsor under
$H^{1}(X_{\text{\'et}}, \boldsymbol{\mu}_{2})$.
The extra structure is very reminiscent
of $\Spin$ and $\Spin^{c}$ structures in differential topology (e.g.\ \cite{Atiyah:1971zr,Atiyah:1976ly})
even to the point of having the same group $H^{1}(X,\ZZ/2)$ parametrizing the choices of
$\Spin$ structure
if $X$ is a complex projective variety.
Derived Witt groups might therefore be called \emph{an $\SL^{c}$-oriented} cohomology theory.

A cohomology theory is \emph{symplectically oriented} if to each symplectic
bundle $(E,\phi)$ over a variety $X$ and each closed subset $Z \subset X$ there is
an isomorphism $\Th^{E,\phi}_{Z} \colon A_{Z}(X) \to A_{Z}(E)$ satisfying the
properties which we will describe in detail in Definition \ref{D:symp.orient}.

Oriented cohomology theories are also symplectically oriented.
Witt groups, Witt cohomology, hermitian $K$-theory and oriented Chow groups
(also called Chow-Witt groups or hermitian $K$-cohomology)
are symplectically oriented as are symplectic
and special linear algebraic cobordism $\operatorname{MSp}^{*,*}$ and $\operatorname{MSL}^{*,*}$
defined along the lines of Voevodsky's \cite[\S 6.3]{Voevodsky:1998kx}
algebraic cobordism $\operatorname{MGL}^{*,*}$.

We will actually prove our quaternionic projective bundle theorem
using a structure which is weaker \emph{a priori} but is equivalent in the end.

\begin{definition}
\label{def.sp.thom}
A \emph{symplectic Thom structure} on a ring cohomology theory $A$ on a category of
schemes is
rule which assigns to each rank $2$ symplectic bundle $(E,\phi)$ over a
scheme $X$ in the category an $A(X)$-central element
$\Th(E,\phi) \in A_X(E)$ with the following properties:
\begin{enumerate}
\item For an isomorphism $u \colon (E,\phi) \cong (E_{1},\phi_{1})$
one has $\Th(E,\phi) = u^{A}\Th(E_{1},\phi_{1})$.

\item For a morphism $f \colon Y \to X$ with pullback map $f_{E} \colon f^{*}E \to E$
one has $f_{E}^{A}(\Th(E,\phi)) = \Th(f^{*}(E,\phi)) \in A_{Y}(f^{*}E)$.

\item For the trivial rank $2$ bundle $X \times \Aff^{2} \to X$ with symplectic structure given by
$\Hyp(\OO_{X}) = \bigl( \OO_{X}^{\oplus 2},
\bigl( \begin{smallmatrix} 0 & -1 \\ 1 & 0 \end{smallmatrix} \bigr) \bigr)$,
the map $\cup \Th(\Hyp(\OO_{X})) \colon A(X) \to A_{X}(X \times \Aff^{2})$ is an isomorphism.
\end{enumerate}

\end{definition}

Global nondegeneracy follows from the nondegeneracy condition above
via local trivializations of the symplectic bundle and a
Mayer-Vietoris argument.

\begin{proposition}
\label{P:nondeg}
Let $A$ be a ring cohomology theory with a symplectic Thom
structure.  Then for any rank $2$ symplectic bundle $(E,\phi)$ over a
scheme $X$  the map $\cup \Th(E,\phi) \colon A(X) \to A_X(E)$ is an
isomorphism.
\end{proposition}

For $(E,\phi)$ a rank $2$ symplectic bundle over a scheme $X$,
let $e^A \colon A_X(E) \to A(E)$ be the extension of supports map, and
let $z^A \colon A(E) \to A(X)$ be the restriction to the zero section (or
to any section).  Define the {\em Borel class}\/ of $(E,\phi)$ as
\begin{equation}
\label{eq.Borel}
b(E,\phi) = - z^Ae^A(\Th(E,\phi)) \in A(X).
\end{equation}
The sign follows the convention in differential topology (for instance Milnor-Stasheff \cite{Milnor:1974if})
where one has
$b_{i}(\xi) = (-1)^{i}c_{2i}(\xi)$ for a real or quaternionic vector bundle $\xi$.
Like the symplectic Thom classes, the Borel classes are
$A(X)$-central and are functorial with respect to isomorphisms of
symplectic bundles over $X$ and with respect to pullbacks along maps
$X' \to X$.

Now suppose that $i \colon X \to Y$ is a closed embedding of codimension
$2$ of regular schemes with a normal bundle $N = N_{X/Y}$ equipped
with a symplectic form $\phi$.  We then say that $i \colon X \to Y$ has a
{\em symplectic normal bundle}\/ $(N,\phi)$.  We can compose the isomorphism
of Proposition \ref{P:nondeg} with the deformation to
the normal bundle isomorphism when the latter exists
\begin{equation}
\label{diag.transfer}
i_{A,\flat} \colon  A(X) \xrightarrow[\cong]{\cup \Th(N,\phi)}
A_{X}(N) \xrightarrow[\cong]{d_{X/Y}} A_{X}(Y).
\end{equation}
The isomorphism $i_{A,\flat}$ is a \emph{Thom isomorphism},
while its composition $i_{A,\natural} \colon A(X) \to A_X(Y) \to A(Y)$
with the extension of supports is a {\em direct image map}.
The symbols $\flat$ and $\natural$ are placeholders indicating that
the maps depend on more than $i$ and $A$, namely the symplectic form
$\phi$ and the symplectic Thom structure.  The dependence on $\phi$
is very real in certain theories.  For Witt groups, replacing $\phi$ by $a \phi$ with $a
\in k^\times$ multiplies the direct image map by the class $\langle a
\rangle \in W(k)$.

We will need several facts about the Thom isomorphisms.  For a section
of a rank $2$ symplectic bundle $(E,\phi)$ the deformation to the normal
bundle map is the identity.  Hence:

\begin{proposition}
\label{P:tr=th}
Let $z \colon X \to E$ be the zero section of a rank $2$ symplectic bundle
$(E,\phi)$ over $X$ with symplectic normal bundle $(N_{X/E},\phi) =
(E,\phi)$, then the Thom isomorphism $z_{A,\flat} \colon A(X) \to A_X(E)$
coincides with $\cup \Th(E,\phi)$.
\end{proposition}

Now suppose we have a pullback diagram
\begin{equation}
\label{diag.pullback}
\vcenter{
\xymatrix @M=5pt @C=50pt {
X' = X \times_Y Y' \ar[r]^-{i'} \ar[d]_-{g'} &
Y' \ar[d]^-g \\
X \ar[r]^-i &
Y
}}
\end{equation}
Proposition \ref{P:torindep.1} and the functoriality of Thom classes give
us the next lemma.

\begin{proposition}
\label{P:torindep}
If in the pullback square \eqref{diag.pullback} all the schemes are regular,
and $i$ and $i'$ are closed embeddings of codimension $2$ with
symplectic normal bundles $(N_{X/Y},\phi)$ and $(N_{X'/Y'}, g'^*\phi)
\cong (g'^* N_{X/Y}, g'^*\phi)$ and the deformation to the normal bundle isomorphisms exist,
then we have $g^A i_{A,\flat} =
i'_{A,\flat} g'^A$.
\end{proposition}

The $A(X)$-centrality of $\Th(N,\phi)$ means that $\cup \Th(N,\phi)$
is an isomorphism of two-sided $A(X)$-modules and via $i^{A}$ also of two-sided $A(Y)$-modules.
The deformation to the normal bundle maps $\nu_{X/Y} \colon A_{X}(Y) \to A_{X}(N)$
are also isomorphisms of two-sided $A(Y)$-modules as discussed earlier.
 We therefore have the next proposition.

\begin{proposition}
\label{P:projection}
The Thom isomorphism $i_{A,\flat} \colon A(X) \to A_X(Y)$ is a two-sided
$A(Y)$-module isomorphism, and $i_{A,\flat}(1)$ is $A(Y)$-central.
Thus for $a \in A_X(Y)$ and $b \in A(Y)$ we have
\begin{gather*}
\begin{align*}
i_{A,\flat}(a \cup i^A b) & = i_{A,\flat}(a) \cup b, &
i_{A,\flat}(i^Ab \cup a) & = b \cup i_{A,\flat}(a), &&
\end{align*}
\\
i_{A,\flat}i^A(b) = i_{A,\flat}(1)  \cup b = b \cup
i_{A,\flat}(1).
\end{gather*}
\end{proposition}

We now prove some key formulas.

\begin{proposition}
\label{P:pont.zero}
Let $(E,\phi)$ be a rank $2$ symplectic bundle over a regular scheme $Y$ with a section
$s \colon Y \to E$ meeting the zero section $z$ tranversally in
$X$.  Suppose that $A$ is a ring cohomology theory with a symplectic
Thom structure, and let $\overline{e}^A \colon A_X(Y) \to A(Y)$ be the extension of supports map.
Then we have
\begin{equation}
\label{E:p.2nd}
b(E,\phi)  = -\overline{e}^{A} s^{A} (\Th(E,\phi)),
\end{equation}
Moreover, if the inclusion $i \colon X \into Y$ has a deformation to the normal bundle isomorphism,
then for all $b \in A(Y)$ we have
\begin{equation}
\label{E:i_*i^*}
i_{A,\natural}i^A(b) = \overline e^A i_{A,\flat} i^A(b)  = - b \cup b(E,\phi).
\end{equation}
\end{proposition}

\begin{proof}
In the diagram
\[
\xymatrix @M=5pt @C=40pt {
A(Y) \ar[r]^-{\cup \Th(E,\phi)}_-{= \ z_{A,\flat}} \ar[d]^-{i^A} &
A_Y(E) \ar[d]^-{s^A} \ar[r]^-{e^A}
&
A(E) \ar[d]_-{s^{A}} \ar@/^2pc/[d] ^-{z^A}
\\
A(X) \ar[r]_-{i_{A,\flat}}
&
A_X(Y) \ar[r]^-{\overline e^A}
&
A(Y) \ar@/_1pc/[u] |-{\pi^{A}}
}
\]
the righthand rectangle commutes by functoriality.
The pullbacks along the two sections of $\pi \colon E \to Y$ are
left inverses of the same isomorphism $\pi^{A}$, so they
satisfy
$s^{A} = z^{A}$.
We get
\[
b(E,\phi) = -z^{A}e^{A}(\Th(E,\phi)) = -s^{A}e^{A}(\Th(E,\phi)) = -\overline{e}^{A}s^{A}(\Th(E,\phi)).
\]

The lefthand rectangle of the diagram commutes
using the label $z_{A,\flat}$ by Proposition \ref{P:torindep}, and the equality
$ z_{A,\flat} = \cup \Th(E,\phi)$ is Proposition
\ref{P:tr=th}.  It follows that $\overline e^A i_{A,\flat} i^A(b) =
i_{A,\natural}i^A(b)$ is the same as $\overline{e}^{A} s^{A}(b \cup \Th(E,\phi))$.
Since all the maps are two-sided $A(Y)$-modules maps, that is the same
as $b \cup \overline{e}^{A} s^{A}(\Th(E,\phi)) = - b \cup b(E,\phi)$ using
formula \eqref{E:p.2nd}.
%
\end{proof}

\begin{proposition}
\label{P:p=0}
For any ring cohomology theory with a symplectic Thom structure
a rank $2$ symplectic bundle  $(E,\phi)$ with a nowhere vanishing
section has $b(E,\phi) = 0$.
\end{proposition}

\begin{proof}
The nowhere vanishing section $s \colon Y \to E$ meets the zero section in $\varnothing$.
Since we have $\Th(E,\phi) \in A_{Y}(E)$, we have $s^A(\Th(E,\phi)) \in  A_{\varnothing}(Y) = 0$.
Then formula \eqref{E:p.2nd}
gives $b(E,\phi) = - \overline{e}^{A}s^{A}(\Th(E,\phi))  = 0$.
\end{proof}

\section{The quaternionic projective bundle theorem}

We return to the situation where $(V,\phi)$ is a symplectic space of
dimension $2n+2$, and $\HP^n = \Gr(2,V) \smallsetminus
\GrSp(2,V,\phi)$ is the affine scheme parametrizing $2$-dimensional
subspaces $U \subset V$ such that $\phi \rest U$ is nondegenerate.

\begin{theorem}[Quaternionic projective bundle theorem for trivial bundles]
\label{th.HPn.triv}
Let $A$ be a ring cohomology theory with a symplectic Thom structure.
Let $(\shf U,\phi \rest {\shf U})$ be the tautological rank $2$
symplectic subbundle over $\HP^n$ and $\zeta = b(\shf U, \phi
\rest{\shf U})$ its Borel class.  Then for any scheme $S$
we have $A(\HP^n \times S) = A(S)[\zeta]/(\zeta^{n+1})$.
\end{theorem}

\begin{proof}
It is enough to consider the case $S = k$.

We need to show that $ (1, \zeta, \dots, \zeta^{n}) \colon
A(k)^{\oplus (n+1)} \to A(\HP^n) $ is an isomorphism of two-sided
$A(k)$-modules and that we have $\zeta^{n+1} = 0$.

We go by induction on $n$.  For $n=0$ we have $\HP^0 = k$,
and the tautological rank $2$ symplectic subbundle of the trivial symplectic
bundle of rank $2$ over $k$ is the trivial bundle.  Its Borel class
verifies $\zeta = 0$ by Proposition \ref{P:p=0}.

For $n \geq 1$ we have morphisms smooth schemes by Theorem
\ref{th.basic} or \ref{T:normal.geom}.
\begin{equation}
\label{diag.bdl.th.1}
\vcenter{
\xymatrix @M=5pt @C=50pt {
N^{+} = \overline X_2  \ar[r]^-i_-{\text{closed}}
\ar[d]_-{\text{$\Aff^2$-bundle}}^-q
&
\HP^n \ar[rd]
&
Y = X_0  \ar@{_{(}->}[l]_-{\text{open}}^-j \ar[d]^-f
\\
\HP^{n-1}
&
&
k
}}
\end{equation}
This yields a localization exact sequence and various maps of
two-sided $A(k)$-modules.
\begin{equation}
\label{diag.bdl.th.2}
\vcenter{
\xymatrix @M=5pt @C=40pt {
\cdots \ar[r]^-{\partial \,[ = \, 0]}
& A_{N^{+}}(\HP^n) \ar[r]^-{e^A}
& A(\HP^n) \ar[r]^-{j^A}
\ar[ld]^-{i^A}
& A(Y) \ar[r]^-{\partial \,[ = \, 0]}
& \cdots
\\
A(\HP^{n-1}) \ar[r]^-{q^{A}}_-{\cong}
& A(N^{+})
\ar[u]^-{i_{A,\flat}}_-\cong
&
& A(k) \ar[u]_-{t^A}^-\cong \ar[lu]^-{\cup 1}
}}
\end{equation}
The map $t^A \colon A(k) \to A(Y)$ is an isomorphism by Theorem
\ref{A.X2i} or \ref{T:Ga.cohom}, so $j^A$ is an epimorphism split by $\cup 1$, the
boundary map $\partial$ vanishes, and $(1, e^A) \colon A(k) \oplus
A_{N^{+}}(\HP^n) \arrowiso A(\HP^n)$ is an isomorphism.


The map $q^A \colon A(\HP^{n-1}) \to A(N^{+})$ is an isomorphism
because $q$ is an $\Aff^2$-bundle by Theorem \ref{th.subgrass} or \ref{T:normal.geom}.
The locus $N^{+} \subset \HP^{n}$ is the transversal intersection of the section $s_{+}$ of the
rank $2$ symplectic
bundle $(\shf U,\psi)$ and the zero section.  So it has a symplectic normal bundle
$(\shf U,\psi) \rest{N^{+}}$.  In addition the deformation to the normal bundle isomorphisms
$d(N^{+},\HP^{n})$ exist by Proposition \ref{P:N+1}. So the Thom isomorphisms
$i_{A,\flat} \colon A(N^{+}) \to A_{N^{+}}(\HP^{n})$ are defined.
Writing $\tau = - e^Ai_{A,\flat}q^A$,
we get an isomorphism of
two-sided $A(k)$-modules $(1, \tau) \colon A(k) \oplus A(\HP^{n-1})
\arrowiso A(\HP^n)$.

Write $(\shf V, \overline \phi \rest{ \shf V })$ for the tautological rank $2$
symplectic subbundle on $\HP^{n-1}$ and $\xi =
b(\shf V, \overline \phi \rest{ \shf V })$ for its Borel class.  By
induction we have an isomorphism of two-sided $A(k)$-modules $ (1,
\xi, \dots, \xi^{n-1}) \colon A(k)^{\oplus n} \arrowiso
A(\HP^{n-1}).  $ We therefore have an isomorphism
\[
(1, \tau(1), \tau(\xi), \dots, \tau(\xi^{n-1})) \colon A(k)^{\oplus
  (n+1)} \arrowiso A(\HP^n).
\]
By Theorem \ref{th.subgrass} or \ref{T:normal.geom} the two symplectic bundles $q^*
(\shf V, \overline \phi \rest{ \shf V })$ and $i^*(\shf U, \phi \rest
{\shf U})$ on $N^{+}$ are isomorphic.  By functoriality of the
Borel classes, this gives $q^A \xi = i^A \zeta$ and therefore
also $q^A (\xi^{\ell}) = i^A (\zeta ^{\ell})$.  We also have
$- e^Ai_{A,\flat}i^A(b) = \zeta \cup b$ for $b \in A(\HP^n)$ by Proposition \ref{P:pont.zero}.
This gives us
\[
\tau(\xi^{\ell}) = - e^A i_{A,\flat}q^A(\xi^{\ell}) = - e^A
i_{A,\flat} i^A (\zeta^{\ell}) = \zeta \cup \zeta^{\ell} =
\zeta^{\ell + 1}
\]
for all $\ell$.  This gives the desired isomorphism
\[
(1, \zeta, \zeta^{2}, \dots, \zeta^{n}) \colon A(k)^{\oplus
  (n+1)} \arrowiso A(\HP^n).
\]
Finally, by induction we have $\xi^{n} = 0$, which gives
$\zeta^{n+1} = \tau(\xi^{n}) = 0$.
\end{proof}

If $(E, \phi)$ is a rank $2n$ symplectic bundle over a
scheme $S$, we can define a quaternionic projective bundle
$\HP_S(E, \phi) = \Gr_S(2, E) \smallsetminus \GrSp_S(2,
E,\phi)$.  A Mayer-Vietoris argument gives the general quaternionic
projective bundle theorem.

\begin{theorem}
\label{th.HPn.twist}
\textup{(Quaternionic projective bundle theorem).}
Let $A$ be a ring cohomology theory with a symplectic Thom structure.
Let $(E, \phi)$ be a rank $2n$ symplectic bundle over a scheme
$S$, let $(\shf U,\phi \rest {\shf U})$ be the tautological
rank $2$ symplectic subbundle over the quaternionic projective bundle
$\HP_S(E, \phi)$, and let $\zeta = b(\shf U,\phi \rest{\shf U})$ be its Borel class.
Write $\pi \colon \HP_{S}(E,\phi) \to S$ for the projection.
Then for any closed subset $Z \subset X$ we have an
isomorphism of two-sided $A(S)$-modules
$(1, \zeta, \dots, \zeta^{n-1}) \colon A_{Z}(S)^{\oplus n} \arrowiso
A_{\pi^{-1}(Z)}(\HP_S(E,\phi))$,
and we have
unique classes $b_i(E,\phi) \in A(S)$ for $1 \leq i \leq n$
such that there is a relation
\[
\zeta^{n} - b_1(E,\phi) \cup \zeta^{n-1} + b_2(
E, \phi) \cup \zeta^{n-2} - \cdots +
(-1)^{n} b_{n}(E,\phi) = 0.
\]
If $(E,\phi)$ is trivial, then $b_i(E,\phi) = 0$
for $1 \leq i \leq n$.
\end{theorem}

\begin{definition}
\label{D:pontclass}
The classes $b_{i}(E,\phi)$ ($i=1,\dots,n$) of Theorem \ref{th.HPn.twist} are the \emph{Borel classes}
of $(E,\phi)$ with respect to the symplectic Thom structure of the ring cohomology theory $A$.  For
$i > n$ one sets $b_{i}(E,\phi) = 0$, and one sets $b_{0}(E,\phi) = 1$ and $b_{i}(E,\phi) = 0$ for $i < 0$.
\end{definition}

The Borel classes are universally $A(S)$-central, because they
are the components of the universally $A(S)$-central $\zeta^{n}$ in a two-sided $A(S)$-module
decomposition $A(\HP_S(E,\phi)) \cong A(S)^{\oplus n}$ which is compatible with base
change.

The Borel classes are compatible with base change.  This implies that they are $\Aff^{1}$-deformation invariant in the following sense.

\begin{proposition}
\label{P:deform.invariant}
Let $(E_{0},\phi_{0})$ and $(E_{1},\phi_{1})$ be symplectic bundles on a scheme $S$.
Suppose there exists a symplectic bundle $(E,\phi)$ on
$S \times \Aff^{1}$ with $(E, \phi) \rest{S \times \{0\} } \cong (E_{0},\phi_{0})$ and
 $(E, \phi) \rest{S \times \{1\} } \cong (E_{1},\phi_{1})$.
 Then we have $b_{i}(E_{0},\phi_{0}) = b_{i}(E_{1},\phi_{1})$ for all $i$.
\end{proposition}

A vector bundle $ L$ has an associated \emph{hyperbolic symplectic bundle}
$\Hyp( L) = \left(  L \oplus  L \dual ,
\bigl( \begin{smallmatrix} 0&1\\-1&0 \end{smallmatrix} \bigr) \right)$.

\begin{proposition}
\label{P:subl}
Suppose that $( E, \phi)$ is a symplectic bundle over a scheme
$S$ with a sublagrangian subbundle $L \subset  E$.  Let
$( E_0, \phi_0) = ( L^\perp/ L, \overline \phi) \oplus
\Hyp( L)$.  Then we have $b_i( E,\phi) = b_i( E_0,\phi_0)$
for all $i$.
\end{proposition}

This is because there exists a symplectic bundle $( E_1, \phi_1)$
over $S \times \Aff^1$ with $( E_1, \phi_1) \rest {S \times \{t\}
} \cong ( E, \phi)$ for $t \neq 0$ and $( E_1, \phi_1) \rest{
  S \times \{ 0 \} } \cong ( E_0, \phi_0)$.

%

\begin{theorem}[Nilpotence]
\label{T:nilp.1}
Let $(E,\phi)$ be a symplectic bundle on a scheme $X$.
Then its Borel classes $b_{i}(E,\phi) \in A(X)$ are nilpotent.
\end{theorem}

\begin{proof}
Recall from \S \ref{S:cohomology} that we are assuming that our schemes are quasi-compact.
So we may cover $X$ by $n$
open subsets $U_{\alpha}$ such that $(E,\phi)$ trivializes over each $U_{\alpha}$.
We show that all products
$b_{i_{1}}(E,\phi) \cup \cdots \cup b_{i_{n}}(E,\phi)$ of $n$
Borel classes of $(E,\phi)$ vanish.

For each $U_{\alpha}$ write $Z_{\alpha} = X \smallsetminus U_{\alpha}$.  The restriction
of the Borel class $b_{i_{\alpha}}(E,\phi)$ to $U_{\alpha}$ vanishes because it is a Borel
class of a symplectic bundle trivial on $U_{\alpha}$.  So $b_{i_{\alpha}}(E,\phi)$ is the image of a class
in $A_{Z_{\alpha}}(X)$ under extension of supports.  It follows that
$b_{i_{1}}(E,\phi) \cup \cdots \cup b_{i_{n}}(E,\phi)$ is the image of a class in
$A_{Z_{1}\cap \cdots \cap Z_{n}}(X) = A_{\varnothing}X = 0$ under extension of supports.
\end{proof}

\section{Asymptotic cohomology of quaternionic flag varieties}
\label{S:flag.asymp}

The direct system of trivial symplectic bundles
%
$\Hyp(\shf O) \mono \Hyp(\shf O^{\oplus 2}) \mono \Hyp(\shf O^{\oplus 3}) \mono \cdots$
%
over the base
generates a direct system of quaternionic projective spaces
$
\HP^{0} \xrightarrow{i_{0}} \HP^{1} \xrightarrow{i_{1}} \HP^{2} \to \cdots
$
and an inverse system of cohomology rings
$
\cdots \to A(\HP^{2}) \to A(\HP^{1}) \to A(\HP^{0}).
$
Each $\HP^{n}$ has a rank $2$ universal subbundle $(\shf U_{n}, \phi_{n})$,
and under the inclusion maps we have isomorphisms $i_{n-1}^{*}(\shf U_{n},\phi_{n}) \cong (\shf U_{n-1},\phi_{n-1})$.
Theorem \ref{th.HPn.triv} and the functoriality of Borel classes gives us the following theorem.

\begin{theorem}
\label{T:HP.asymp}
Let $A$ be a ring cohomology theory with a symplectic Thom structure.  Then each map $i_{n-1}^{A} \colon A(\HP^{n}) \to A(\HP^{n-1})$ in the inverse system of cohomology rings is surjective, and
we have an isomorphism $\varprojlim A(\HP^{n}) \cong A(k)[[y]]$ with the indeterminate $y$ corresponding to the element of
$\varprojlim A(\HP^{n})$ given by the system of elements $( b(\shf U_{n},\phi_{n}))_{n \in \NN}$.
\end{theorem}


More generally, let $r \geq 1$.  For any $n \geq r$ we will write
$\HFlag(1^{r};n) = \HFlag({1,\dots,1};n)$ with the $1$ repeated $r$ times.
The system of trivial symplectic bundles also generates a direct system of flag bundles
\[
\HFlag(1^{r};r) \mono \HFlag(1^{r};r+1) \mono \HFlag(1^{r};r+2) \mono \cdots
\]
and an inverse system of cohomology rings
\[
\cdots \to A(\HFlag(1^{r};r+2)) \to A(\HFlag(1^{r};r+1)) \to A(\HFlag(1^{r};r)).
\]
Each $\HFlag(1^{r};n)$ has $r$ rank $2$ universal symplectic subbundles $(\shf U^{(1)}_{n},\phi^{(1)}_{n}), \dots, (\shf U^{(r)}_{n}, \phi^{(r)}_{n})$ plus a rank $2n-2r$ universal symplectic subbundle $(\shf V_{r,n}, \psi_{r,n})$ and a decomposition as an orthogonal direct sum
\begin{equation}
\label{E:univ.bdl}
\Hyp(\OO^{\oplus n}) \cong
(\shf U^{(1)}_{n},\phi^{(1)}_{n}) \perp \cdots \perp (\shf U^{(r)}_{n}, \phi^{(r)}_{n}) \perp (\shf V_{r,n}, \psi_{r,n})
\end{equation}
For the inclusion maps $j_{n} \colon \HFlag(1^{r};n) \to \HFlag(1^{r};n+1)$ we have natural isomorphisms
$j_{n}^{*}(\shf U^{(i)}_{n+1},\phi^{(i)}_{n+1}) \cong (\shf U^{(i)}_{n},\phi^{(i)}_{n})$ for $i = 1, \dots, r$.
The rest of this section will be devoted to proving the following theorem.

\begin{theorem}
\label{T:HFlag.asymp}
Let $A$ be a ring cohomology theory with a symplectic Thom structure.  Let $r \geq 1$ be an integer.  Then the maps $j_{n}^{A}\colon A(\HFlag(1^{r};n+1)) \to A(\HFlag(1^{r};n))$ are surjective, and we have
\begin{equation}
\label{E:HFlag.lim}
\varprojlim_{n \to \infty} A(\HFlag(1^{r};n)) \cong A(k) [[y_{1},\dots, y_{r}]]
\end{equation}
with the indeterminate $y_{i}$ corresponding to the element of $\varprojlim A(\HFlag(1^{r};n))$ given by the system $( b(\shf U^{(i)}_{n},\phi^{(i)}_{n}))_{n \geq r}$.
\end{theorem}

The \emph{Borel polynomial} of a symplectic bundle $(E,\phi)$ of
rank $2r$ is
\[
B_{E,\phi}(t) = t^{r}-b_{1}(E,\phi)t^{r-1} + b_{2}(E,\phi)t^{r-2} - \cdots + (-1)^{r}b_{r}(E,\phi),
\]
while the \emph{total Borel class} is
\[
b_{t}(E,\phi) = 1 + b_{1}(E,\phi)t + \cdots + b_{r}(E,\phi)t^{r}.
\]
We will need the following lemma, which is a weak version of the Cartan sum formula.

\begin{lemma}
\label{L:divides.2}
Suppose the symplectic bundle $(E,\phi)$ is an orthogonal direct summand of the symplectic bundle
$(F,\psi)$ on $X$.  Then the Borel polynomial $B_{E,\phi}(t)$ divides the Borel polynomial
$B_{F,\psi}(t)$ in $A(X)[t]$.
\end{lemma}

\begin{proof}
There is an embedding $i \colon \HP_{X}(E,\phi) \subset \HP_{X}(F,\psi)$
such that the tautological rank $2$ symplectic subbundle $(\shf U_{1},\psi\rest{U_{1}})$ on
$\HP_{X}(F,\psi)$
restricts to the tautological rank $2$ symplectic subbundle $(\shf U_{2},\phi\rest{U_{2}})$ of
$\HP_{X}(E,\phi)$.  Hence $i^{A} \colon A( \HP_{X}(F,\psi) ) \to A( \HP_{X}(E,\phi) )$ sends
$\zeta_{1} = b(\shf U_{1},\psi\rest{U_{1}}) \mapsto \zeta_{2} = b(\shf U_{2},\psi\rest{U_{2}})$.
So we have $B_{F,\psi}(\zeta_{2}) = 0$.  The division of $B_{F,\psi}(t)$ by the
monic polynomial $B_{E,\phi}(t)$
yields a remainder $R(t) \in A(\HP_{X}(E,\phi))[t]$ of degree at most $\frac 12 \rk E -1$ such that
$R(\zeta_{2}) = 0$.  So the remainder vanishes.
\end{proof}

\begin{proof}[Proof of Theorem \ref{T:HFlag.asymp}]

To prove the theorem we will calculate the rings $A(\HFlag(1^{r};n))$.
Each of the
relative quaternionic flag bundles is an iterated quaternionic projective bundle, so its cohomology
ring is of the form
\begin{equation}
\label{E:I.rn.1}
A(\HFlag(1^{r};n)) = A(k)[y_{1},\dots,y_{r}]/I_{r,n}
\end{equation}
The maps $A(\HFlag(1^{r};n+1) \to A(\HFlag(1^{r};n))$ are surjective because they are maps
of $A(k)$-algebras which send a set of generators onto a set of generators.

The construction of $\HFlag(1^{r};n)$ as an iterated flag bundle gives us
\begin{equation}
\label{E:I.rn}
I_{r,n} = (B_{1}(y_{1}),B_{2}(y_{1},y_{2}),\dots,B_{r}(y_{1},\dots,y_{r}))
\end{equation}
where each $B_{i}(y_{1},\dots,y_{i})$ is the Borel polynomial in the last variable $y_{i}$
of the symplectic bundle
\begin{equation}
\label{E:G.ni}
(G_{n,i},\gamma_{n,i}) =
(U_{n}^{(i)},\phi_{n}^{(i)}) \perp \cdots \perp (U_{n}^{(r)},\phi_{n}^{(r)}) \perp (\shf V_{r,n}, \psi_{r,n}).
\end{equation}
This bundle is an orthogonal direct summand of the $\Hyp(\OO^{\oplus n})$ of \eqref{E:univ.bdl},
so by Lemma \ref{L:divides.2} the polynomial $P_{i}(y_{1},\dots,y_{i})$ divides
$B_{\Hyp(\OO^{\oplus n})}(y_{i}) = y_{i}^{n}$ in
$A(k)[y_{1},\dots,y_{i}]/(B_{1},\dots,B_{i-1})$.  From this we deduce
\begin{equation}
\label{E:inclusion.1}
(y_{1}^{n},y_{2}^{n},\dots,y_{r}^{n}) \subset I_{r,n}.
\end{equation}

Each $(G_{n,i},\gamma_{n,i})$ is an orthogonal direct summand of $(G_{n,i-1},\gamma_{n,i})$,
so $B_{G_{n,i},\gamma_{n,i}}(t)$ divides $B_{G_{n,i-1},\gamma_{n,i-1}}(t)$.
Since the difference in the degrees is $1$, we have
\begin{align}
\label{E:z.i}
B_{G_{n,i-1},\gamma_{n,i-1}}(t) & = (t-z_{i})B_{G_{n,i},\gamma_{n,i}}(t),
&
z_{i} & = b_{1}(G_{n,i-1},\gamma_{n,i-1}) - b_{1}(G_{n,i},\gamma_{n,i})
\end{align}
%
We deduce
$b_{t}(G_{n,i-1},\gamma_{n,i-1}) = (1+z_{i}t)b_{t}(G_{n,i},\gamma_{n,i})$.
Since $b_{t}(G_{n,1},\gamma_{n,1}) = b_{t}(\Hyp(\OO^{\oplus n})) = 1$ we get
\[
b_{t}(G_{n,i},\gamma_{n,i}) = \frac{1}{\prod_{m = 1}^{i-1}(1+z_{m}t)}
\]
in $A(k)/(B_{1},\dots,B_{i-1})[t]$.

Write $\coeff(t^{i},f(t))$ for the coefficient of $t^{i}$ in the power series or polynomial $f(t)$.
The Borel polynomial and total Borel class are related by the formula
\[
B_{E,\phi}(y)
= \coeff\left(t^{r+1}, \frac{b_{-t}(E,\phi)}{1-yt}\right)
\]
where $b_{-t}(E,\phi)$ means that one substitutes $-t$ for $t$ in the series
$b_{t}(E,\phi)$.
Hence we have
\[
B_{i}(y_{1},\dots,y_{i}) =
 \coeff\left( t^{n-i+1}, \frac{1}{\prod_{m = 1}^{i-1}(1-z_{m}t) \cdot (1-y_{i}t)} \right).
\]

Let
$h_{i}(u_{1},\dots,u_{s})$ be the $i^{\text{th}}$ \emph{complete symmetric polynomial},
the sum of all the monomials
in $u_{1},\dots,u_{s}$ of degree $i$.  Set $h_{0}=1$.   Their generating function is
\[
H(t) = \sum_{i=0}^{\infty}h_{i}(u_{1},\dots,u_{s})t^{i} =
\frac{1}{\prod_{m=1}^{s} (1-u_{m}t)}.
\]
Thus we have
\begin{equation}
\label{E:complete}
B_{i}(y_{1},\dots,y_{i}) = h_{n-i+1}(z_{1},\dots,z_{i-1},y_{i}).
\end{equation}

We claim that the $z_{i}$ lie in the ideal $(y_{1},\dots,y_{r}) \subset A(k)[y_{1},\dots,y_{r}]$.
That is because by the universal property of $\HFlag(1^{r};n)$,
fixing a decomposition $\Hyp(\OO^{\oplus n}) = \Hyp(\OO)^{\perp r}
\perp \Hyp(\OO^{\oplus n-r})$ gives a section $s \colon k \to \HFlag(1^{r};n)$
of the structural map.  Since we have $s^{*}(\shf U_{n}^{(i)},\phi_{n}^{(i)}) = \Hyp(\OO)$,
we have $s^{A}y_{i} = b(\Hyp(\OO)) = 0$.  So we have $(y_{1},\dots,y_{r}) = \ker s^{A}$.
However, the $s^{*}(G_{n,i},\gamma_{n,i})$ are also trivial, so from \eqref{E:z.i} we have
$s^{A}z_{i}=0$, proving the claim.

It now follows from \eqref{E:complete} that we have
$B_{i}(y_{1},\dots,y_{i}) \in (y_{1},\dots,y_{i})^{n-i+1}$.  This gives us
\begin{equation}
\label{E:inclusion.2}
I_{n,r} \subset (y_{1},\dots,y_{n})^{n-r+1}
\end{equation}
The inclusions \eqref{E:inclusion.1} and \eqref{E:inclusion.2} together give
$\varprojlim A(\HFlag(1^{r};n)) = A(k)[[y_{1},\dots,y_{r}]]$.
\end{proof}

\section{The splitting principle and the sum formula}

The quaternionic flag varieties we studied in the previous section are iterated quaternionic projective
bundles over quaternionic Grassmannians
\begin{equation}
\label{E:Flag.Gr}
\HGr(r,n) \leftarrow
\HFlag(1,r-1;n) \leftarrow
\cdots \leftarrow
\HFlag(1^{r-2},2;n) \leftarrow
\HFlag(1^{r};n).
\end{equation}
The quaternionic projective bundle theorem \ref{th.HPn.twist} implies that
$A(\HFlag(1^{r};n))$ is a free module over $A(\HGr(r,n))$ for any ring cohomology theory $A$
with a symplectic Thom structure.  More precisely,
when one pulls back the universal rank $2r$ symplectic bundle $(\shf U_{n},\phi_{n})$ on $\HGr(r,n)$
to $\HFlag(1^{r};n)$, it splits into the orthogonal direct sum of the $r$ rank $2$ universal symplectic
subbundles
\begin{equation}
\label{E:r.bdl}
(\shf U^{(1)}_{n},\phi^{(1)}_{n}) \perp \cdots \perp (\shf U^{(r)}_{n}, \phi^{(r)}_{n})
\end{equation}
on $\HFlag(1^{r};n)$.  For $1 \leq i \leq r-1$ the
bundle $(\shf U^{(i)}_{n}, \phi^{(i)}_{n})$ is the pullback to $\HFlag(1^{r};n)$ of the tautological  rank $2$
symplectic subbundles of the Quaternionic Projective Bundle Theorem for the $i$-th projective bundle  of
\eqref{E:Flag.Gr}, which is an $\HP^{r-i}$-bundle.
Writing $\overline{y}_{i} = b(\shf U^{(i)}_{n}, \phi^{(i)}_{n})$ we see the following.

\begin{proposition}
\label{P:Flag.Gr}
Let $A$ be a ring cohomology theorem with a symplectic Thom structure.
The projection $q \colon \HFlag(1^{r};n) \to \HGr(r,n)$ induces a monomorphism
$q^{A} \colon A(\HGr(r,n)) \to A(\HFlag(1^{r};n))$ under which
$A(\HFlag(1^{r};n))$ is a free $A(\HGr(r,n))$-module of rank $r!$ with basis
\begin{equation}
\label{E:base}
\mathbb B_{r} =
\{ \overline{y}_{1}^{a_{1}}\overline{y}_{2}^{a_{2}} \cdots \overline{y}_{r-1}^{a_{r-1}} \mid \text{$0 \leq a_{i} \leq r-i$ for all $i$} \}.
\end{equation}
\end{proposition}

Let $(E, \phi)$ be a symplectic bundle of rank $2r$ on a scheme $X$,
and let $q \colon \HFlag_{X}(E,\phi) \to X$ be the associated complete quaternionic flag bundle.  The pullback of $E$  splits as the orthogonal direct sum of the $r$ universal symplectic subbundles of rank $2$
\begin{equation}
\label{E:roots}
q^{*}(E, \phi) \cong (\shf U_{1},\psi_{1}) \perp (\shf U_{2}, \psi_{2}) \perp \cdots \perp (\shf U_{r},\psi_{r}).
\end{equation}
Write $u_{i} = b(\shf U_{i},\psi_{i}) \in A(\HFlag_{X}(E,\phi))$.  We will call the $u_{i}$ the \emph{Borel roots} of $(E,\phi)$.

\begin{theorem}[Symplectic splitting principle]
\label{T:splitting}
The map $q^{A} \colon A(X) \to A(\HFlag_{X}(E,\phi))$ is injective and makes $A(\HFlag_{X}(E,\phi))$ into a free two-sided $A(X)$-module of rank $r!$ with basis the $\mathbb B_{r}$ of \eqref{E:base}.  Moreover, the Borel classes $b_{i}(E,\phi) \in A(X)$ coincide after pullback by $q^{A}$ with the
elementary symmetric polynomials $e_{i}(u_{1},\dots,u_{r})$ in the Borel roots.
\end{theorem}

The proof uses two lemmas.

\begin{lemma}
\label{L:NZD}
Let $R$ be a ring, and let $a_{1},\dots,a_{n} \in R$ be central elements such that
$a_{i}-a_{j}$ is not a zero divisor for all $i \neq j$.  If the polynomials $t-a_{1}$, $\dots$, $t-a_{n}$
all divide $h \in R[t]$, then $\prod_{i=1}^{n}(t-a_{i})$ also divides $h$.
\end{lemma}

\begin{proof}
By induction we may assume that $h = g \prod_{i=1}^{n-1}(t-a_{i})$.  We then have
$g(a_{n}) \prod
(a_{n}-a_{i}) = 0$.  Since
the $a_{n}-a_{i}$ are not zero divisors, we have $g(a_{n}) = 0$, and so $t-a_{n}$ divides $g$.
\end{proof}

\begin{lemma}
\label{L:affine}
Let $(E,\phi)$ be a symplectic bundle on an affine scheme $X = \Spec R$.
Suppose that $E$ can be generated by $n$ global sections.  Then
we can embed
$(E,\phi)$ as a symplectic subbundle of the trivial symplectic bundle of rank $2n$.
\end{lemma}

This lemma is well known (see for example \cite{Knus:1991qr}).   Hypotheses
like $\frac 12 \in R$ are not necessary for alternating forms.

\begin{proof}[Proof of Theorem \ref{T:splitting}]
The first sentence of the theorem is simply a relative version of Proposition \ref{P:Flag.Gr}.
It remains only to prove the second sentence.

We first treat the special case of the tautological rank $2r$ symplectic subbundle $(\shf U_{r,n},\phi_{r,n})$
on $\HGr(r,n)$.
Let  $b_{i} \in \varprojlim A(\HGr(r,n))$ be the element
corresponding to the inverse system of Borel classes
$\bigl( b_{i}(\shf U_{r,n},\phi_{r,n}) \bigr)_{n \geq r}$.
Recall the $y_{i} \in \varprojlim A(\HFlag(1^{r};n))$ of Theorem \ref{T:HFlag.asymp}
given by the system $(b(\shf U^{(i)}_{n},\phi^{(i)}_{n}))_{n \geq r}$.
Since $(\shf U_{n}^{(i)},\phi_{n}^{(i)})$ is an orthogonal direct summand of
the pullback to $\HFlag(1^{r};n)$ of the tautological bundle
$(\shf U_{r,n},\phi_{r,n})$ of $\HGr(r,n)$,
the Borel polynomial
$t-b(\shf U_{n}^{(i)},\phi_{n}^{(i)})$
divides the Borel polynomial
\[
B_{\shf U_{r,n},\phi_{r,n}}(t) = t^{r}-b_{1}(\shf U_{r,n}, \phi_{r,n}) t^{r-1} + \cdots +
(-1)^{r} b_{r}(\shf U_{r,n}, \phi_{r,n}).
\]
The quotient polynomial of degree $r-1$ at level $n$ restricts to the quotient at level $n-1$
because the quotients and remainders for division by monic polynomials
with central coefficients are unique.  So
the quotient polynomials also form an inverse system, and $t-y_{i}$ divides
$B(t) = t^{r}-b_{1}t^{r-1}+ \cdots + (-1)^{r}b_{r}$ in the inverse limit $A(k)[[y_{1},\dots,y_{r}]]$.
Lemma \ref{L:NZD}
applies, and we get $\prod_{i=1}^{r}(t-y_{i}) = B(t)$.  Hence the $b_{i}$ are the elementary
symmetric polynomials in the $y_{i}$.
It follows that the $b_{i}(\shf U_{r,n}, \phi_{r,n})$ are the elementary symmetric polynomials
in the $u_{i} = b(\shf U^{(i)}_{n},\phi^{(i)}_{n})$.

We next treat the case where $X$ is affine.  By Lemma \ref{L:affine} $(E,\phi)$ can be embedded
as an orthogonal direct summand of some trivial symplectic bundle.  This is classified by
a map $f \colon X \to \HGr(r,n)$.
Since $f^{A}$ and the map for the corresponding quaternionic flag bundles
pull back the Borel classes and roots of $(\shf U_{r,n},\phi_{r,n})$
to those of $(E,\phi)$, the Borel classes of $(E,\phi)$ are also the elementary symmetric
polynomials in the Borel roots.

Finally suppose $X$ is general.  Recall from \S\ref{S:cohomology} that we are assuming
that our schemes are quasi-compact with an ample family of line bundles. Therefore
there is an affine bundle $g \colon Y \to X$ with $Y$ an affine scheme.  The Borel classes of
$g^{*}(E,\phi)$ are the elementary symmetric polynomials in its Borel roots.
Then $g^{A}$ and the induced map on the cohomology of the quaternionic flag bundles are isomorphisms
and send the Borel classes and roots of $(E,\phi)$ to those of $g^{*}(E,\phi)$.
So the Borel classes of $(E,\phi)$ are the elementary symmetric polynomials in its Borel roots.
\end{proof}

\begin{theorem}[Cartan sum formula]
\label{T:Cartan}
Suppose $(F, \psi) \cong (E_{1},\phi_{1}) \perp (E_{2},\phi_{2})$ is an orthogonal direct sum of symplectic bundles over a scheme $X$.  Then for all $i$ we have
\begin{align}
\label{E:sum.formula}
b_{t}(F,\psi) & = b_{t}(E_{1},\phi_{1})b_{t}(E_{2},\phi_{2}), \\
\label{E:sum.2nd}
b_{i}(F,\psi) & = b_{i}(E_{1},\phi_{1}) + \sum_{j=1}^{i-1} b_{i-j}(E_{1},\phi_{1}) b_{j}(E_{2},\phi_{2}) +
b_{i}(E_{2},\phi_{2}).
\end{align}
\end{theorem}

The first Borel class is additive, and the top Borel class is multiplicative.

\begin{proof}
Consider the fiber bundles
\begin{align*}
p \colon \HFlag_{X}(E_{1},\phi_{1}) \times_{X} \HFlag_{X}(E_{2},\phi_{2}) \to X, &&
q\colon \HFlag_{X}(F,\psi) \to X
\end{align*}
We have orthogonal direct sum decompositions
\begin{align}
\label{E:Cartan.decomp}
p^{*}(E_{1},\phi_{1}) & \cong \bigperp _{i=1}^{r} (\shf U_{i},\phi_{i}), &
p^{*}(E_{2},\phi_{2}) & \cong \bigperp _{j=1}^{s} (\shf U'_{j},\phi'_{j}), &
q^{*}(F,\psi) & \cong \bigperp _{\ell=1}^{r+s} (\shf V_{\ell},\psi_{\ell}),
\end{align}
By Theorem \ref{T:splitting} over $\HFlag(E_{1},\phi_{1}) \times_{X} \HFlag(E_{2},\phi_{2})$ we have
\begin{equation}
\label{E:Cartan.two}
b_{t}(E_{1},\phi_{1})  b_{t}(E_{2},\phi_{2}) =
\prod _{i=1}^{r} \bigl( 1+b(\shf U_{i}, \phi_{i}) t \bigr)
\prod _{j=1}^{s} \bigl( 1+b(\shf U'_{j}, \phi'_{j}) t \bigr)
\end{equation}
while over $\HFlag(F,\psi)$ we have
$b_{t}(F,\psi) =
\prod _{\ell=1}^{r+s} \bigl( 1+b(\shf V_{\ell}, \psi_{\ell}) t \bigr)$.
We also have the orthogonal direct sum
$p^{*}(F,\psi) =
\tbigperp _{i=1}^{r}
p^{*}
(\shf U_{i},\phi_{i})
\perp
\tbigperp _{j=1}^{s}
p^{*}
(\shf U'_{j},\phi'_{j})$.
This decomposition is classified by a unique map
\[
f \colon \HFlag(E_{1},\phi_{1}) \times_{X} \HFlag(E_{2},\phi_{2}) \to \HFlag_{X}(F,\psi)
\]
such that
\[
f^{*}(\shf V_{\ell},\psi_{\ell}) \cong
\begin{cases}
p^{*}
(\shf U_{\ell},\phi_{\ell}) & \text{for $\ell=1,\dots,r$}, \\
p^{*}
(\shf U'_{\ell-r},\phi'_{\ell-r}) & \text{for $\ell = r+1, \dots, r+s$}.
\end{cases}
\]
It follows that when we pull $b_{t}(F,\psi)$ back along $f$, we get
\[
b_{t}(F,\psi) =
\prod _{i=1}^{r} \bigl( 1+b(\shf U_{i}, \phi_{i}) t \bigr)
\prod _{j=1}^{s} \bigl( 1+b(\shf U'_{j}, \phi'_{j}) t \bigr)
= b_{t}(E_{1},\phi_{1}) b_{t}(E_{2},\phi_{2}).
\]
Equating the terms of degree $i$ in this equality of series gives formula \eqref{E:sum.2nd}.
\end{proof}

The Cartan sum formula, Proposition \ref{P:subl}, and nilpotence combine to show that the total Borel class may be defined for Grothendieck-Witt classes of symplectic bundles, giving maps
\begin{equation}
\label{E:GW}
b_{t} \colon \GW^{-}(X) \to A(X)[t]^{\times},
\end{equation}
functorial in $X$, and sending sums to products.  Note that since symplectic bundles have only finitely many nonzero Borel classes, and they are all nilpotent, the same holds for virtual differences of bundles as well.  Hence the image is in $A(X)[t]^{\times}$.  Actually the morphism sends
\[
b_{t} \colon \widetilde{\GW}^{-}(X) = \frac{\GW^{-}(X)}{\ZZ [\Hyp(\shf O)]} \to A(X)[t]^{\times},
\]
and in particular, the first Borel class is an additive map
$b_{1} \colon \widetilde {\GW}^{-}(X) \to A(X)$.

\section{Cohomology of quaternionic Grassmannians}

We recall some facts about symmetric polynomials.
They may mostly be found in Macdonald's book \cite[Chap.\ 1, \S\S 1--3]{Macdonald:1995zl}.

Let $\Lambda_{r} \subset \ZZ[y_{1},\dots,y_{r}]$ be the ring of symmetric polynomials in $r$ variables.
Let $e_{i}$ denote the $i$-th elementary symmetric polynomial, and
$h_{i}$ the $i$-th \emph{complete symmetric polynomial}, the sum of all the monomials of degree $i$.
Set $e_{0} = h_{0} = 1$.
We have $\Lambda_{r} = \ZZ[e_{1},\dots,e_{r}]$.
The generating functions are
\begin{align}
\label{E:EH}
E(t) & = \sum_{i \geq 0} e_{i}t^{i} = \prod_{j=1}^{r}(1+y_{i}t), &
H(t) & = \sum_{i \geq 0} h_{i}t^{i} = \prod_{j=1}^{r}(1-y_{i}t)^{-1}.
\end{align}
So we have
\begin{align}
\label{E:H.recur}
E(t)H(-t) & = 1, &
h_{m} + \sum_{i=1}^{r} (-1)^{r} e_{i}h_{m-i} & = 0.
\end{align}
So we also have $\Lambda_{r} = \ZZ[h_{1},\dots,h_{r}]$.
The $h_{i}$ with $i > r$ are nonzero but are dependent on $h_{1},\dots,h_{r}$.

Let $\lambda = (\lambda_{1},\lambda_{2},\dots,\lambda_{r})$ be a partition of length $l(\lambda) \leq r$.
Write $\delta = (r-1,r-2, \dots,1,0)$ and
$a_{\lambda + \delta} = \det(y_{i}^{\lambda_{j}+r-j})_{1\leq i,j \leq r}$.  Then $a_{\lambda+\delta}$ is
a skew-symmetric polynomial and therefore divisible by the Vandermonde determinant $a_{\delta}$.
The quotient $s_{\lambda} = a_{\lambda+\delta}/a_{\delta}$ is
the \emph{Schur polynomial} for $\lambda$.
It is symmetric of degree $\lvert \lambda \rvert = \sum \lambda_{i}$.
One has $s_{(1^{i})} = e_{i}$ and $s_{(i)} = h_{i}$.
The $a_{\lambda+\delta}$ with $l(\lambda) \leq r$ form a $\ZZ$-basis of the skew-symmetric polynomials
in $r$ variables, so the $s_{\lambda}$ with $l(\lambda) \leq r$ form a $\ZZ$-basis of $\Lambda_{r}$.
Denote by $\lambda'$ the partition dual to $\lambda$.  One has formulas
\begin{equation}
\label{E:schur.det}
s_{\lambda} = \det(e_{\lambda'_{i}-i+j})_{1 \leq i,j \leq m}
= \det(h_{\lambda_{i}-i+j})_{1\leq i,j, \leq r},
\end{equation}
for $m \geq l(\lambda')$ and $r \geq l(\lambda)$.

The ring of \emph{symmetric functions} $\Lambda$ in an infinite number of variables is the
inverse limit of the $\Lambda_{r}$ in the
category of graded rings.  It is a polynomial ring in infinitely many indeterminates
\[
\Lambda = \ZZ[e_{1},e_{2}, \dots] = \ZZ[h_{1},h_{2}, \dots]
\]
It has an involution $\omega$ with $\omega(e_{i}) = h_{i}$ and $\omega(h_{i}) = e_{i}$ for all $i$
and $\omega(s_{\lambda}) = s_{\lambda'}$ for all partitions $\lambda$.
The quotient map $\Lambda \onto \Lambda_{r}$ is the quotient
by the ideal $(e_{r+1},e_{r+2}, \dots)$.  The ring $\Lambda$
has as $\ZZ$-basis the Schur functions
$s_{\lambda}$ with $\lambda$ ranging over all partitions.  If, however,
$l(\lambda) = \lambda'_{1} > r$, then the first row of the first determinant of \eqref{E:schur.det} is
$(e_{\lambda'_{1}}, e_{\lambda'_{1}}+1, \dots, e_{\lambda'_{1}}+m-1)$.  All these entries
are sent to $0$ in $\Lambda_{r}$.
Set
\[
\Pi_{r,n-r} = \{ \text{partitions $\lambda$ with length $l(\lambda) = \lambda'_{1} \leq r$
and with $\lambda_{1} \leq n-r$} \}
\]
The set $\Pi_{r,n-r}$ has $\binom{n}{r}$ members.
For the ideal $I_{n-r} = (h_{n-r+1},h_{n-r+2},\dots,h_{n}) \subset  \Lambda_{r}$
the quotient map $\Lambda_{r} \onto \Lambda_{r}/I_{n-r}$ sends
$s_{\lambda} \mapsto 0$ for all $\lambda \not\in \Pi_{r,n-r}$.
The quotient $\Lambda_{r}/I_{n-r}$ is free over $\ZZ$
with basis $\{ s_{\lambda} \mid \lambda \in \Pi_{r,n-r} \}$.

Now suppose $(E,\phi)$ is a symplectic bundle of rank $2r$ over $S$.
Let $q \colon \HFlag_{S}(E,\phi) \to S$ be the associated complete quaternionic flag bundle, let
\[
q^{*}(E, \phi) \cong
(\shf U_{1},\psi_{1}) \perp (\shf U_{2}, \psi_{2}) \perp \cdots \perp (\shf U_{r},\psi_{r}).
\]
be the splitting of the pullback of $E$ as the orthogonal direct sum of the $r$ universal
symplectic subbundles of rank $2$ on $\HFlag_{S}(E,\phi)$, and let
$y_{i} = b(\shf U_{i},\psi_{i}) \in A(\HFlag_{S}(E,\phi))$ be the Borel roots of $(E,\phi)$.
For any partition $\lambda$ of length $l(\lambda) \leq r$ write
\begin{equation}
\label{E:schur.p}
s_{\lambda}(E,\phi) = s_{\lambda}(y_{1},\dots,y_{r})
= \det(b_{\lambda'_{i}-i+j}(E,\phi))_{1 \leq i,j \leq m} \in A(S)
\end{equation}
with $m \geq l(\lambda')$.

We now complete the calculation of the cohomology of  $\HFlag(1^{r};n)$
begun in \S\ref{S:flag.asymp}.

\begin{theorem}
\label{T:flag}
On $\HFlag(1^{r};n)$ let $y_{1},\dots,y_{r}$ be the first Borel classes of the $r$
tautological rank $2$ subbundles.  Then  $A(\HFlag(1^{r};n)) = A(k)[y_{1},\dots,y_{r}]/I_{r,n}$
with
\begin{equation}
\label{E:hn.1}
I_{r,n} = (h_{n}(y_{1}),h_{n-1}(y_{1},y_{2}),\dots,h_{n-r+1}(y_{1},\dots,y_{r})),
\end{equation}
and if we now write $h_{i}$ for the complete symmetric polynomial in all $r$ variables
$y_{1},\dots,y_{r}$, then we also have
\begin{equation}
\label{E:hn.2}
A(\HFlag(1^{r};n)) = A(k)[y_{1},\dots,y_{r}]/
(h_{n},h_{n-1}, \dots, h_{n-r+1}).
\end{equation}
\end{theorem}

\begin{proof}
Using the Cartan sum formula, we see that the classes $z_{i}$ of \eqref{E:z.i} are actually
$z_{i} = y_{i}$.   Hence \eqref{E:complete} can be rewritten as
$P_{i}(y_{1},\dots,y_{i}) = h_{n-i+1}(y_{1},\dots,y_{i})$, and we get \eqref{E:hn.1}.

Since the complete symmetric polynomials are the sums of all monomials of a given degree, we
have
\[
h_{m}(y_{1},\dots,y_{i-1},y_{i}) = h_{m}(y_{1},\dots,y_{i-1}) + y_{i}h_{m-1}(y_{1},\dots,y_{i-1},y_{i}).
\]
This formula can be used to show that one can add one variable to each of the generators of
$I_{r,n}$ which have less than $r$ variables without changing the ideal.  One repeats this until
all the generators have $r$ variables.  This gives \eqref{E:hn.2}.
\end{proof}

\begin{theorem}
\label{T:Grass}
Let $(\shf U,\phi)$ be the tautological symplectic bundle of rank $2r$ on $\HGr(r,n)$, and let
$(\shf U^{\perp},\psi)$ be its orthogonal complement.  Then for any ring
cohomology theory $A$ with a symplectic Thom structure and any $S$ the map
\begin{equation}
\label{E:HGr.cohom.1}
A(S)[e_{1},\dots,e_{r}] / (h_{n-r+1},\dots,h_{n}) \xrightarrow{\cong} A(\HGr(r,n) \times S)
\end{equation}
sending $e_{i} \mapsto b_{i}(\shf U,\phi)$ for all $i$
is an isomorphism of rings, the map
\begin{equation}
\label{E:HGr.cohom.2}
A(S)^{\oplus \binom{n}{r}} \xrightarrow{(s_{\lambda}(\shf U,\phi))_{\lambda\in\Pi_{r,n-r}}}
A(\HGr(r,n) \times S)
\end{equation}
is an isomorphism of two-sided $A(S)$-modules, and
for all partitions $\lambda$ we have
\begin{equation}
\label{E:schur.complement}
s_{\lambda}(\shf U,\phi) = (-1)^{\lvert \lambda \rvert} s_{\lambda'}(\shf U^{\perp},\psi).
\end{equation}
\end{theorem}

\begin{proof}
The proof is much like the case of ordinary Grassmannians.
It is enough to consider $S = k$.
Write $A_{r,n} = A(k)[e_{1},\dots,e_{r}]/(h_{n-r+1},\dots,h_{n})$.
Since $A(\HGr(r,n)) \to A(\HFlag(1^{r};n))$ is injective, the complete symmetric polynomials
$h_{i}$ with $i > n-r$ vanish in $A(\HGr(r,n))$ as well.  We thus get the map
$\gamma \colon A_{r,n} \to A(\HGr(r,n))$ sending $e_{i} \mapsto b_{i}(\shf U,\phi)$.
By Proposition \ref{P:Flag.Gr}
\[
\mathbb B_{r} =
\{{y}_{1}^{a_{1}}{y}_{2}^{a_{2}} \cdots{y}_{r-1}^{a_{r-1}} \mid \text{$0 \leq a_{i} \leq r-i$ for all $i$} \}.
\]
is a basis of $A(\HFlag(1^{r};n))$ as a two-sided $A(\HGr(r,n))$-module, and it is also
a basis of $\ZZ[y_{1},\dots,y_{r}]$ as a free module over the ring of
symmetric polynomials $\ZZ[e_{1},\dots,e_{r}]$ and therefore also a basis
of $A(k)[y_{1},\dots,y_{r}]/(h_{n-r+1},\dots,h_{n})$ as a two-sided free
$A_{r,n}$-module.  By Theorem \ref{T:flag}
the map of free modules $A(k)[y_{1},\dots,y_{r}]/(h_{n-r+1},\dots,h_{n}) \to A(\HFlag(1^{r};n))$ is
an isomorphism.  So $\gamma \colon A_{r,n} \to A(\HGr(r,n))$
is also an isomorphism.
%

Over the quaternionic Grassmannian $\HGr(r,n)$ the orthogonal direct sum
$(\shf U,\phi) \perp (\shf U^{\perp},\psi)$ is the trivial symplectic bundle $\Hyp(\OO)^{\oplus n}$ and
has vanishing Borel classes.  So the Cartan sum formula \eqref{E:sum.formula} gives
$b_{t}(\shf U,\phi) b_{t}(\shf U^{\perp},\psi) = 1$.
For the generating series of \eqref{E:EH}
we have $\gamma(E(t)) = b_{t}(\shf U,\phi)$.
So from the identity $E(t)H(-t) = 1$ of \eqref{E:H.recur}, we see we have
$\gamma(H(-t)) = b_{t}(\shf U^{\perp},\psi)$ and thus
$(-1)^{i} \gamma(h_{i}) =  b_{i}(\shf U^{\perp},\psi)$.
The formula $s_{\lambda}(\shf U,\phi) = (-1)^{\lvert \lambda \rvert} s_{\lambda'}(\shf U^{\perp},\psi)$
now follows from \eqref{E:schur.det}.
The sign change $h_{i} \mapsto (-1)^{i}h_{i}$ comes from the
involution of $\Lambda_{r}$ sending $f \mapsto (-1)^{\deg(f)}f$ for all homogeneous symmetric
polynomials.  Hence the sign $(-1)^{\lvert \lambda \rvert}$ in front of
$s_{\lambda'}(\shf U^{\perp},\psi)$.
\end{proof}

The usual Mayer-Vietoris argument now gives the following generalization of Theorem \ref{T:Grass}.

\begin{theorem}
\label{T:Grass.bdl}
Let $(E,\phi)$ be a symplectic bundle of rank $2n$ over $S$.  Let $(\shf U,\phi\rest{\shf U})$ be
the tautological subbundle of rank $2r$ on $\HGr_{S}(r,E,\phi)$, and let
$(\shf U^{\perp}, \phi\rest{U^{\perp}})$ be its orthogonal complement.  Then for any
ring cohomology theory with a symplectic Thom structure
\begin{gather*}
A(S)^{\oplus \binom{n}{r}}
\xrightarrow[\cong]{(s_{\lambda}(\shf U,\phi\rest{\shf U}))_{\lambda \in \Pi_{r,n-r}}}
A(\HGr_{S}(r,E,\phi)), \\
A(S) ^{\oplus \binom{n}{r}}
\xrightarrow[\cong]{(s_{\lambda}(\shf U^{\perp},\phi\rest{\shf U^{\perp}}))_{\lambda \in \Pi_{n-r,r}}}
A(\HGr_{S}(r,E,\phi))
\end{gather*}
are isomorphisms of two-sided $A(S)$-modules.
\end{theorem}

Let $(\shf U_{r,n},\phi_{r,n}) \into \Hyp(\OO)^{\oplus n}$ be the universal tautological symplectic
bundle of rank $2r$ on $\HGr(r,n)$.  Let $\eta_{r,n} \colon \shf U_{r,n} \subset \Hyp(\OO)^{\oplus n}$
be the inclusion.  Consider the inclusions of symplectic subbundles
\begin{equation}
\label{E:inclusions}
\vcenter{
\xymatrix @M=5pt @C=40pt @R=10pt {
\shf U_{r,n}  \ar@{^{(}->}[r]^-{
\left(
\begin{smallmatrix}
\eta_{r,n} \\0
\end{smallmatrix}
\right)
} &
\Hyp(\OO)^{\oplus n} \oplus \Hyp(\OO) = \Hyp(\OO)^{\oplus n+1},
\\
\Hyp(\OO) \oplus \shf U_{r,n} \ar@{^{(}->}[r]^-{
\left(
\begin{smallmatrix}
1 & 0\\ 0 & \eta_{r,n}
\end{smallmatrix}
\right)
} &
\Hyp(\OO) \oplus \Hyp(\OO)^{\oplus n} = \Hyp(\OO)^{\oplus n+1}.
}}
\end{equation}
They are classified by maps $\alpha_{r,n} \colon \HGr(r,n) \to \HGr(r,n+1)$ and
$\beta_{r,n} \colon \HGr(r,n) \to \HGr(r+1,n+1)$ respectively.  Let
$\gamma_{r,n} = \beta_{r,n+1}\alpha_{r,n} = \alpha_{r+1,n+1}\beta_{r,n} \colon
\HGr(r,n) \to \HGr(r+1,n+2)$.  We have direct systems of quaternionic Grassmannians
\begin{equation}
\label{E:systems}
\begin{gathered}
\HGr(r,r) \xrightarrow{\alpha_{r,r}} \HGr(r,r+1) \xrightarrow{\alpha_{r,r+1}}
\HGr(r,r+2) \to \cdot \to \HGr(r,n) \xrightarrow{\alpha_{r,n}} \cdots
\\
\HGr(0,0) \xrightarrow{\gamma_{0,0}} \HGr(1,2) \xrightarrow{\gamma_{1,2}} \HGr(2,4) \to \cdots \to \HGr(n,2n) \xrightarrow{\gamma_{n,2n}} \cdots
\end{gathered}
\end{equation}

\begin{theorem}
\label{T:HGr.lim.cohom}

For any ring cohomology theory $A$ with a symplectic Thom structure
and for any $S$ the maps
\begin{gather*}
(\alpha_{r,n} \times 1_{S})^{A} \colon A(\HGr(r,n+1) \times S) \to A(\HGr(r,n) \times S)
\\
(\beta_{r,n} \times 1_{S})^{A} \colon A(\HGr(r+1,n+1) \times S) \to A(\HGr(r,n) \times S)
\end{gather*}
are split surjective, and we have isomorphisms
\begin{gather}
\label{E:HGr(r,infty)}
A(S) [[b_{1},\dots,b_{r}]] \xrightarrow{\cong} \varprojlim\limits_{n \to \infty} A(\HGr(r,n) \times S)
\\
\label{E:HGr(infty,2.infty)}
A(S)[[b_{1},b_{2},  b_{3}, \dots ]]
\xrightarrow{\cong}
\varprojlim
\limits_{n\to \infty}
A(\HGr(n,2n) \times S)
\end{gather}
with each variable $b_{i}$ sent to the inverse system of $i^{\text{th}}$ Borel classes
$(b_{i}(\shf U_{r,n}))_{n \geq r}$ or
$(b_{i}(\shf U_{n,2n}))_{n\in \NN}$.
\end{theorem}

The theorem follows from the explicit generators and relations for the $A(\HGr(r,n) \times S)$
given in Theorem \ref{T:Grass} in the same way as for ordinary Grassmannians.

\section{Recovering Thom classes from Borel classes}
\label{S:Pont.thom}

In this section we show that a symplectic Thom structure is determined by its system of Borel
classes.  This section is joint work with Alexander Nenashev.

\begin{definition}
\label{D:Pont}
A \emph{Borel structure} on a ring cohomology theory $B$ on a category of schemes is a rule
assigning to every rank $2$ symplectic bundle $(E,\phi)$ over a scheme $S$ in the category a
central element $b(E,\phi) \in B(S)$ with the following properties:

\begin{enumerate}
\item
For $(E_{1},\phi_{1}) \cong (E_{2},\phi_{2})$ we have $b(E_{1},\phi_{1}) = b(E_{2},\phi_{2})$.

\item
For a morphism $f \colon Y \to S$ we have $f^{B}(b(E,\phi)) = b(f^{*}(E,\phi))$.

\item
For the tautological rank $2$ symplectic subbundle
$(\shf U,\phi\rest{\shf U})$ on $\HP^{1}$ the maps
\[
(1,b(\shf U,\phi\rest{\shf U})) \colon B(S) \oplus B(S) \to B(\HP^{1} \times S)
\]
are isomorphisms.

\item
For a rank $2$ symplectic space $(V,\phi)$ viewed as a trivial symplectic bundle over $k$
we have $b(V,\phi) = 0$ in $B(k)$.

\end{enumerate}
\end{definition}

The Borel classes associated to a symplectic Thom structure by formula \eqref{eq.Borel}
form a Borel structure because of the functoriality of the Thom classes, the
Quaternionic Projective Bundle Theorem \ref{th.HPn.triv}, and
Proposition \ref{P:p=0}.

\begin{theorem}
\label{T:unique}
Every Borel structure on a ring cohomology theory is the system of Borel classes of a unique
symplectic Thom structure whose classes are given by formula \eqref{E:Thom} below.
\end{theorem}

The rest of this section is devoted to proving this theorem.
Our strategy goes as follows.
Let $(E,\phi)$ be a rank $2$ symplectic bundle over  $S$.
We will study the $\HP^{1}$ bundle associated to the rank $4$ symplectic bundle $(F,\psi)$ with
\begin{align}
\label{E:F.psi}
F & = \OO_{S} \oplus E \oplus \OO_{S}, &
\psi  = & \begin{pmatrix} 0 & 0 & 1 \\ 0 & \phi & 0 \\ -1 & 0 & 0 \end{pmatrix}.
\end{align}
The bundle $(F,\psi)$ is of the form studied in Theorem \ref{T:normal.geom},
so the $\HP^{1}$ bundle has the properties of that proposition.  Write
\[
\begin{gathered}
P  = \HP^{1}_{S}(F,\psi),
\\
N^{-} = \HP^{1}_{S}(F,\psi) \cap \Gr_{S}(2,E \oplus \OO_{S}),
\end{gathered}
\hspace{2cm}
\begin{gathered}
N^{+} = \HP^{1}_{S}(F,\psi) \cap \Gr_{S}(2,\OO_{S} \oplus E),
\\
S = N^{+} \cap N^{-} = \HP^{0}_{S}(E,\phi).
\end{gathered}
\]
Write $(U,\psi \rest{U})$ for the universal rank $2$ symplectic subbundle over $P$.
In the closed embedding $\HP^{0}_{S}(E,\phi) \subset \HP^{1}_{S}(F,\psi)$
of Theorem \ref{T:normal.geom}, the small quaternionic projective bundle is
now an $\HP^{0}$ bundle over $S$ and is therefore $S$ itself.
The restriction of $(U,\psi\rest{U})$ to this $S = \HP^{0}(E,\phi)$ is the
tautological rank $2$ symplectic subbundle of the rank $2$ symplectic bundle $(E,\phi)$.
So it is $(E,\phi)$ itself.  It now follows from Theorem \ref{T:normal.geom}
that the two summands of the normal bundle $N_{1}$ and $N_{2}$ satisfy
\begin{equation}
\label{E:XWj}
N^{+} \cong E^{\vee} \cong E
\hspace{3cm}
N^{-} \cong E^{\vee} \cong E
\end{equation}
We have the following diagram.
\begin{equation}
\label{E:*diag}
\vcenter{
\xymatrix @M=5pt @C=40pt {
**[l] E = N^{-}  \ar[r]^-{f} \ar@<3pt>[d]^-{\pi}
& P \ar[r]^-{s} \ar[dl]|-{\ h\ }
& U
\\
S \ar@<3pt>[u]^-{z} \ar@{_{(}->}[r]_-{g}
& N^{+} \ar[u]_-{i} \ar@{_{(}->}[r]_-{i}
& P \ar[u]_-{z_{U}}
}}
\end{equation}
The loci $S$, $N^{+}$, $N^{-}$, $P$ and the vector bundles $U \to P$ and $E \to S$ are as above,
the maps $i$, $f$, $g$, and the map $S\to N^{-}$ are the inclusions,
and $h$, $\pi$, and the map $N^{-} \to S$ are the projections to $S$.
We have $hf = \pi$.
The section
$s \colon P \to U$ corresponds to the vector bundle map
$s_{+} \colon \OO_{P} \to U$ of Theorem \ref{T:normal.geom}(d)
whose zero locus is $N^{+}$.  In particular, $s$ is transversal to the zero section $z_{U}$ of $U$.
The locus $N^{-}$ can be identified with the vector bundle $E$ by the isomorphism of
\eqref{E:XWj}, and $z$ and $\pi$ correspond to the zero section and structural map of the
vector bundle.
The intersections $z_{U}(P) \cap s(P) = N^{+}$ and $N^{+} \cap N^{-} = S$ are transversal.

We now prove a series of lemmas.  The first one applies to a slightly more general situation.

\begin{lemma}
\label{L:B.SS}
Let $B$ be a ring cohomology theory with a Borel structure.
Let $(G,\omega)$ be a rank $4$ symplectic bundle over a scheme $S$,
and let $(\shf U, \omega\rest{\shf U})$ be its tautological rank $2$ symplectic subbundle.
Then the map $(1,b(\shf U, \omega\rest{\shf U})) \colon B(S) \oplus B(S) \to B(\HP^{1}_{S}(G,\omega))$
is an isomorphism.
\end{lemma}

\begin{proof}
The lemma is true for a trivial $(G,\omega)$, which has $\HP^{1}_{S}(G,\omega) = \HP^{1} \times S$,
by the axioms of a Borel structure.
The general case follows by a Mayer-Vietoris argument.
\end{proof}

\begin{lemma}
\label{L:pont.to.thom}
Let $B$ be a ring cohomology theory with a Borel structure.
In the situation of \eqref{E:F.psi}--\eqref{E:*diag} let $\overline{e}^{B} \colon B_{X}(P) \to B(P)$ be
the supports extension map.  Then there exists a unique element $\theta_{E,\phi} \in B_{X}(P)$
satisfying $\overline{e}^{B}(\theta_{E,\phi}) = b(U, \psi\rest{U})$.
Moreover, $\theta_{E,\phi}$ is  $B(S)$-central, and the map $\theta \colon B(S) \to B_{X}(P)$
sending $\alpha \mapsto h^{B}(\alpha) \cup \theta_{E,\phi}$ is an isomorphism.
\end{lemma}

\begin{proof}
We are in the situation of Theorem \ref{th.Ga}.
This gives us two pieces of information.  First, write $Y = P \smallsetminus N^{+}$,
let $j \colon Y \into P$ be the inclusion, and
let $q = hj \colon Y \to S$ be the projection.
Theorem \ref{T:Ga.cohom}(a) then says that
$q^{B} \colon B(S) \to B(Y)$ is an isomorphism.

Second, by Theorem \ref{th.Ga} itself there are maps
\[
Y \xleftarrow{g_{1}} Y_{1} \xleftarrow[\cong]{g_{2}} Y_{2} \xrightarrow{q} \HGr_{S}(0,E,\phi) = S
\]
with $g_{1}$ an $\Aff^{1}$-bundle, $g_{2}$ an $\Aff^{0}$-bundle, and
$q$ an $\Aff^{5}$-bundle.  Moreover,
$g_{2}^{*}g_{1}^{*}(U\rest{Y})$ has two tautological sections $e,f$ over $Y_{2}$ which together give an
isometry with the trivial rank $2$ symplectic bundle
$\bigl( \OO_{S}^{\oplus 2}, \bigl(
\begin{smallmatrix}
0 & -1 \\ 1 & 0
\end{smallmatrix}
\bigr) \bigr)$.
Therefore $g_{2}^{B}g_{1}^{B}j^{B}(b(U,\psi \rest {U})) = b(g_{2}^{*}g_{1}^{*}j^{*}(U,\psi \rest {U}))$
is the pullback by $Y_{2} \to k$ of the Borel class of the trivial symplectic
bundle over $k$, which vanishes by the axioms of a Borel structure.
So we have $g_{2}^{B}g_{1}^{B}j^{B}(b(U,\psi\rest{U})) = 0$.
Moreover, since $g_{1}$ and $g_{2}$ are affine bundles,
$g_{2}^{B}g_{1}^{B}$ is an isomorphism.
So we have $j^{B}(b(U,\psi\rest{U})) = 0$.

Now consider the diagram of two-sided $B(S)$-modules.
\[
\xymatrix @M=5pt @C=40pt @R=25pt {
0 \ar[r] &
B(S) \ar[r]^-{(0,1)} \ar@{.>}[d]^-{\cong}_-{\theta} &
B(S) \oplus B(S) \ar[r]^-{(1,0)} \ar[d]_-{\left(1,b(U,\psi\rest{U}) \right) \circ h^{B}}^-{\cong} &
B(S) \ar[r] \ar[d]_-{q^{B}}^-{\cong} &
0
\\
\cdots \ar[r]_-{\partial = 0} &
B_{N^{+}}(P) \ar[r]_-{\overline{e}^{B}} &
B(P) \ar[r]_-{j^{B}} &
B(Y) \ar[r]_-{\partial = 0} &
\cdots
}
\]
The bottom row is the localization exact sequence.
The middle vertical arrow is an isomorphism by Lemma \ref{L:B.SS}, and $q^{B}$ is an isomorphism
by the discussion above.  Since we have $q^{B} = j^{B}h^{B}$, that also implies that $j^{B}$ is surjective,
and therefore $\partial$ vanishes.
The righthand square commutes because we have $j^{B}(b(U,\psi\rest{U})) = 0$.
It now follows that the diagram can be completed by a unique isomorphism $\theta$ making the lefthand
square commute.  Setting $\theta_{E,\phi} = \theta(1_{S})$, we get an element with all the desired
properties because $\theta$ is an isomorphism of two-sided $B(S)$-modules.
\end{proof}


\begin{theorem}
\label{T:exist}
Let $B$ be a ring cohomology theory with a Borel structure.
In the situation of \eqref{E:F.psi}--\eqref{E:*diag} let $\theta_{E,\phi} \in B_{X}(P)$
be the unique element
satisfying $\overline{e}^{B}(\theta_{E,\phi}) = b(U, \psi\rest{U})$ of Lemma \ref{L:pont.to.thom}.  Then the assignment
\begin{equation}
\label{E:Thom}
\Th(E,\phi) = -f^{B}(\theta_{E,\phi}) \in B_{S}(E)
\end{equation}
defines a symplectic Thom structure on $B$ whose Borel classes are the given Borel structure.
\end{theorem}

\begin{proof}
We verify the axioms of a symplectic Thom structure.

First, the $\Th(E,\phi)$ are supposed to be $B(S)$-central.
By Lemma \ref{L:B.SS} the element $\theta_{E,\phi} \in B_{N^{+}}(P)$ is $B(S)$-central, and
since $f \colon E = N^{-} \to P$ is compatible with the maps $\pi$ and $h$ to $S$, the map
$f^{B} \colon B_{N^{+}}(P) \to B_{S}(E)$ is a map of two-sided $B(S)$-modules.
So $f^{B}(\theta_{E,\phi}) = -\Th(E,\phi)$ is $B(S)$-central.

The functoriality conditions on the $\Th(E,\phi)$ follow from the functoriality of the constructions
of \eqref{E:F.psi}--\eqref{E:*diag}.

The maps $\cup \Th(E,\phi) \colon B(S) \to B_{S}(E)$ are isomorphisms
because they are compositions
\[
B(S) \xrightarrow[\cong]{-\theta} B_{N^{+}}(P) \xrightarrow[\cong]{f^{B}} B_{S}(E)
\]
of maps which are isomorphisms because of Lemma \ref{L:pont.to.thom} and Proposition \ref{P:iso}.

Hence the $\Th(E,\phi)$ define a symplectic Thom structure.

Finally, the Borel classes $-z^{B}e^{B}(\Th(E,\phi))$ defined by the symplectic Thom structure are the original ones because using the commutative diagram
\[
\xymatrix @M=5pt @C=30pt {
B_{X}(P) \ar[r]^-{\overline{e}^{B}} \ar[d]^-{f^{B}} &
B(P) \ar[d]^-{f^{B}}
\\
B_{S}(E) \ar[r]^-{e^{B}} &
B(E) \ar[r]^-{z^{B}} &
B(S)
}
\]
we see that we have
\[
-z^{B}e^{B}(\Th(E,\phi)) = z^{B}e^{B}f^{B}(\theta_{E,\phi}) =
z^{B}f^{B}\overline{e}^{B}(\theta_{E,\phi}) = z^{B}f^{B}(b(U,\psi)) = b(E,\phi)
\]
using functoriality of the Borel classes.
\end{proof}

\begin{theorem}
\label{T:unique.2}
Let $B$ be a ring cohomology theory with a Borel structure.  Suppose that the system of Borel
classes are those of some symplectic Thom structure with classes $\Th(E,\phi)$.
In the situation of \eqref{E:F.psi}--\eqref{E:*diag} let $\theta_{E,\phi} \in B_{N^{+}}(P)$
be the unique element
satisfying $\overline{e}^{B}(\theta_{E,\phi}) = b(U, \psi\rest{U})$ of Lemma \ref{L:pont.to.thom}.  Then
we have $\Th(E,\phi) = - f^{B}(\theta_{E,\phi}) \in B_{S}(E)$.
\end{theorem}

Thus the symplectic Thom structure inducing the Borel structure is unique.

\begin{proof}
In Proposition \ref{P:pont.zero} we saw that $b(U,\psi\rest{U})$ is the image of
$-\Th(U,\psi\rest{U})$ under the composition
\[
B_{P}(U) \xrightarrow{s^{B}} B_{X}(P) \xrightarrow{\overline{e}^{B}} B(P).
\]
Since $\theta_{E,\phi}$ is the unique element with $\overline{e}^{B}(\theta_{E,\phi}) = b(U,\psi\rest{U})$,
we have $\theta_{E,\phi} = -s^{B}(\Th(U,\psi\rest{U}))$.
To complete the proof of the lemma, we need to prove
$f^{B}s^{B}(\Th(U,\psi\rest{U})) = \Th(E,\phi)$.

Now consider the diagram.
\begin{equation*}
\vcenter{
\xymatrix @M=5pt @C=40pt {
E \ar[d]^-{\pi} \ar[r]^-{i_{1}} &
E \oplus E  \ar@<3pt>[d]^-{b_{2}} \ar[r]^-{f_{1}} &
U \ar@<3pt>[d]
\\
S  \ar[r]^-{z} &
E \ar@<3pt>[u]^-{\Delta} \ar[r]^-{f}  &
P \ar@<3pt>[u]^-{s}
}}
\end{equation*}
According to Theorem \ref{T:normal.geom}(e), the pullbacks to $N^{-}$ of the two bundles $U \to P$ and
$E \to S$ are isomorphic.  Pulling further back along the isomorphism $E \cong W$ gives
an isomorphism $j_{2}^{*}f^{*}U \cong \pi^{*}E = E \oplus E$ of bundles over $E$.
According to
Lemma \ref{L:subtle}, the section $s$ of $U$ whose zero locus is $X$ pulls back to the
tautological diagonal section $\Delta$ of $\pi^{*}E = E \oplus E$.  This gives the righthand
square.  The lefthand square is clear.  The functoriality of the Thom classes gives
$\Th(E,\phi) = i_{1}^{B} f_{1}^{B}(\Th(U,\psi)) \in B_{S}(E)$.
The first projection $b_{1} \colon E \oplus E \to E$ is an $\Aff^{2}$-bundle
with $b_{1}^{-1}(S) =  0 \oplus E$, and it has sections $i_{1}$ and $\Delta$.  It follows that the maps
$i_{1}^{B}$ and $\Delta^{B} \colon B_{0 \oplus E}(E \oplus E) = B_{E}(\pi^{*}E) \to B_{S}(E)$
are equal.  Substituting, and using also $sf = f_{1}\Delta$, we get the desired formula
\[
\Th(E,\phi) = i_{1}^{B} f_{1}^{B}(\Th(U,\psi)) = \Delta^{B}f_{1}^{B}(\Th(U,\psi)) = f^{B}s^{B}(\Th(U,\psi)) = -f^{B}(\theta_{E,\phi}).
\qedhere
\]
\end{proof}

\begin{proof}[Proof of Theorem \ref{T:unique}]
Theorem \ref{T:exist} showed that every Borel structure consists of the Borel classes associated
to the symplectic Thom structure given by formula \eqref{E:Thom}, and Theorem \ref{T:unique.2}
showed that this was the only symplectic Thom structure inducing the given Borel structure.
\end{proof}

\section{Thom classes of higher rank bundles}
\label{S:higher}

In this section we define Thom classes for higher rank symplectic bundles and prove some of
their properties.  The construction proceeds as in \S \ref{S:Pont.thom} only in higher rank and
using top Borel classes.
But now that Theorem \ref{T:unique} has been proven establishing the equivalence of
symplectic Thom structures and Borel structures, we may return to using a
ring cohomology theory $A$ with a symplectic Thom structure.
So we now have at our disposal our results on quaternionic Grassmannians.

Let $(E,\phi)$ be a symplectic bundle of rank $2r$ over $S$,
and let $(F,\psi)$ be the rank $2r+2$ symplectic bundle of \eqref{E:basic.setup} or \eqref{E:F.psi}.
Set
\begin{equation}
\label{E:PWXS.2}
\begin{gathered}
P  = \HGr_{S}(r,F,\psi),
\\
S = N^{+} \cap N^{-} = \HGr_{S}(r,E,\phi),
\end{gathered}
\hspace{7mm}
\begin{gathered}
N^{+} = \HGr_{S}(r,F,\psi) \cap \Gr_{S}(2r,\OO_{S} \oplus E),
\\
N^{-} = \HGr_{S}(r,F,\psi) \cap \Gr_{S}(2r,E \oplus \OO_{S}).
\end{gathered}
\end{equation}
Write $(U,\psi \rest{U})$ for the universal rank $2r$ symplectic subbundle over $P$.
The restriction of $(U,\psi \rest{U})$ to $S = \HGr_{S}(r,E,\phi)$ is again $(E,\phi)$.
We have the same isomorphisms $N^{+} \cong E$ and $N^{-} \cong E$
as in \eqref{E:XWj} and the same commutative diagram \eqref{E:*diag}.  The main difference
is that the relative dimensions over $S$ of the various loci and bundles are now
\begin{align*}
\dim_{S} S & = 0, &
\dim_{S} E & = \dim_{S} N^{+}  = \dim_{S} N^{-} = 2r, &
\dim_{S} P & = 4r, &
\dim_{S} U & = 6r.
\end{align*}

\begin{proposition}
\label{P:pont.theta}
Let $A$ be a ring cohomology theory with a symplectic Thom structure.
In the situation above let $\overline{e}^{A} \colon A_{N^{+}}(P) \to A(P)$ be
the supports extension map.  Then there exists a unique element $\theta_{E,\phi} \in A_{N^{+}}(P)$
satisfying $\overline{e}^{A}(\theta_{E,\phi}) = b_{r}(U, \psi\rest{U})$.
Moreover, $\theta_{E,\phi}$ is  $A(S)$-central, and the map $\theta \colon A(S) \to A_{N^{+}}(P)$
sending $\alpha \mapsto h^{A}(\alpha) \cup \theta_{E,\phi}$ is an isomorphism.
\end{proposition}

\begin{proof}
We claim we have an isomorphism of two-sided $A(S)$-modules
\[
(1,b_{1}(U,\psi\rest{U}), \dots, b_{r}(U,\psi\rest{U}) ) \colon
A(S)^{\oplus r+1} \xrightarrow{\cong} A(P).
\]
This is because Theorem \ref{T:Grass.bdl} establishes that $A(P)$ is a free two-sided
$A(S)$-module with basis $s_{\lambda}(\shf U,\psi\rest{U})$ for
$\lambda \in \Pi_{r,1} = \{ (1^{i}) \mid 0 \leq i \leq r \}$.  But those particular
Schur polynomials are $1,e_{1},\dots,e_{r}$,
so the characteristic classes are $1,b_{1},\dots,b_{r}$.

Similarly, if we write $(\shf U_{r-1},\phi\rest{\shf U_{r-1}})$ for the tautological
rank $2r-2$ symplectic subbundle on $\HGr_{S}(r-1,E,\phi)$, we have an isomorphism of
two-sided $A(S)$-modules
\[
(1,b_{1}(\shf U_{r-1},\phi\rest{\shf U_{r-1}}), \dots, b_{r-1}(\shf U_{r-1},\phi\rest{\shf U_{r-1}}) ) \colon
A(S)^{\oplus r} \xrightarrow{\cong} A(\HGr_{S}(r-1,E,\phi)).
\]

Let $Y = P \smallsetminus N^{+}$ and let $j \colon Y \into P$ be the inclusion.
By Theorem \ref{T:Ga.cohom}(b) there is a map
$\sigma \colon A(\HGr_{S}(r-1,E,\phi)) \to Y \subset P$ classifying the rank $2r$
symplectic subbundle $\OO \oplus \shf U_{r-1} \oplus \OO$
of the pullback of $F$, and the pullback map
\[
\sigma^{A} \colon A(Y) \xrightarrow{\cong} A(\HGr_{S}(r-1,E,\phi))
\]
is an isomorphism.  Composing with the restriction gives a map
$\sigma^{A}j^{A} \colon A(P) \to A(Y) \cong A(\HGr_{S}(r-1,E,\phi))$
sending $1 \mapsto 1$ and (in simplified notation)
\[
b_{i}(U) \mapsto b_{i}(\shf O \oplus \shf U_{r-1} \oplus \shf O) =
\begin{cases}
b_{i}(\shf U_{r-1}) & \text{for $i = 1, \dots, r-1$}, \\
0 & \text{for $i=r$},
\end{cases}
\]
the equalities coming from the Cartan sum formula.
Therefore in the localization sequence
\[
\cdots \xrightarrow{\partial} A_{N^{+}}(P) \xrightarrow{\overline{e}^{A}}
A(P) \xrightarrow{j^{A}} A(Y) \xrightarrow{\partial} \cdots
\]
the map $j^{A}$ is split surjective, $\partial$ vanishes, and $\overline{e}^{A}$ is split injective
with image the free direct summand $A(S) \cdot b_{r}(U,\psi\rest{U})$.
So there exists indeed a unique element $\theta_{E,\phi} \in A_{X}(P)$
satisfying $\overline{e}^{A}(\theta_{E,\phi}) = b_{r}(U, \psi\rest{U})$.
Moreover, $\theta_{E,\phi}$ is $A(P)$-central like the Borel class and therefore also
$A(S)$-central, and the map $\theta \colon A(S) \to A_{X}(P)$
sending $\alpha \mapsto h^{A}(\alpha) \cup \theta_{E,\phi}$ is indeed an isomorphism.
\end{proof}

\begin{theorem}
\label{T:exist.2}
Let $A$ be a ring cohomology theory with a symplectic Thom structure.
For a symplectic bundle $(E,\phi)$ of rank $2r$ on $S$, let $P$, $N^{+}$, $N^{-}$, $f$, etc.,
be as in \eqref{E:PWXS.2} and \eqref{E:*diag}, and let $\theta_{E,\phi} \in A_{N^{+}}(P)$
be the unique element
satisfying $\overline{e}^{A}(\theta_{E,\phi}) = b(U, \psi\rest{U})$ of Proposition \ref{P:pont.theta}.
Then the assignment
\begin{equation}
\label{E:Thom.2}
\Th(E,\phi) = (-1)^{r} f^{A}(\theta_{E,\phi}) \in A_{S}(E)
\end{equation}
gives a system of classes with the following properties\textup{:}
\begin{enumerate}
\item
Each $\Th(E,\phi) \in A_{S}(E)$ is $A(S)$-central.
\item For an isomorphism $\gamma \colon (E,\phi) \cong (E_{1},\phi_{1})$
we have $\Th(E,\phi) = \gamma^{A}\Th(E_{1},\phi_{1})$.

\item For $u \colon T \to S$, writing $u_{E} \colon u^{*}E \to E$ for the pullback,
we have $u_{E}^{A}(\Th(E,\phi)) = \Th(u^{*}(E,\phi)) \in A_{T}(u^{*}E)$.

\item The maps $\cup \Th(E,\phi) \colon A(S) \to A_{S}(E)$ are isomorphisms.

\item We have $\Th \bigl( (E_{1},\phi_{1}) \perp (E_{2},\phi_{2}) \bigr)
= q_{1}^{A}\Th(E_{1},\phi_{1}) \cup q_{2}^{A}\Th(E_{2},\phi_{2})$,
where $q_{1},q_{2}$ are the projections from $E_{1} \oplus E_{2}$
onto its factors.
\end{enumerate}
Moreover, for $e^{A} \colon A_{S}(E) \to A(E)$ the extension of supports map, and $z^{A} \colon A(E) \to A(S)$
the restriction to the zero section, we have
\begin{equation}
\label{E:top.p}
b_{r}(E,\phi) = (-1)^{r} z^{A}e^{A}(\Th(E,\phi)),
\end{equation}
while for $(E,\phi)$ of rank $2$ the class $\Th(E,\phi)$ just defined is the same as the
class in the symplectic Thom structure.
\end{theorem}

The classes $\Th(E,\phi)$ are called the \emph{symplectic Thom classes}.

\begin{proof}[First part of the proof of Theorem \ref{T:exist.2}]
The proof of (1)--(4) and \eqref{E:top.p}
is identical to the proof of Theorem \ref{T:exist}.  The coincidence of the two
symplectic Thom classes for rank $2$ bundles was proven in Theorem \ref{T:unique.2}.
We will prove (5) at the end of this section.
\end{proof}

We may now generalize the direct image maps of \eqref{diag.transfer}.
When $i \colon X \to Y$ is a closed embedding  of smooth varieties of even codimension
$2r$ whose normal bundle $N = N_{X/Y}$ is equipped
with a specified symplectic form $\phi$, we say that $i \colon X \to Y$ has a
{\em symplectic normal bundle}\/ $(N,\phi)$.  We can compose the isomorphism
of Theorem \ref{T:exist.2}(4) with the deformation to
the normal bundle isomorphisms when they exist
and extension supports maps gives \emph{direct image maps}.
\begin{equation*}
i_{A,\flat} \colon  A(X) \xrightarrow[\cong]{\cup \Th(N,\phi)}
A_{X}(N) \xrightarrow[\cong]{d_{X/Y}} A_{X}(Y),
\hspace{20mm}
i_{A,\natural} \colon A(X) \xrightarrow{i_{A,\flat}} A_X(Y) \to A(Y)
\end{equation*}

The analogues of Propositions  \ref{P:tr=th}, \ref{P:torindep}, \ref{P:projection} and \ref{P:pont.zero}
hold for these more general direct image maps except that for a symplectic bundle $(E,\phi)$ of rank $2r$
equations \eqref{E:p.2nd} and \eqref{E:i_*i^*} should read
\begin{gather}
\label{E:p.2nd.r}
b_{r}(E,\phi)  = (-1)^{r}\overline{e}^{A} s^{A} \Th(E,\phi),
\\
\label{E:i_*i^*.r}
i_{A,\natural}i^A(b) = \overline e^A i_{A,\flat} i^A(b)  = (-1)^{r} b \cup b_{r}(E,\phi).
\end{gather}

Now let $(E,\phi)$ be symplectic bundle of rank $2n-2$ on $S$, and let
$(F,\psi) = (E,\phi) \perp \Hyp(\OO_{S})$ be as in \eqref{E:basic.setup}.
Let $N^{+} = \HGr_{S}(r,F,\psi) \cap \Gr(2r,\OO_{S}\oplus E)$ and
$Y = \HGr_{S}(r,F,\psi) \smallsetminus N^{+}$.
Let $u \colon N^{+} \to \HGr_{S}(r,E,\phi)$ be the bundle map of Theorem \ref{T:normal.geom}, and
let $i \colon N^{+} \into \HGr_{S}(r,F,\psi)$ be the inclusion.  Let
$\sigma \colon \HGr_{S}(r-1,E,\phi) \to Y \into \HGr_{S}(r,F,\psi)$ be the map of Theorem
\ref{T:Ga.cohom}.  Let $\tau = e^{A}q^{A}i_{A,\flat}$.
\begin{equation}
\label{E:Grass.synth}
\vcenter{
\xymatrix @M=5pt @C=12pt {
\cdots \ar[r]^-{\partial = 0}
& A_{N^{+}}(\HGr_{S}(r,F,\psi)) \ar[r]^-{e^{A}}
& A(\HGr_{S}(r,F,\psi)) \ar[r] \ar@{-->}[dr]_-{\sigma^{A}}
& A(Y) \ar[r]^-{\partial = 0} \ar[d]_-{\cong}
& \cdots
\\
& A(\HGr_{S}(r,E,\phi)) \ar[u]_-{\cong}^-{i_{A,\flat}q^{A}} \ar@{-->}[ru]^-{\tau}
& & A(\HGr_{S}(r-1,E,\phi))
}}
\end{equation}
Write $\shf U_{r}$ for the tautological bundles
on $\HGr_{S}(r,E,\phi)$, write $\shf U_{r-1}$ for the one on $\HGr_{S}(r-1,E,\phi)$,
and $\shf V_{r}$ for the one on $\HGr_{S}(r,F,\psi)$.  Write  $s_{\lambda}(\shf U_{r})$ for the Schur
polynomials in the Borel classes as in \eqref{E:schur.p}.

\begin{theorem}
\label{T:grass.synth}
The maps $\tau$ and $\sigma$ act on the $A(S)$-bases of the cohomology of the
quaternionic Grassmannian bundles by
\begin{align*}
\tau \colon & s_{\lambda}(\shf U_{r}) \mapsto (-1)^{r}b_{r}(\shf V_{r})s_{\lambda}(\shf V_{r}) =
(-1)^{r} s_{
\lambda+(1^{r})
}(\shf V_{r})
\qquad \text{for $\lambda \in \Pi_{r,n-1-r}$}
\\
\sigma \colon & s_{\lambda}(\shf V_{r}) \mapsto
\begin{cases}
s_{\lambda}(\shf U_{r-1}) & \text{for $\lambda \in \Pi_{r-1,n-r}$}, \\
0 & \text{for $\lambda \in \Pi_{r,n-r} \smallsetminus \Pi_{r-1,n-r}$}.
\end{cases}
\end{align*}
\end{theorem}

\begin{proof}
For a partition $\lambda$ of length $l(\lambda) \leq r$, one has equality $l(\lambda) = r$ if and only if
the dual partition is of the form $\lambda' = (r, \lambda_{2}',\dots,\lambda_{m}')$.  When that occurs
formula \eqref{E:schur.det} gives
\[
s_{\lambda} =  \begin{vmatrix}
e_{r} & 0 & \cdots & 0 \\
e_{\lambda_{2}'-1} & e_{\lambda_{2}'} & \cdots & e_{\lambda_{2}'+m-2} \\
\vdots & \vdots & \ddots & \vdots \\
e_{\lambda_{m}'-m+1} & e_{\lambda_{m}'-m+2} & \cdots & e_{\lambda_{m}'}
\end{vmatrix}
= e_{r}s_{\lambda-(1^{r})}.
\]
in $\Lambda_{r}$.
Therefore, keeping in mind that the pullbacks of the tautological subbundles along the projection
$q \colon N^{+} \to \HGr_{S}(r,E,\phi)$ and the inclusion $i \colon N^{+} \to \HGr_{S}(r,F,\psi)$
are isometric, we have
\[
\tau(s_{\lambda}(\shf U_{r})) = e^{A}i_{A,\flat}q^{A}(s_{\lambda}(\shf U_{r})) =
e^{A}i_{A,\flat}i^{A}(s_{\lambda}(\shf V_{r}))
= (-1)^{r}b_{r}(\shf V_{r})s_{\lambda}(\shf V_{r}) = (-1)^{r}s_{\lambda+(1^{r})}(\shf V_{r}),
\]
using \eqref{E:i_*i^*.r}.

The map $\sigma$ classifies the rank $2r$ symplectic bundle
$\OO \oplus \shf U_{r-1} \oplus \OO$
on $\HGr_{S}(r{-}1,E,\phi)$,  so we have
\[
\sigma^{A}(s_{\lambda}(\shf V_{r})) = s_{\lambda}(\OO \oplus \shf U_{r-1} \oplus \OO)
= s_{\lambda}(\shf U_{r-1})
\]
because the Borel classes of $\shf U_{r-1}$ and $\OO \oplus \shf U_{r-1} \oplus \OO$ are equal by
the Cartan sum formula.
For the $\lambda \in \Pi_{r-1,n-r}$ this is one of the elements of the
basis of $A(\HGr_{S}(r-1,E,\phi)$ as a two-sided $A(S)$-module.
For $\lambda \in \Pi_{r,n-r} \smallsetminus \Pi_{r-1,n-r}$, i.e.\ those with $l(\lambda) = r$, we
have $s_{\lambda}(\shf U_{r-1}) = b_{r}(\shf U_{r-1}) s_{\lambda-(1^{r})}(\shf U_{r-1}) = 0$,
since the tautological
symplectic subbundle on $\HGr_{S}(r-1,E,\phi)$ is of rank $2r-2$ and has $b_{r}(\shf U_{r-1}) = 0$.
\end{proof}

\begin{proof}[End of the proof of Theorem \ref{T:exist.2}]
We now prove property (5) of the theorem.
Suppose $\rk E_{1} = 2r_{1}$ and $\rk E_{2} = 2r_{2}$.  Let
\begin{align*}
P_{r_{1}+r_{2}} & = \HGr_{S} \bigl( r_{1}+r_{2}, (E_{1},\phi_{1}) \perp (E_{2}, \phi_{2}) \perp
\Hyp(\OO_{S}) \bigr),
\\
P_{r_{1}} & = \HGr_{S} \bigl( r_{1}, (E_{1},\phi_{1}) \perp \Hyp(\OO_{S}) \bigr),
\\
P_{r_{2}} & = \HGr_{S} \bigl( r_{2},  (E_{2}, \phi_{2}) \perp \Hyp(\OO_{S}) \bigr),
\\
F_{r_{1},r_{2}} & = \HFlag_{S} \bigl( r_{1},r_{2}; (E_{1},\phi_{1}) \perp (E_{2}, \phi_{2}) \perp
\Hyp(\OO_{S}) \bigr),
\\
G_{r_{1}} & = \HGr_{S} \bigl( r_{1}, (E_{1},\phi_{1}) \perp (E_{2}, \phi_{2}) \perp
\Hyp(\OO_{S}) \bigr),
\\
G_{r_{2}} & = \HGr_{S} \bigl( r_{2}, (E_{1},\phi_{1}) \perp (E_{2}, \phi_{2}) \perp
\Hyp(\OO_{S}) \bigr)
\end{align*}
We have maps
\[
\xymatrix @M=5pt @C=30pt {
P_{r_{1}} \ar[r]^-{t_{1}}
& F_{r_{1},r_{2}} \ar[d]^-{\rho} \ar[r]^-{\rho_{1}} \ar[rd]^-{\rho_{2}}
& G_{r_{1}}
\\
P_{r_{2}} \ar[ru]^-{t_{2}}
& P_{r_{1}+r_{2}}
& G_{r_{2}}
}
\]
defined as follows.
Over $F_{r_{1},r_{2}}$ there are orthogonal tautological symplectic subbundles $U_{1}$, $U_{2}$
of ranks $2r_{1}$ and $2r_{2}$ respectively.  The maps $\rho_{i}$ are classified by the $U_{i}$ and
$\rho$ by $U_{1} \oplus U_{2}$.  All three projections are quaternionic Grassmannian bundles.

Over $P_{r_{1}}$ there is a tautological rank $2r_{1}$ subbundle
$\shf V_{1} \subset E_{1} \oplus \Hyp(\OO)$.  The map $t_{1}$ is classified by the pair of orthogonal
symplectic subbundles $\shf V_{1} \perp E_{2} \subset E_{1} \oplus E_{2} \oplus \Hyp(\OO)$.  The map
$t_{2}$ is defined analogously, reversing the roles of $E_{1}$ and $E_{2}$.

In $P_{r_{1}}$,  $P_{r_{2}}$ and $P_{r_{1}+r_{2}}$ there are the loci and maps of \eqref{E:*diag}
\begin{gather*}
\vcenter{
\xymatrix @M=5pt @C=35pt {
**[l] E_{1} = N_{1}^{-}  \ar[r]^-{f_{1}} \ar@<3pt>[d]^-{\pi_{1}}
& P_{r_{1}} \ar[r]^-{s_{1}} \ar[dl]|-{\ h_{1}\ }
& U_{1}
\\
S_{1} \ar@<3pt>[u]^-{z_{1}} \ar@{_{(}->}[r]_-{g_{1}}
& N_{1}^{+} \ar[u]_-{i_{1}} \ar@{_{(}->}[r]_-{i_{1}}
& P_{r_{1}} \ar[u]_-{z_{U_{1}}}
}}
\hspace{5mm}
\vcenter{
\xymatrix @M=5pt @C=35pt {
**[l] E_{2} = N_{2}^{-}  \ar[r]^-{f_{2}} \ar@<3pt>[d]^-{\pi_{2}}
& P_{r_{2}} \ar[r]^-{s_{2}} \ar[dl]|-{\ h_{2}\ }
& U_{2}
\\
S_{2} \ar@<3pt>[u]^-{z_{2}} \ar@{_{(}->}[r]_-{g_{2}}
& N_{2}^{+} \ar[u]_-{i_{2}} \ar@{_{(}->}[r]_-{i_{2}}
& P_{r_{2}} \ar[u]_-{z_{U_{2}}}
}}
\\
\vcenter{
\xymatrix @M=5pt @C=30pt {
**[l] E_{1}\oplus E_{2} = N^{-}  \ar[r]^-{f} \ar@<3pt>[d]^-{\pi}
& P_{r_{1}+r_{2}} \ar[r]^-{s} \ar[dl]|-{\ h\ }
& U
\\
S \ar@<3pt>[u]^-{z} \ar@{_{(}->}[r]_-{g}
& N^{+} \ar[u]_-{i} \ar@{_{(}->}[r]_-{i}
& P_{r_{1}+r_{2}} \ar[u]_-{z_{U}}
}}
\end{gather*}
Writing $\mu_{1} \colon \OO_{E_{1}} \to \pi_{1}^{*}E_{1}$ and
$\mu_{2} \colon \OO_{E_{2}} \to \pi_{2}^{*}E_{2}$ for
the tautological sections.
The compositions
\begin{align*}
\gamma_{1} \colon E_{1} = N^{-}_{1} \xrightarrow{f_{1}} P_{r_{1}}
\xrightarrow{t_{1}} F_{r_{1},r_{2}}
&&
\gamma_{2} \colon E_{2} = N^{-}_{2} \xrightarrow{f_{2}} P_{r_{2}}
\xrightarrow{t_{2}} F_{r_{1},r_{2}}
\end{align*}
are the maps classified by the orthogonal pairs of subbundles which are the images of
\begin{gather*}
\pi_{1}^{*}E_{1}  \xrightarrow{
\left(\begin{smallmatrix}
0  \\ 1  \\ 0  \\ \mu_{1}^{\vee} \phi_{1}
\end{smallmatrix}\right)
}
\OO_{E_{1}} \oplus \pi_{1}^{*} E_{1} \oplus \pi_{1}^{*}E_{2} \oplus \OO_{E_{1}}
\xleftarrow{
\left(\begin{smallmatrix}
0 \\ 0 \\ 1 \\  0
\end{smallmatrix}\right)
}
\pi_{1}^{*}E_{2}
\\
\pi_{2}^{*}E_{1}  \xrightarrow{
\left(\begin{smallmatrix}
0  \\ 1  \\ 0  \\ 0
\end{smallmatrix}\right)
}
\OO_{E_{2}} \oplus \pi_{2}^{*} E_{1} \oplus \pi_{2}^{*}E_{2} \oplus \OO_{E_{2}}
\xleftarrow{
\left(\begin{smallmatrix}
0 \\ 0 \\ 1 \\  \mu_{2}^{\vee} \phi_{2}
\end{smallmatrix}\right)
}
\pi_{2}^{*}E_{2}
\end{gather*}
The restriction of $\rho \colon F_{r_{1},r_{2}} \to P_{r_{1}+r_{2}}$ to the inverse image of
$E_{1} \oplus E_{2} = N^{-}$ has a section
$\gamma \colon N^{-} \to \rho^{-1}(N^{-}) \subset F_{r_{1},r_{2}}$
classified by the orthogonal pair of subbundles
\[
\pi^{*}E_{1}  \xrightarrow{
\left(\begin{smallmatrix}
0  \\ 1  \\ 0  \\ \mu_{1}^{\vee} \phi_{1}
\end{smallmatrix}\right)
}
\OO_{E_{1} \oplus E_{2}} \oplus \pi^{*} E_{1} \oplus \pi^{*}E_{2} \oplus \OO_{E_{1} \oplus E_{2}}
\xleftarrow{
\left(\begin{smallmatrix}
0 \\ 0 \\ 1 \\  \mu_{2}^{\vee} \phi_{2}
\end{smallmatrix}\right)
}
\pi^{*}E_{2}
\]
The restriction of $\gamma \colon E_{1} \oplus E_{2} \to F_{r_{1},r_{2}}$ to each factor $E_{i}$
thus coincides with $\gamma_{i} \colon E_{i} \to F_{r_{1},r_{2}}$.

In $G_{r_{1}}$ there is the locus
$\overline{N}^{+}_{1}  = G_{r_{1}} \cap \Gr_{S}(2r_{1},\OO_{S} \oplus E_{1} \oplus E_{2})$
and there is an analogous locus $\overline{N}^{+}_{2} \subset G_{r_{2}}$.  We have
$\rho_{1}^{-1}(\overline{N}^{+}_{1}) \cap \rho_{2}^{-1}(\overline{N}^{+}_{2}) = \rho^{-1}(N^{+})$.

Call $U_{1}$ all the tautological bundles of rank $2r_{1}$, call $U_{2}$ all the tautological bundles of
rank $2r_{2}$, and call $U$ all the tautological bundles of rank $2r_{1} + 2r_{2}$.

By Proposition \ref{P:pont.theta} and Theorem \ref{T:grass.synth} the extension of supports maps
$A_{\overline{N}^{+}_{i}}(G_{r_{i}}) \to A(G_{r_{i}})$ and $A_{N^{+}_{i}}(P_{r_{i}}) \to A(P_{r_{i}})$ and
$A_{N^{+}}(P_{r_{1}+r_{2}}) \to A(P_{r_{1}+r_{2}})$ are all split injective.
Since the $\rho_{i} \colon F_{r_{1},r_{2}} \to G_{r_{i}}$ and $\rho \colon F_{r_{1},r_{2}}$
are quaternionic Grassmannian bundles,
we may pull back along the $\rho_{i}$ and $\rho$ and see that the
extension of supports maps $A_{\rho_{i}^{-1}(\overline{N}^{+}_{i})}(F_{r_{1},r_{2}}) \to A(F_{r_{1},r_{2}})$
and $A_{\rho^{-1}(N^{+})}(F_{r_{1},r_{2}})$ are also split injective.

In the notation of \eqref{E:Grass.synth} let
\[
\overline{\theta}_{E_{1},\phi_{1}} =
(-1)^{r} i_{A,\flat}q^{A}(1_{\overline{N}^{+}_{1} \cap \overline{N}^{-}_{1}})
\in A_{\overline{N}^{+}_{1}}(G_{r_{1}}).
\]
Then by Theorem \ref{T:grass.synth} and the functoriality of the Borel classes
\begin{align*}
\overline{\theta}_{E_{1},\phi_{1}} \in A_{\overline{N}^{+}_{1}}(G_{r_{1}}), &&
\rho_{1}^{A}(\overline{\theta}_{E_{1},\phi_{1}}) \in
A_{\rho_{1}^{-1}(\overline{N}^{+}_{1})}(F_{r_{1},r_{2}}),
&&
t_{1}^{A}\rho_{1}^{A}(\overline{\theta}_{E_{1},\phi_{1}}) \in A_{N^{+}_{1}}(P_{r_{1}}),
\end{align*}
are the unique classes in their respective cohomology groups whose images under the extension of
supports maps are the Borel classes $p_{r_{1}}(U_{1},\phi_{1})$.  Therefore
$t_{1}^{A}\rho_{1}^{A}(\overline{\theta}_{E_{1},\phi_{1}})$ is the $\theta_{E_{1},\phi_{1}}$ of
Theorem \ref{T:unique.2}, and
$\Th(E_{1},\phi_{1}) = \gamma_{1}^{A} \rho_{1}^{A}(\overline{\theta}_{E_{1},\phi_{1}})
\in A_{S}(E_{1})$, while $q_{1}^{A}\Th(E_{1},\phi_{1}) = \gamma^{A} \rho_{1}^{A}(\overline{\theta}_{E_{1},\phi_{1}})
\in A_{E_{2}}(E_{1} \oplus E_{2})$.

There are analogous classes
\begin{align*}
\overline{\theta}_{E_{2},\phi_{2}} \in A_{\overline{N}^{+}_{2}}(G_{r_{2}}), &&
\rho_{2}^{A}(\overline{\theta}_{E_{2},\phi_{2}}) \in
A_{\rho_{2}^{-1}(\overline{N}^{+}_{2})}(F_{r_{1},r_{2}}),
\end{align*}
such that $\rho_{2}^{A}(\overline{\theta}_{E_{2},\phi_{2}})$ is the unique class whose extension of
supports is $p_{r_{2}}(U_{2},\phi_{2}) \in A(F_{r_{1},r_{2}})$ and
$\gamma^{A}\rho_{2}^{A}(\overline{\theta}_{E_{2},\phi_{2}}) = q_{2}^{A}\Th(E_{2},\phi_{2})$.

Similarly there are classes
\begin{align*}
\theta_{E_{1}\oplus E_{2}, \phi_{1} \oplus \phi_{2}} \in A_{N^{+}}(P_{r_{1}+r_{2}}), &&
\rho^{A}(\theta_{E_{1}\oplus E_{2},\phi_{1} \oplus \phi_{2}}) \in A_{\rho^{-1}(N^{+})}(F_{r_{1},r_{2}}),
\end{align*}
which are the unique classes whose images under the extension of supports maps are the
Borel classes $p_{r_{1}+r_{2}}(U)$.  We have
\[
\Th(E_{1}\oplus E_{2}, \phi_{1} \oplus \phi_{2}) =
f^{A}\theta_{E_{1}\oplus E_{2}, \phi_{1} \oplus \phi_{2}} =
\gamma^{A}\rho^{A}\theta_{E_{1}\oplus E_{2}, \phi_{1} \oplus \phi_{2}}.
\]

Now
\[
\rho_{1}^{A}(\overline{\theta}_{E_{1},\phi_{1}}) \rho_{2}^{A}(\overline{\theta}_{E_{2},\phi_{2}})
\in A_{\rho_{1}^{-1}(\overline{N}^{+}_{1}) \cap \rho_{2}^{-1}(\overline{N}^{+}_{2})}(F_{r_{1},r_{2}}) =
A_{\rho^{-1}(N^{+})}(F_{r_{1},r_{2}})
\]
is a class whose image under the extension of supports map is
$b_{r_{1}}(U_{1},\phi_{1}) b_{r_{2}}(U_{2},\phi_{2}) = b_{r_{1}+r_{2}}(U,\phi)$.
(Top Borel classes multiply.)  Since $\rho^{A}(\theta_{E_{1}\oplus E_{2},\phi_{1} \oplus \phi_{2}})$
was the unique class with that property we have
\[
\rho_{1}^{A}(\overline{\theta}_{E_{1},\phi_{1}}) \rho_{2}^{A}(\overline{\theta}_{E_{2},\phi_{2}})
= \rho^{A}(\theta_{E_{1}\oplus E_{2},\phi_{1} \oplus \phi_{2}}).
\]
Applying $\gamma^{A}$ now gives
\[
q_{1}^{A}(\Th(E_{1},\phi_{1})) \, q_{2}^{A}(\Th(E_{2},\phi_{2})) =
\Th(E_{1}\oplus E_{2}, \phi_{1} \oplus \phi_{2}).
\qedhere
\]
\end{proof}

\section{Symplectic orientations}
\label{S:orientations}

A ring cohomology theory can be symplectically oriented by any of five structures
satisfying different axioms.  We have already seen two of them: a symplectic Thom structure
(Definition \ref{def.sp.thom})
and a Borel structure (Definition \ref{D:Pont}).  Here are the other three
(cf.\ \cite[Definitions 3.26, 3.32, 3.1]{Panin:2003rz}).

\begin{definition}
\label{D:Pont.classes}
A \emph{Borel classes theory} on a ring cohomology theory $A$
on a category of schemes is a system of assignments to every
symplectic bundle $(E,\phi)$ over every scheme $S$ in the category
of elements $b_{i}(E,\phi) \in A(S)$ for all
$i \geq 1$ satisfying
\begin{enumerate}
\item
For $(E_{1},\phi_{1}) \cong (E_{2},\phi_{2})$ we have $b_{i}(E_{1},\phi_{1}) = b_{i}(E_{2},\phi_{2})$
for all $i$.

\item
For a morphism $f \colon Y \to S$ we have $f^{A}(b_{i}(E,\phi)) = b_{i}(f^{*}(E,\phi))$ for all $i$.

\item
For the tautological rank $2$ symplectic subbundle
$(\shf U,\phi\rest{\shf U})$ on $\HP^{1}$ the maps
\[
(1,b_{1}(\shf U,\phi\rest{\shf U})) \colon A(S) \oplus A(S) \to A(\HP^{1} \times S)
\]
are isomorphisms for all $S$.

\item
For a rank $2$ symplectic space $(V,\phi)$ viewed as a trivial symplectic bundle over
$k$ we have $b_{1}(V,\phi) = 0$ in $A(k)$.

\item
For an orthogonal direct sum of symplectic bundles $(E,\phi) \cong (E_{1},\phi_{1}) \perp (E_{2},\phi_{2})$
we have $b_{i}(E,\phi) = b_{i}(E_{1},\phi_{1}) + \sum_{j=1}^{i-1} b_{i-j}(E_{1},\phi_{1}) b_{j}(E_{2},\phi_{2})
+ b_{i}(E_{2},\phi_{2})$ for all $i$.

\item
For $(E,\phi)$ of rank $2r$ we have $b_{i}(E,\phi) = 0$ for $i > r$.
\end{enumerate}
\end{definition}

One may also set $b_{0}(E,\phi) = 1$ and even $b_{i}(E,\phi) = 0$ for $i < 0$.

\begin{definition}
\label{D:thom.class}
A \emph{symplectic Thom classes theory} on a ring cohomology theory $A$ on a category of schemes
is a system of assignments
to every
symplectic bundle $(E,\phi)$ over every scheme $X$ in the category
of an element $\Th(E,\phi) \in A_{X}(E)$
satisfying conditions (1)--(5) of Theorem \ref{T:exist.2}.
\end{definition}

\begin{definition}
\label{D:symp.orient}
A \emph{symplectic orientation} on a ring cohomology theory $A$ on a category of schemes
is a system of assignments to every
symplectic bundle $(E,\phi)$ over every scheme $X$ in the category and every
closed subset $Z \subset X$ with $X \smallsetminus Z$ in the category of an
isomorphism $\Th^{E,\phi}_{Z} \colon A_{Z}(X) \to A_{Z}(E)$ with the following properties.
\begin{enumerate}
\item
Let $\pi \colon E \to X$ be the structure map, and let $\pi^{A} \colon A_{Z}(X) \to A_{\pi^{-1}(Z)}(E)$ be
the pullback.
Then for all $a \in A(X)$ and $b \in A_{Z}(X)$ one has
\begin{align*}
\Th^{E,\phi}_{Z}(a \cup b) & = \Th^{E,\phi}_{X}(a) \cup \pi^{A}b, &
\Th^{E,\phi}_{Z}(b \cup a) = \pi^{A}b \cup \Th^{E,\phi}_{X}(a).
\end{align*}

\item
For every isometry of symplectic bundles $\phi \colon (E,\phi) \to (F,\psi)$ the following diagram
commutes
\[
\xymatrix @M=5pt @C=30pt {
A_{Z}(X) \ar[r]_-{\cong}^-{\Th^{F,\psi}_{Z}} \ar@{=}[d]_-{1}
& A_{Z}(F) \ar[d]_-{\cong}^-{\phi^{A}}
\\
A_{Z}(X) \ar[r]_-{\cong}^-{\Th^{E,\phi}_{Z}}
& A_{Z}(E).
}
\]

\item
For every morphism $f \colon X' \to X$ with $Z' \subset X'$ closed and $f^{-1}(Z) \subset Z'$, then for
$(E',\phi') = f^{*}(E,\phi)$ and $g \colon E' \to E$ the pullback of $f$ along $\pi \colon E \to X$,
the following diagram commutes
\[
\xymatrix @M=5pt @C=30pt {
A_{Z}(X) \ar[r]_-{\cong}^-{\Th^{E,\phi}_{Z}} \ar[d]_-{f^{A}}
& A_{Z}(E) \ar[d]^-{g^{A}}
\\
A_{Z'}(X') \ar[r]_-{\cong}^-{\Th^{E',\phi'}_{Z'}}
& A_{Z'}(E').
}
\]

\item
For every pair of sympectic bundles $(E_{1},\phi_{1})$ and $(E_{2},\phi_{2})$ over a scheme $X$,
with structural maps $b_{i} \colon E_{i} \to X$, the following diagram
commutes
\[
\xymatrix @M=5pt @C=50pt {
A_{Z}(X) \ar[r]_-{\cong}^-{\Th^{E_{1},\phi_{1}}_{Z}} \ar[d]^-{\cong}_-{\Th^{E_{2},\phi_{2}}_{Z}}
& A_{Z}(E_{1}) \ar[d]^-{\Th_{Z}^{b_{1}^{*}E_{2},b_{1}^{*}\phi_{2}}}_-{\cong}
\\
A_{Z}(E_{2}) \ar[r]^-{\Th_{Z}^{b_{2}^{*}E_{1},b_{2}^{*}\phi_{1}}}_-{\cong}
& A_{Z}(E_{1}\oplus E_{2}).
}
\]

\end{enumerate}
\end{definition}

Property (1) of Definition \ref{D:symp.orient} implies that the $\Th^{E,\phi}_{Z}$ are $A(X)$-bimodule
maps, but at least formally it is somewhat stronger (\cite[Lemma 3.33]{Panin:2003rz} did not seem
obvious to the second author with the formally weaker property).

\begin{theorem}
\label{T:bijections}
Let $A$ be a ring cohomology theory.

\subthm
There are inverse bijections
\[
\{\text{symplectic Thom structures on $A$}\} \longleftrightarrow \{\text{Borel structures on $A$}\}
\]
given by the formulas \eqref{eq.Borel} and \eqref{E:Thom}.

\subthm
There are inverse bijections
\[
\{\text{Borel class structures on $A$}\} \longleftrightarrow \{\text{Borel structures on $A$}\}
\]
given by forgetfulness and the map which assigns to a Borel structure the Borel classes
assigned by Definition \ref{D:pontclass} to the associated symplectic Thom structure.

\subthm
There are inverse bijections
\[
\{\text{symplectic Thom class structures on $A$}\} \longleftrightarrow
\{\text{symplectic Thom structures on $A$}\}
\]
given by forgetfulness and Theorem \ref{T:exist.2}.

\subthm
There are inverse bijections
\[
\{\text{symplectic Thom class structures on $A$}\} \longleftrightarrow
\{\text{symplectic orientations on $A$}\}
\]
given by the formulas $\Th^{E,\phi}_{Z}(b) = \pi^{A}b \cup \Th(E,\phi)$ and
$\Th(E,\phi) = \Th^{E,\phi}_{X}(1_{X})$.
\end{theorem}

\begin{proof}
(a) This is Theorem \ref{T:unique}.

(b) The right-to-left map is well-defined because the Borel classes associated to a
symplectic Thom structure satisfy the functoriality axioms (1)(2) by simple arguments, axiom
(3) concerning $\HP^{1}$ by Theorem \ref{th.HPn.triv}, the triviality axiom (4)
by Proposition \ref{P:p=0}, the Cartan sum formula (5) by
Theorem \ref{T:Cartan}, and the dimension axiom (6) holds by definition.
The right-to-left-to-right roundtrip assigns to a given Borel structure the Borel structure
associated to the associated Thom structure.  This is the identity map because of Theorem
\ref{T:exist}.  The left-to-right map is injective because of the splitting principle (Theorem \ref{T:splitting})
and the Cartan sum formula (cf.\ \cite[Theorem 3.27]{Panin:2003rz}).

(c) The right-to-left map is well-defined by Theorem \ref{T:exist.2}, and
the right-to-left-to-right roundtrip is the identity by the last phrase in that theorem.
The left-to-right map is injective because of the splitting principle and the multiplicativity formula (5)
for the symplectic Thom classes (cf.\ \cite[Lemma 3.34]{Panin:2003rz}).

(d) Left to the reader.
\end{proof}

\section{More on the cohomology of the open stratum}
\label{S:original}

This is a version of our original proof of Theorem \ref{A.X2i}(b).  We use the following geometry,
with $(V,\phi)$ a symplectic
space of dimension $2n+2$ over a field and $\GrSp(2,V,\phi) \subset \Gr(2,V)$
the closed subvariety parametrizing totally isotropic subspaces.

\begin{theorem}
\label{2nd}
The open stratum $X_{0} \subset \HP^{n}$ is a dense open subvariety of an open
subvariety $Y_{0} \subset \Gr(2,V)$ with complement $Z_{0} =
Y_{0} \smallsetminus X_{0} = Y_{0} \cap \GrSp(2,V,\phi)$, and there is an
open subvariety $Y'_{0} \subset Y_{0}$ containing $Z_{0}$
for
which there exists a diagram in which the arrows have the stated
properties and the horizontal arrows commute with each other and with
all upward  arrows and with the two solid downward arrows.
\begin{equation}
\label{A2n.bdl}
\vcenter{
\xymatrix @M=5pt @C=60pt @R=30pt {
  Z_{0} \ar[r]_-{\text{\rm closed}}
  \ar@<4pt>[d]^-{\text{\scriptsize\txt{\rm restriction of the \\ \rm $\Aff^{2n}$-bundle}}}
  \ar@<4pt>@/^1pc/[rr]^-{\text{\rm closed}} &
  Y'_{0} \ar@{^{(}->}@<-2pt>[r]_-{\text{\rm open}} \ar@<4pt>[d]^-{\text{\rm
      $\Aff^{2n}$-bundle}} &
  Y_{0} \ar@{.>}@<2pt>[d]^-{\text{\scriptsize\txt{\rm different \\ \rm
      $\Aff^{2n}$-bundle}}} &
  X_{0} = Y_{0} \smallsetminus Z_{0}
  \ar@{_{(}->}@<2pt>[l]_-{\text{\rm open}}
  \\
  \PP^{2n-1} \ar[r]^-{\text{\rm closed}}
  \ar@<4pt>[u]
  \ar@<-4pt>@/_1pc/[rr]_-{\text{\rm hyperplane}} &
  \PP^{2n} \smallsetminus 0
  \ar@{^{(}->}@<-2pt>[r]^-{\text{\rm open}}
  \ar@<4pt>[u]
  &
  \PP^{2n} \ar@<4pt>[u]
  &
  \Aff^{2n} \ar@{_{(}->}@<2pt>[l]_-{\text{\rm open}}
  \ar@<4pt>[u]
}}
\end{equation}
\end{theorem}

\begin{proof}
In $(V,\phi)$ we have
a $1$-dimensional subspace $E = \langle e \rangle$, a vector $f$ with $\phi(e,f) = 1$, and
the $\phi$-nondegenerate subspace $F = \langle e,f \rangle^{\perp}$.  We have $E^{\perp} = E \oplus F$.
We set
\begin{align*}
X_{0} & = \{ U \subset V \mid \text{$U \not\subset E^{\perp}$ and $\phi\rest{U}$ is nondegenerate}\},
\\
Y_{0}  & = \{ U \subset V \mid U
  \not\subset E^\perp \}, \\
Z_{0} & = Y_{0} \smallsetminus X_{0} = Y_{0} \cap \GrSp(2,V,\phi),
\\
Y'_{0}  & = \{ U \subset V \mid \text{$U
  \not\subset E^\perp$ and $U \cap E = \{0\}$} \}.
\end{align*}
Clearly $Y_{0}$ is an open subvariety of the Grassmannian, and
$X_{0}$ and $Y'_{0}$ are open in $Y_{0}$.   Moreover, any
$U$ in $Y_{0} \smallsetminus Y'_{0}$ is of the form $U = \langle u,v
\rangle$ with $0 \neq u \in E$ and $v \not\in E^\perp$.  We therefore
 have (up to nonzero multiples) $u
= e$ and $v = f + v_0$ with $v_0
\in E^\perp$.  We then have $\phi(u,v) =1$, so $\phi\rest{U}$ is
nondegenerate.
So we have $Y_{0} \smallsetminus Y'_{0} \subset X_{0}$ and $Z_{0} \subset Y'_{0}$.

For the vector bundle structures, we use Bia\l{}ynicki-Birula's theory
of cell decompositions induced by $\GG_m$-actions.
Consider the
one-parameter subgroup $\lambda_1 \colon \GG_m \into \GL(V)$
given by
%
\begin{align*}
\lambda_1(t) \cdot v & = \begin{cases}
v & \text{for $v \in E^{\perp} = \langle e \rangle \oplus F$}, \\
t^{-1}v & \text{for $v \in \langle f \rangle$}.
\end{cases}
\end{align*}
The fixed-point locus of the induced action on $\Gr(2,V)$ has
two components
\begin{align*}
\Gamma_{0} &  = \Gr(2,E^{\perp}), &
\Gamma_{1} & = \{ \langle v,f \rangle \mid v \in E^{\perp} \} \cong \PP(E^{\perp}) = \PP^{2n}
\end{align*}
For each component there
is a stratum
$\Upsilon_{i} = \{ x \in \Gr(2,V) \mid \lim_{t\to 0} \lambda_1(t)
x \in \Gamma_{i} \}$.
The assignment $x \mapsto \lim_{t\to 0} \lambda_1(t)x$ gives maps
$\pi_1 \colon \Upsilon_{i} \to \Gamma_{i}$ which send any $U \in
\Gr(2,V)$ to its associated graded space with respect to the
filtration on $U$ induced by
$0 \subset E^{\perp} \subset V$.  The maps $\pi_1
\colon \Upsilon_{i} \to \Gamma_{i}$ are vector bundles identifiable with
the normal bundles $N_{\Gamma_{i} / \Upsilon_{i}}$.  In this case we have
\begin{align*}
\Upsilon_{0} &  = \Gr(2,E^{\perp}) = \Gamma_{0}, &
\Upsilon_{1} & = Y_{0}.
\end{align*}
The fiber of $\pi_{1} \colon Y_{0} \to \PP^{2n}$ over $\langle v,f \rangle \in \PP^{2n}$
consists of points corresponding to spaces $\langle v, u + f \rangle$ with $u \in E^{\perp}$.  These
points are parametrized by the classes $\overline u \in E^{\perp}/\langle v \rangle$.
So $Y_{0}$ is isomorphic to the
tautological rank $2n$ quotient bundle  $\shf Q$
on $\PP^{2n}$.

Now $\lambda_1$ does not preserve the symplectic form $\phi$ because
it does not act on $e$ and $f$ with opposite weights.  So
to study the part of $Y_{0}$ in $\GrSp(2,V,\phi)$ we replace it by
$\lambda_2 \colon \GG_m \into \Sp(2,V,\phi)$ acting linearly on $V$ with
\begin{align*}
\lambda_2(t) \cdot v & =
\begin{cases}
tv & \text{for $v \in \langle e \rangle$}, \\
v & \text{for $v \in F$}, \\
t^{-1}v & \text{for $v \in \langle f \rangle$}.
\end{cases}
\end{align*}
The induced action of $\lambda_2$ on $\Gr(2,V)$ comes with four fixed-point
loci
\begin{align*}
\Gamma'_{0} & = \{ \langle e,v \rangle \mid v \in F \}, &
\Gamma'_{1} &  = \Gr(2,F), &
\Gamma'_{2} & = \{ \langle e,f \rangle \}, &
\Gamma'_{3} & = \{ \langle u,f \rangle \mid u \in F \}  = \PP^{2n-1}
\end{align*}
and vector bundles $\pi_2 \colon \Upsilon'_{i} \to
\Gamma'_{i}$ much as before.  This time, however, we have $Y_{0} =
\Upsilon'_{2} \cup \Upsilon'_{3}$ and $Y'_{0} =
\Upsilon'_{3}$.  We have $Y_{0} \smallsetminus
Y'_{0} = \Upsilon'_{2} = \{ \langle e,v+f \rangle \mid v \in F \} \cong \Aff^{2n}$, and we may see that it intersects
$\Gamma_{1} = \PP^{2n}$ in the unique point $\langle e,f \rangle =
\Gamma'_{2}$.  We will call this point $0 \in \PP^{2n}$.

Now the Bia\l{}ynicki-Birula vector bundle $Y'_{0} = \Upsilon'_{3}
\to \Gamma'_{3} = \PP^{2n-1}$ is
of rank $2n+1$.  The base, the fixed point set $\PP^{2n-1} = \{
\langle u, f \rangle \mid u \in \PP(F) \}$, is entirely within
$\GrSp(2,V,\phi)$.  Since the one-parameter subgroup respects
$\GrSp(2,V,\phi)$, the locus $Z_{0}$ is also a bundle over
$\PP^{2n-1}$, indeed a subbundle of $Y'_{0}$ of rank $2n$.

As a bundle $Y'_{0}$ is isomorphic to $N_{\PP^{2n-1}/Y'_{0}}$. At
any point $U_{\overline u} = \langle u , f \rangle \in
\PP^{2n-1}$, the fiber of the normal bundle is canonically
isomorphic to the quotient of the subspace of $\phi \in \Hom_k( U_{\overline
  u}, V/U_{\overline u}) = T_{U_{\overline u}} \Gr(2,V)$ corresponding
to first order deformations $\langle u + \phi(u) \varepsilon, f
+ \phi(f) \varepsilon \rangle$ which lie within $Y'_{0}$
modulo those lying within $\PP^{2n-1}$.  These first-order
deformations are thus the $\langle u + \alpha e
\varepsilon, f + (v + \beta e)\varepsilon \rangle$ with  $\overline {v} \in F/\langle u \rangle$.  Thus $Y'_{0}$ is isomorphic to a vector bundle $N \to
\PP^{2n-1}$ which we can split as $N = N_0 \oplus L$ with
\begin{align*}
N_0 & \xrightarrow[\cong]{(\overline{v},\beta)} \shf Q \oplus \OO, &
L & \xrightarrow[\cong]{\alpha} \OO(1),
\end{align*}
The subbundle of $N$ corresponding to $Z_{0}$ corresponds to the
first-order deformations satisfying
\[
\phi(u + \alpha e\varepsilon , f + (v
+ \beta e) \varepsilon) = (\alpha + \phi(u,v)) \varepsilon = 0.
\]
So it is defined by the equation $\alpha = - \phi(u,v)$.  This bundle is thus the
graph of a vector bundle map $\psi \colon N_0 \to L$ which fiberwise is
$(\overline v, \beta) \mapsto -\phi(u,v)$ (with $u$ a function of
the base).  Let $N_1 \subset N$ be the graph of $\psi$.  We still
have a direct sum $N = N_1 \oplus L$.  We now have a quotient map $N
\onto N/N_1 \cong L$ which makes $Y'_{0} \cong N$ into a rank $2n$
vector bundle over $L$, with $Z_{0} \cong N_1$ the subbundle lying over the
zero section $\PP^{2n-1}$ of $L$.  Finally the subbundle $L$
consists of the points $\langle \alpha e + u, f \rangle$
with $\overline u \in \PP^{2n-1}$, which is $\PP^{2n}
\smallsetminus 0$.
\end{proof}

\begin{proof}[Second proof of Theorem \ref{A.X2i}\parens{b}]
  By Theorem \ref{th.subgrass} it is enough to treat the open stratum $X_{0}$.
  By Theorem \ref{2nd} $X_{0} \subset Y_{0}$ is an open subvariety
  with complement $Z_{0}$, and there is a closed embedding of pairs
  $(\PP^{2n}, \Aff^{2n}) \to (Y_{0},X_{0})$.  This yields a
  morphism of localization long exact sequences
\begin{equation}
\label{long.exact}
\vcenter{
  \xymatrix @M=5pt @C=14pt {
    A_{Z_{0}}(Y_{0}) \ar[r]
    \ar[d]
    &
    A(Y_{0}) \ar[r]
    \ar[d]_-{\cong}
    &
    A(X_{0}) \ar[r]^-{\partial} \ar[d]_-{i^A}
    &
    A_{Z_{0}}(Y_{0}) \ar[d] \ar[r]
    &
    A(Y_{0}) \ar[d]_-{\cong}
    \\
    A_{\PP^{2n-1}}(\PP^{2n}) \ar[r]
    &
    A(\PP^{2n}) \ar[r]
    &
    A(\Aff^{2n}) \ar[r]^-{\partial}
    &
    A_{\PP^{2n-1}}(\PP^{2n}) \ar[r]
    &
    A(\PP^{2n})
  }
}
\end{equation}
The map $A(Y_{0}) \to A(\PP^{2n})$ is an isomorphism because
$Y_{0} \to \PP^{2n}$ is a vector bundle.
By excision $A_{Z_{0}}(Y_{0}) \to A_{\PP^{2n-1}}(\PP^{2n})$ is
isomorphic to $A_{Z_{0}}(Y'_{0}) \to A_{\PP^{2n-1}}(\PP^{2n}
\smallsetminus 0)$, which again is an isomorphism because of the
vector bundles in Theorem \ref{2nd} and strong homotopy invariance.
The map labeled $i^A$ is therefore also an isomorphism by the five lemma.
Since $i^A t^A \colon A(k) \to A(\Aff^{2n})$ is an isomorphism by
homotopy invariance, it follows that $t^A$ is an isomorphism.
\end{proof}

%

\begin{thebibliography}{10}

\bibitem{Atiyah:1971zr}
{\sc M.~F. Atiyah}, {\em Riemann surfaces and spin structures}, Ann. Sci.
  \'Ecole Norm. Sup. (4), 4 (1971), pp.~47--62.

\bibitem{Atiyah:1976ly}
{\sc M.~F. Atiyah and E.~Rees}, {\em Vector bundles on projective 3-space},
  Invent. Math., 35 (1976), pp.~131--153.

\bibitem{Balmer:1999if}
{\sc P.~Balmer}, {\em Derived {W}itt groups of a scheme}, J. Pure Appl.
  Algebra, 141 (1999), pp.~101--129.

\bibitem{Balmer:2008fk}
{\sc P.~Balmer and B.~Calm{\`e}s}, {\em Witt groups of {G}rassmann varieties},
  (2008).
\newblock Preprint, 2008.

\bibitem{Calmes:2008pi}
{\sc B.~Calm{\`e}s and J.~Hornbostel}, {\em Push-forwards for {W}itt groups of
  schemes},  (2008).
\newblock arXiv:0806.0571.

\bibitem{Conner:1966uk}
{\sc P.~E. Conner and E.~E. Floyd}, {\em The relation of cobordism to
  {$K$}-theories}, Lecture Notes in Mathematics, No. 28, Springer-Verlag,
  Berlin, 1966.

\bibitem{Gille:2007hb}
{\sc S.~Gille}, {\em The general d\'evissage theorem for {W}itt groups of
  schemes}, Arch. Math. (Basel), 88 (2007), pp.~333--343.

\bibitem{Knus:1991qr}
{\sc M.-A. Knus}, {\em Quadratic and {H}ermitian forms over rings}, vol.~294 of
  Grundlehren der Mathematischen Wissenschaften, Springer-Verlag, Berlin, 1991.

\bibitem{Levine:2007ys}
{\sc M.~Levine and F.~Morel}, {\em Algebraic cobordism}, Springer Monographs in
  Mathematics, Springer, Berlin, 2007.

\bibitem{Macdonald:1995zl}
{\sc I.~G. Macdonald}, {\em Symmetric functions and {H}all polynomials}, Oxford
  University Press, Oxford, 2nd~ed., 1995.

\bibitem{Milnor:1974if}
{\sc J.~W. Milnor and J.~D. Stasheff}, {\em Characteristic classes}, Princeton
  University Press, Princeton, N. J., 1974.
\newblock Annals of Mathematics Studies, No. 76.

\bibitem{Morel:2006ul}
{\sc F.~Morel}, {\em Rational stable splitting of {G}rassmanians and rational
  motivic sphere spectrum}.
\newblock First draft, 2006.

\bibitem{Nenashev:2007rm}
{\sc A.~Nenashev}, {\em Gysin maps in {B}almer-{W}itt theory}, J. Pure Appl.
  Algebra, 211 (2007), pp.~203--221.

\bibitem{Panin:2003rz}
{\sc I.~Panin}, {\em Oriented cohomology theories of algebraic varieties},
  $K$-Theory, 30 (2003), pp.~265--314.
\newblock Special issue in honor of Hyman Bass on his seventieth birthday. Part
  III.

\bibitem{Voevodsky:1998kx}
{\sc V.~Voevodsky}, {\em {$\mathbf{A}\sp 1$}-homotopy theory}, Doc. Math.,
  Extra Vol. I (1998), pp.~579--604 (electronic).

\end{thebibliography}

\end{document}